\documentclass[leqno, letterpaper, 11pt]{article}
\usepackage{amssymb}
\usepackage{amsmath}
\usepackage{amsfonts}
\usepackage{amsthm}
\usepackage[usenames]{color}
\usepackage{graphicx}
\usepackage{verbatim}
\usepackage[reftex]{theoremref}
\usepackage{bbm}
\usepackage{hyperref}
\usepackage{caption}
\usepackage{subcaption}

\usepackage{pgf,tikz}
\usetikzlibrary{arrows}

\graphicspath{{figures/}}

\newcommand{\note}[1]{{\bf \textcolor{blue}
{[#1\marginpar{\textcolor{red}{***}}]}}}

\newcommand{\bR}{{\mathbb R}}
\newcommand{\bZ}{{\mathbb Z}}

\newcommand{\cA}{{\mathcal A}}

\newcommand{\cZ}{{\mathcal Z}}

\newcommand{\cO}{\mathcal O}
\newcommand{\cP}{\mathcal P}
\newcommand{\cT}{\mathcal T}

\newcommand{\cS}{\mathcal S}

\newcounter{mycount}

\newtheorem{theorem}{Theorem}[section]
\newtheorem{lemma}[theorem]{Lemma}
\newtheorem{prop}[theorem]{Proposition}
\newtheorem{cor}[theorem]{Corollary}

\theoremstyle{definition}
\newtheorem{ex}[theorem]{Example}
\newtheorem{remark}[theorem]{Remark}

\newcounter{example}

\newcounter{open}

\newcounter{figno}

\newcounter{tableno}

\numberwithin{equation}{section}
\numberwithin{figure}{section}
\numberwithin{table}{section}
\numberwithin{example}{section}

\newcommand{\prob}[1]{\mathbb{P}\left(#1\right)}
\newcommand{\probsub}[2]{\mathbb{P}_{#2}\left(#1\right)}
\newcommand{\Esub}[2]{\mathbb{E}_{#2}\left(#1\right)}
\newcommand{\abs}[1]{\left|#1\right|}
\newcommand{\ceil}[1]{\left\lceil#1\right\rceil}
\newcommand{\floor}[1]{\left\lfloor#1\right\rfloor}
\newcommand{\supp}{\mathop{\mathrm{supp}}}

\newcommand{\Z}{\cZ}

\newcommand{\ZE}{\widetilde \cZ}
\newcommand{\gammaE}{\widetilde\gamma}
\newcommand{\TE}{\widetilde \cT}
\newcommand{\RE}{\widetilde R}
\newcommand{\IE}{\widetilde I}

\newcommand{\length}{\text{\tt length}}
\newcommand{\area}{\text{\tt area}}
\newcommand{\squrep}{\text{\tt square}}

\newcommand{\origin}{\mathbf 0}
\newcommand{\X}{\mathbf \times}
\newcommand{\tmax}{{T_{{\mathrm {max}}}}}

\newcommand{\Econv}{\overset{\mathrm E}{\longrightarrow}}
\newcommand{\Span}{\text{\tt Span}}

\newcommand{\zswk}{\Z^{\swarrow k}}
\newcommand{\Il}{\widetilde{I}_{\text{lower}}}
\newcommand{\Iu}{\widetilde{I}_{\text{upper}}} 
\newcommand{\row}{\text{\tt row}}
\newcommand{\col}{\text{\tt col}}
\newcommand{\hE}{\widetilde{\row}}
\newcommand{\vE}{\widetilde{\col}}
\newcommand{\myll}{\text{\scalebox{2}{$\llcorner$}}}

\newcommand{\angk}{A_{>k}}

\begin{document}

\begin{center}\Large
{\bf Neighborhood growth dynamics on the Hamming plane}\footnote{Version 1, \today}
\end{center}

\begin{center}

{\sc Janko Gravner}\\
{\rm Mathematics Department}\\
{\rm University of California}\\
{\rm Davis, CA 95616, USA}\\
{\rm \tt gravner{@}math.ucdavis.edu}
\end{center}
\begin{center}
{\sc David Sivakoff}\\
{\rm Departments of Statistics and Mathematics}\\
{\rm The Ohio State University}\\
{\rm Columbus, OH 43210, USA}\\
{\rm \tt dsivakoff{@}stat.osu.edu}
\end{center} 
\begin{center}
{\sc Erik Slivken}\\
{\rm Mathematics Department}\\
{\rm University of California}\\
{\rm Davis, CA 95616, USA}\\
{\rm \tt erikslivken{@}math.ucdavis.edu}
\end{center} 

\begin{abstract} 
We initiate the study of general neighborhood growth dynamics on 
two dimensional Hamming graphs. 
The decision to add a point is made by counting the currently 
occupied points on the horizontal and the vertical line through it, and checking whether 
the pair of counts lies outside a fixed Young diagram.  
We focus on two related extremal quantities. The first is the size of the
smallest set that eventually occupies the entire plane. The second 
is the minimum of an energy-entropy functional that comes from the scaling of the probability of
eventual full occupation versus the density of the 
initial product measure within a rectangle. We demonstrate the existence of this scaling 
and study these quantities for large Young diagrams. 
\end{abstract}

\let\thefootnote\relax
\let\thefootnote\relax

\section{Introduction}\label{sec-intro}

We consider a long-range deterministic 
growth process on the discrete plane,
restricted for convenience to the first quadrant $\bZ_+^2$. 
This dynamics iteratively enlarges a subset of $\bZ_+^2$ 
by adding points based on counts on the entire 
horizontal and vertical lines through them. The connectivity is therefore that of a 
two-dimensional Hamming graph, 
that is, a Cartesian product of two complete graphs. The papers \cite{Siv, GHPS, Sli, BBLN} 
address some percolation and growth processes on vertices of 
Hamming graphs, but such highly nonlocal growth models remain
largely unexplored. In particular, the few two-dimensional problems addressed so far
appear to be too limited to offer much insight, and we seek to 
remedy this with a class of models we now introduce.

\newcommand{\cI}{\mathcal I}
\newcommand{\bN}{\mathbb N}

For integers $a,b\in \bN^2$, we let $R_{a,b}=([0,a-1]\times [0,b-1])\cap \bZ_+^2$ be 
the discrete $a\times b$ rectangle. A set $\Z=\cup_{(a,b)\in \cI} R_{a,b}$, given by a union
of rectangles 
over some set $\cI\subseteq\bN^2$, is called a (discrete) {\it zero-set\/}. We allow the trivial case
$\Z=\emptyset$, and also the possibility that $\Z$ is infinite. However, 
in most of the paper the zero-sets will be 
finite and therefore equivalent to Young diagrams in the French notation \cite{Rom} (see Figure~\ref{growth-example-Young}).
Our dynamics will be given by iteration of a growth transformation $\cT:2^{\bZ_+^2}\to 2^{\bZ_+^2}$,
and will be determined by the associated zero-set $\Z$, so we 
will commonly not distinguish between the two. 

Fix a zero-set $\Z$. Suppose $A\subseteq \bZ_+^2$ and $x\in \bZ_+^2$. 
Let $L^h(x)$ and $L^v(x)$ be the horizontal and the vertical line through $x$, so that 
the {\it neighborhood\/} of $x$ is $L^h(x)\cup L^v(x)$.
If $x\in A$, then 
$x\in \cT(A)$. If $x\notin A$, we compute the horizontal and vertical counts
$$
\row(x,A)=|L^h(x)\cap A|\ \text{ and }\ \col(x,A)=|L^v(x)\cap A|, 
$$
form the pair $(u,v)=(\row(x,A), \col(x,A))$,
and declare $x\in \cT(A)$ if and only if $(u,v)\notin\Z$. Observe 
that, by definition of a zero set, {\it monotonicity\/} holds: $A\subseteq A'$ implies 
$\cT(A)\subseteq \cT(A')$. We call such a rule a {\it neighborhood growth\/} rule. 
So defined, this class in fact comprises all rules that satisfy the natural monotonicity and 
symmetry assumptions and have only nearest-neighbor dependence under the Hamming
connectivity; see Section~\ref{sec-pattern}.  

\begin{figure}[htb]
\centering
\begin{subfigure}[c]{0.3\textwidth}
\centering
\includegraphics[width=\textwidth]{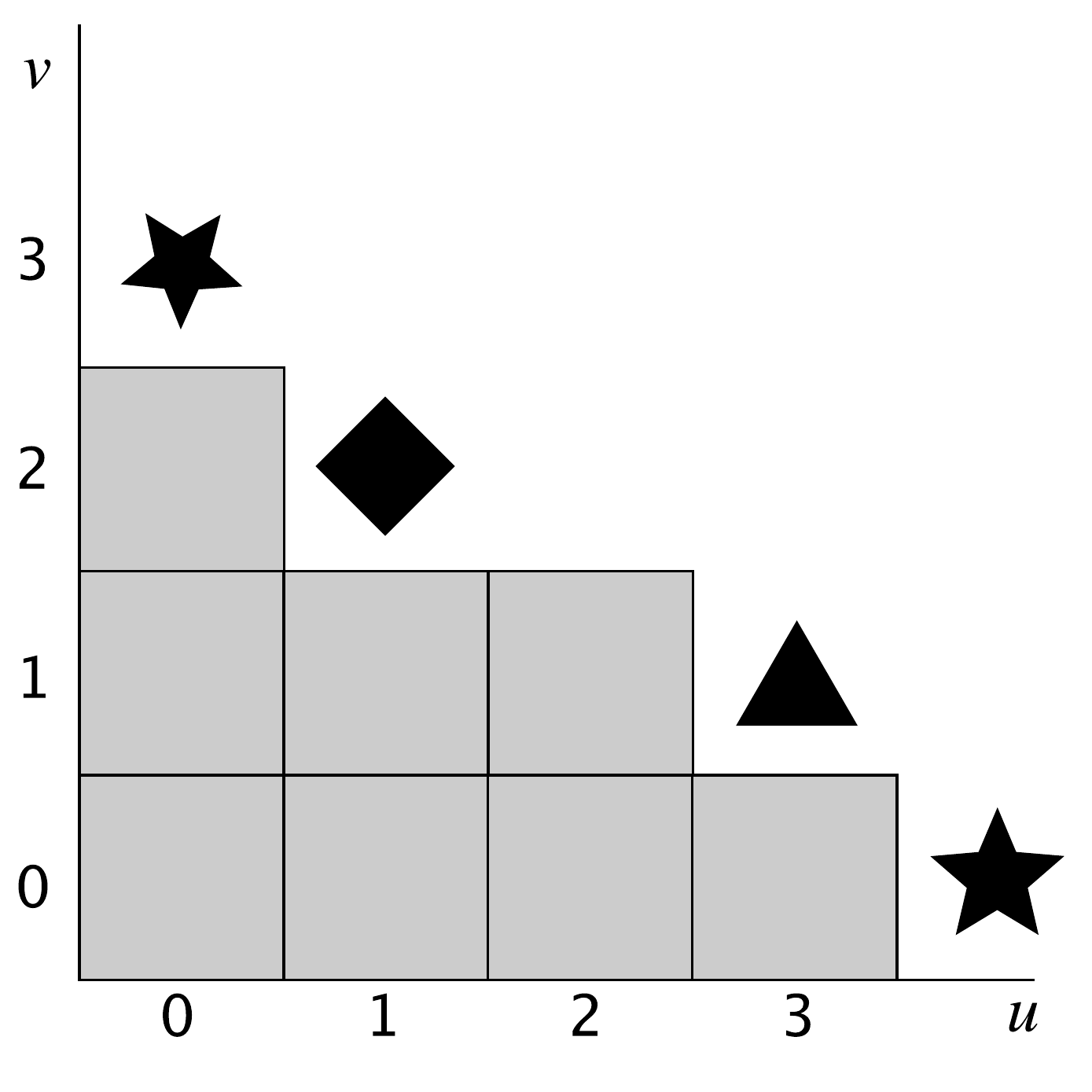}
\caption{\label{growth-example-Young} A zero-set $\cZ$ (grey region). Shapes on external boundary correspond to distinct minimal neighborhood counts that will result in occupation of vertices.  E.g., the diamond signifies occupation by having at least one horizontal and at least two vertical neighbors.}
\end{subfigure}
\quad
\begin{subfigure}[c]{0.6\textwidth}
\centering
\includegraphics[width=\textwidth]{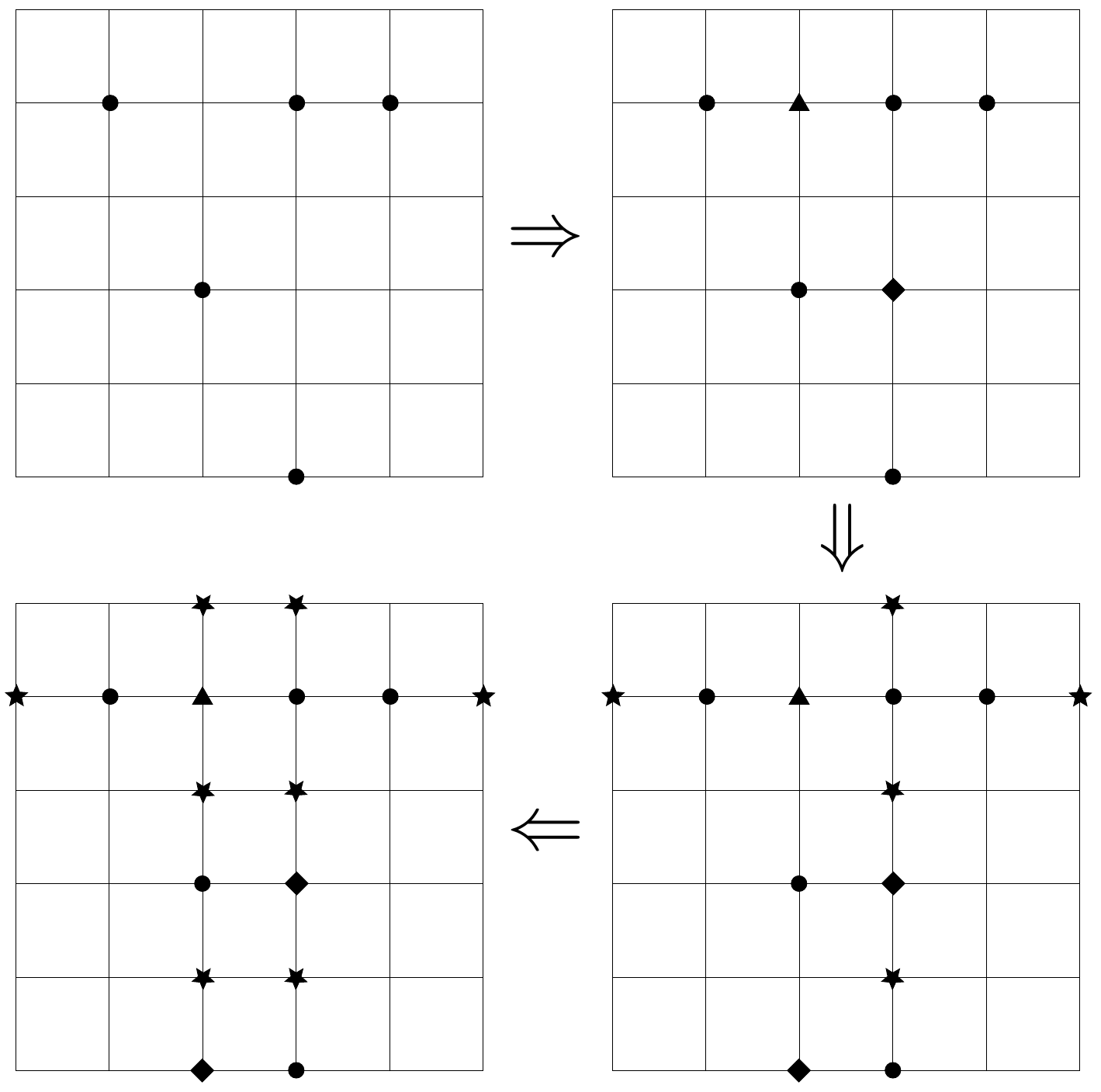}
\caption{\label{growth-example-A} Example of growth from an initial occupied set $A$ (top left, circles).  Different shapes correspond to the row and column counts at the time of occupation, as indicated in~\ref{growth-example-Young}. The last configuration (bottom left) is inert, that is, $\cT^4(A) = \cT^3(A)$.}
\end{subfigure}
\caption{\label{growth-example} An example of neighborhood growth.}
\end{figure}

A given initial set $A\subseteq \bZ_+^2$ and $\cT$ 
then specify the discrete-time trajectory: 
$A_{t}=\cT^t(A)$ for $t\ge 0$. 
The points in $A_t$ and $A_t^c$ are respectively called 
{\it occupied\/} and {\it empty\/} at time $t$. 
We define $A_\infty=\cT^\infty(A)=\cup_{t\ge 0} A_t$ 
to be the set of eventually occupied points.
We say that the set $A$ {\it spans\/} if $A_\infty=\bZ_+^2$.  
We also say that a set $B\subseteq \bZ_+^2$ 
is {\it spanned\/} if $B\subseteq \cT^\infty(A)$ and that $B$ is 
{\it internally spanned\/} by $A$ if the 
dynamics restricted to $B$ spans it: $B=\cT^\infty(A\cap B)$.  See Figure~\ref{growth-example-A} for an example of these dynamics.

The central theme of this paper is minimization of certain functionals on the 
set $\cA$ of all finite spanning sets. Perhaps the simplest such functional is the 
cardinality, which results in the quantity
$$\gamma(\cT)=\gamma(\Z)=\min\{|A|:A\in \cA\}.
$$
Our second functional is related but requires further explanation and notation, and we will introduce it below
when we state our main results.  We first put the topic in the context of previous work. 

\newcommand{\Ttheta}{T_\theta}

The best known special case of neighborhood growth 
is given by an integer threshold $\theta\ge 1$, with the rule 
that $x$ joins the occupied set whenever the entire neighborhood count
is at least $\theta$. This rule makes sense on any graph; in our case it translates to 
triangular $\Z=\Ttheta=\{(u,v):u+v\le \theta-1\}$.  
Such dynamics are known by the name of {\it threshold growth\/} \cite{GG1} or 
{\it bootstrap percolation\/} \cite{CLR}. Bootstrap percolation on graphs with short range 
connectivity has a long and distinguished history as a model for 
metastability and nucleation. The most common setting is a graph of the form 
$[k]^\ell$, a Cartesian product of $\ell$ path graphs of $k$ points, and thus with 
standard nearest neighbor lattice connectivity.  The foundational mathematical paper is
\cite{AL}, which studied what we call the {\it classic\/} bootstrap percolation, which is
the process with $\theta=2$ on $[n]^2$. A brief summary of this paper's ongoing legacy is impossible, so we mention only a few notable successors: \cite{Hol} gives 
the precise asymptotics for the classical bootstrap percolation; 
\cite{BBDM} extends the result for all $[n]^d$ and $\theta$; 
the hypercube $[2]^n$ with $\theta=2$ is analyzed in \cite{BB,BBDM}; 
and a recent paper \cite{BDMS} addresses a bootstrap percolation model with drift. 
The main focus of the voluminous research is estimation 
of the critical probability on large finite sets, that is, 
the initial occupation density $p_c$
that makes spanning occur with probability $1/2$. It is typical for this class 
of models that $p_c$ approaches zero very slowly with increasing system size, 
certainly slower than any power, 
and that the transition in the probability 
of spanning from small to close to $1$ near $p_c$ is very sharp. For example, 
$p_c\sim \pi^2/(18\log n)$ for the classic bootstrap percolation \cite{Hol}. 
Neither slow decay nor sharp transition happen for supercritical
threshold growth on the two-dimensional lattice \cite{GG1} or threshold growth on Hamming graphs 
\cite{GHPS, Sli}, where instead power laws hold.  
One of our main results, Theorem~\ref{intro-ld-rate}, 
shows that, for any neighborhood growth, 
there is a well-defined 
power-law relationship between the density of the initial set, the size of the 
system, and the probability of spanning.  

Another special case is the {\it line growth\/}, where $\Z=R_{a,b}$ for some 
$a,b\in \bN$. This was introduced under the 
name {\it line percolation\/} 
in the recent paper \cite{BBLN}, which proves that $\gamma(R_{a,b})=ab$,
establishes a similar result in higher dimensions, and obtains
the large deviation rate (defined below) for $\Z=R_{a,a}$ on a square. Some 
of our results are therefore extensions of those in \cite{BBLN}. In particular, 
one may ask for which $\Z$ the equality $\gamma(\Z)=\gamma(R_{a,b})$ holds for some
$R_{a,b}\subseteq \Z$. We discuss this in Section~\ref{sec-lgbound}. 

Extremal problems play a prominent role in growth models: they feature in 
the estimation of the nucleation probability, but they are also interesting in 
their own right. For bootstrap percolation, the size of the smallest spanning subset for
$[n]^d$ when $\theta=2$ is known to be $\lfloor d(n-1)/2\rfloor +1$  for all $n$ and $d$ 
\cite{BBM}; the clever argument 
that the smallest spanning set for classic bootstrap percolation on $[n]^2$ 
has size $n$ is a folk classic. 
The situation is much murkier 
for larger $\theta$; see \cite{BPe, BBM} for a review of known results
and conjectures for low-dimensional lattices $[n]^d$ and hypercubes $[2]^n$. 
The smallest spanning sets have also been studied for bootstrap percolation on trees
\cite{Rie2} and certain hypergraphs \cite{BBMR}. However, the 
closest parallel to the analysis of $\gamma$ in the present paper is the 
large neighborhood setting for the threshold growth model on $\bZ^2$ from \cite{GG2}. 
Several related extremal questions, which are not considered in this paper, are also of interest. 
For example, one may ask for the {\it largest\/} size of the  inclusion-minimal set that 
spans (\cite{Mor} addresses this for the classic bootstrap percolation,  
\cite{Rie1} for hypercubes with $\theta=2$, 
and \cite{Rie2} for trees), or for the {\it longest time\/} that a spanning set may take 
to span (this is the subject of a recent paper \cite{BPr} on the classic bootstrap percolation).

We now proceed to our main results, beginning 
with a theorem that gives  basic information on the size of $\gamma$.
The upper bound we give
cannot be improved, as it is achieved by the line growth.  
We do not know whether the $1/4$ in the lower 
bound can be replaced by a larger number.

\begin{theorem} \label{intro-gamma-area-thm}
For all zero sets $\Z$, 
$$
\frac14 |\Z|\le \gamma(\Z)\le |\Z|.
$$
\end{theorem} 

Assume that the initially occupied set is restricted to a rectangle $R_{N,M}$, which 
is large enough to include the entire $\Z$ (which is then, of course, finite). 
Then, as it is easy to see, the 
dynamics spans $\bZ_+^2$ if and only if it internally spans $R_{N,M}$. As all our rectangles 
will satisfy this assumption, we will not distinguish between spanning and their internal spanning.
Now, one may ask if a configuration restricted to the interior of such a rectangle 
requires more sites to span than an unrestricted configuration. Our next result 
answers this question in the negative, establishing a property of obvious importance for 
a computer search for smallest spanning sets. 

\begin{theorem}\label{intro-packing-minimal}
Assume that $a_0,b_0\in \bN$ are such that $\Z\subseteq R_{a_0,b_0}$. Then
$$\gamma(\Z)=\min\{|A|: A\in \cA \text{ and }A\subseteq R_{a_0,b_0}\}.$$
\end{theorem}

Next we consider spanning by random subsets of rectangles $R_{N,M}$. Assume 
that the initial configuration is restricted to $R_{N,M}$, where it is chosen according 
to a product measure with a small density $p>0$. The possibly unequal 
sizes $N$ and $M$ need to increase as 
$p\to 0$, and, given that in all known cases spanning probabilities on Hamming graphs obey 
power laws \cite{GHPS, BBLN}, it is natural to suppose that they scale as powers of $p$. 
Thus we fix $\alpha,\beta\ge 0$ and assume that, as $p\to 0$, $N,M\to\infty$ and 
$$\log N\sim -\alpha\log p,\quad \log M\sim -\beta\log p.$$
We will denote
by $\Span$ the event that the so defined initial set spans, and turn our attention to the
question of the resulting power-law scaling for $\probsub{\Span}{p}$. 
The answer will involve finding the optimal 
energy-entropy balance, so there is a conceptual connection with 
large deviation theory, despite the fact that the probabilities involved are not 
exponential.  Thus we call the quantity 
$$
I(\alpha,\beta)=I(\alpha,\beta,\Z)=\lim_{p\to 0}\frac{\log \probsub{\Span}{p}}{\log p}
$$
the {\it large deviation rate\/} for the event $\Span$, provided it exists.
 
The rate $I$ is given as the minimum, over the spanning sets, of the functional $\rho$ 
that we now define.
For a finite set $A\subseteq \bZ_+^2$, let $\pi_x(A)$ and $\pi_y(A)$ be projections of $A$ on 
the $x$-axis and $y$-axis, respectively. Then let
$$
\rho(\alpha,\beta, A)=\max_{B\subseteq A}\,\left(|B|-\alpha |\pi_x(B)|-\beta|\pi_y(B)|\right). 
$$
The term $|B|$ represents the energy of the 
subset $B$ and the linear combination of sizes of the 
two projections the entropy of $B$. In the next theorem, we 
use the following notation for the {\it outside boundary\/} of a Young diagram 
$Y$:
$$
\partial_o Y=\{(u,v)\in \bZ_+^2\setminus Y: (u-1,v)\in Y\text{ or }(u,v-1)\in Y\}. 
$$
 Also, we use the notation $a\vee b = \max(a,b)$ and $a\wedge b = \min(a,b)$ for real numbers $a,b$.
\begin{theorem}\label{intro-ld-rate}
For any finite zero-set $\Z$, the large deviation rate $I(\alpha,\beta, \Z)$ 
exists.
Moreover, there exists a finite set $\cA_0\subseteq \cA$, independent of $\alpha$ and $\beta$, 
so that 
\begin{equation}\label{ld-vq}
I(\alpha,\beta,\Z)=\inf\{\rho(\alpha,\beta,A): 
A\in \cA\}=\min\{\rho(\alpha,\beta,A): 
A\in \cA_0\}.
\end{equation}
The rate $I(\alpha,\beta,\Z)$ as a function of $(\alpha,\beta)$ is continuous, 
piecewise linear, nonincreasing in both arguments, concave when $\alpha+\beta\le 1$, and 
$I(0,0,\Z)=\gamma(\Z)>I(\alpha,\beta,\Z)$ unless $\alpha=\beta=0$.

Moreover, the support of $I$ is given by 
\begin{equation}\label{ld-support-formula}
\supp I(\cdot,\cdot,\cZ) = \bigcap_{(u,v)\in\partial_o\cZ} \left\{(\alpha,\beta)\in [0,1]^2 : [u(1-\alpha) -\beta] \vee [v(1-\beta) - \alpha] \ge 0 \right\}.
\end{equation}
Furthermore, if $\alpha,\beta \in [0,1]^2 \setminus {\supp I(\cdot,\cdot,\cZ)}$, then $\probsub{\Span}{p} \to 1$. 
\end{theorem}
 
We give explicit formulae for $I(\alpha,\beta, R_{a,b})$ 
and $I(\alpha,\alpha,\Ttheta)$ in Sections~\ref{sec-lgld} and~\ref{sec-bpld}. 
In general, determining an explicit analytical formula for 
this rate even for a moderately large $\Z$ appears to be quite challenging. Figure~\ref{BP support fig} depicts the support of $I(\cdot,\cdot,\Ttheta)$ for several values of $\theta$, and Figure~\ref{LG large deviation fig} shows the function $I(\alpha,\beta, R_{9,4})$.

\begin{figure}[htb]
\centering
\begin{subfigure}[t]{0.45\textwidth}
\centering
\includegraphics[width=.9\textwidth]{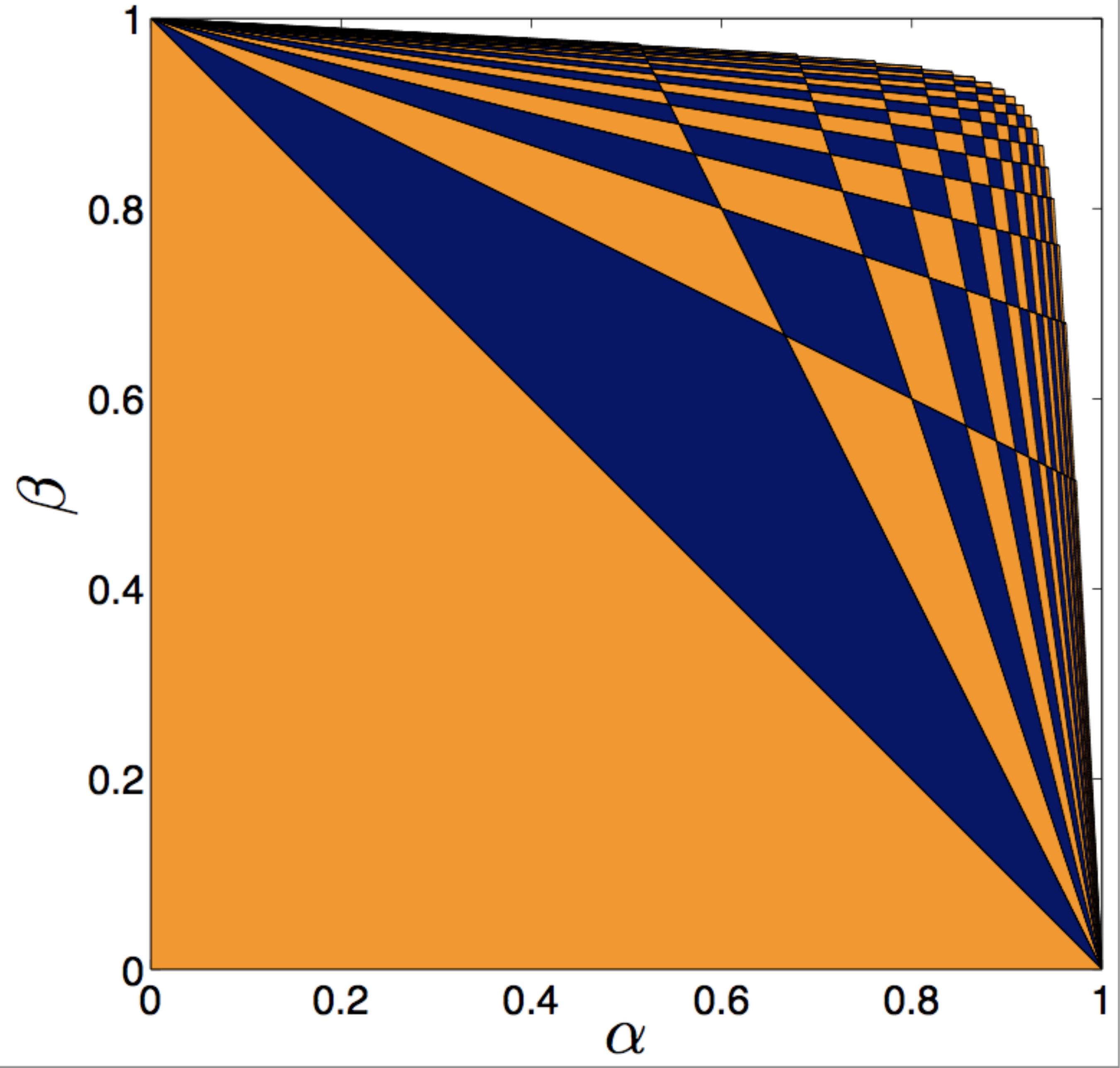}
\caption{\label{BP support fig} Boundaries of the supports of $I(\cdot,\cdot,T_\theta)$ for $\theta=2, \ldots, 20$ (from bottom to top; regions between successive boundaries shaded in alternating colors for visual guidance).}
\end{subfigure} \quad
\begin{subfigure}[t]{0.45\textwidth}
\centering
\includegraphics[width=.95\textwidth]{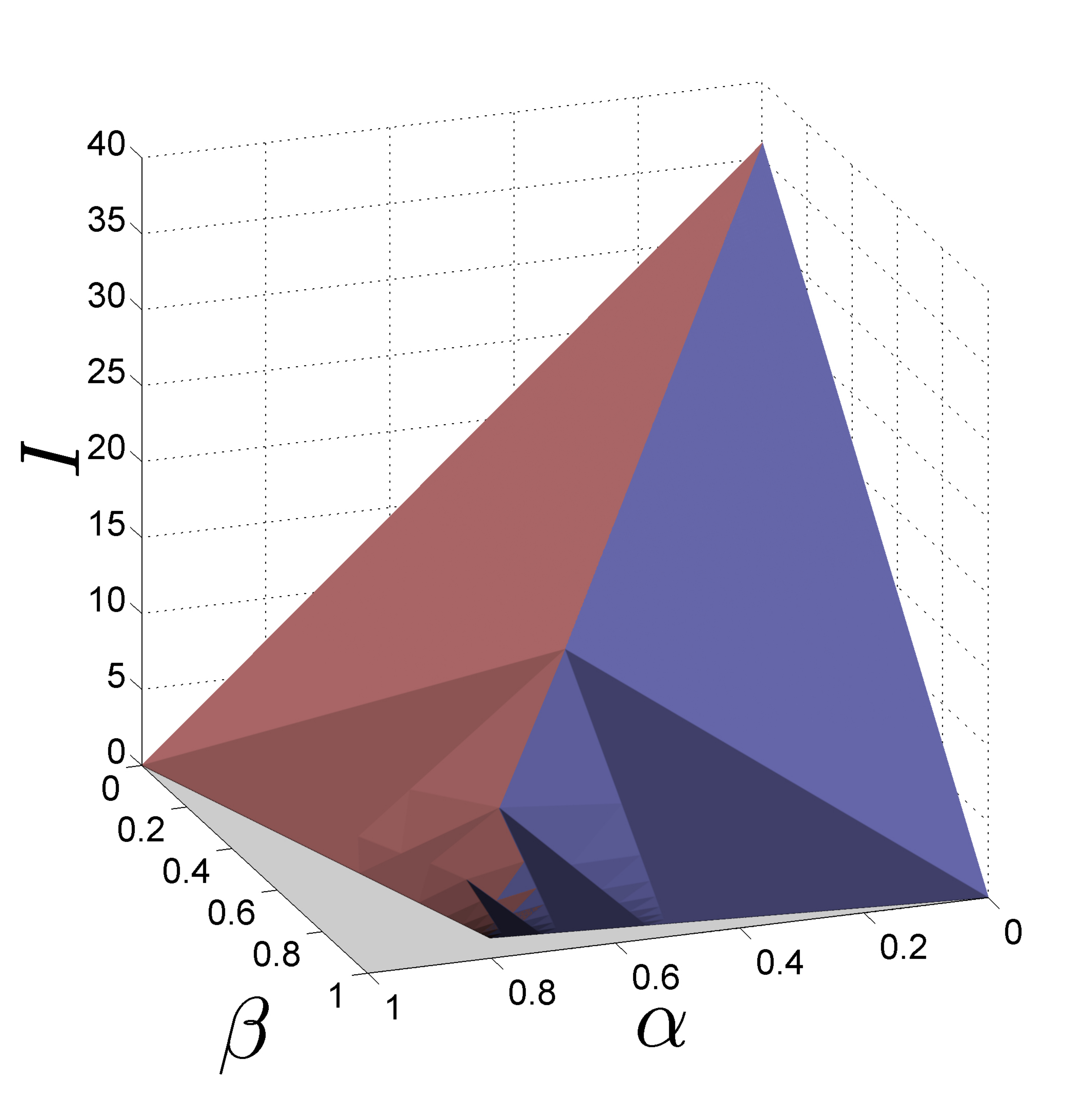}
\caption{\label{LG large deviation fig} The function $I(\alpha,\beta, R_{9,4})$. Lighter shades correspond to steeper gradients.}
\end{subfigure}
\caption{\label{large deviation figures} Examples of $I(\cdot,\cdot,\cZ)$.}
\end{figure}

It is clear that both $\gamma$ and $I$ increase if $\Z$ is enlarged, 
so it is natural to ask how they behave for large $\Z$. 
Theorem~\ref{intro-gamma-area-thm} suggests that $\gamma(\Z)/|\Z|$ 
might converge, and this is indeed true with the proper definition 
of convergence of $\Z$, which we now formulate. 

A Euclidean rectangle is denoted by $\RE_{a,b}=[0,a]\times [0,b]\subseteq\bR_+^2$.   
We define a {\it Euclidean zero-set\/}, or a {\it continuous Young diagram\/}, $\ZE$ to be 
a closed subset of  $\bR_+^2$ such that $(a,b)\in \ZE$ implies 
$\RE_{a,b}\subseteq \ZE$, and such that $\ZE$ is the closure of $\ZE\cap (0,\infty)^2$. 
For Euclidean zero-sets $\ZE_n$ and $\ZE$, we say that the sequence   
$\ZE_n$ {\it E-converges\/} to $\ZE$, $\ZE_n\Econv \ZE$, if 
\begin{enumerate}
\item [] \label{C1} (C1) for any $R>0$, $\ZE_n\cap [0,R]^2\to \ZE\cap [0,R]^2$ in Hausdorff metric; and 
\item [] \label{C2} (C2) $ \area(\ZE_n)\to \area(\ZE)$. 
\end{enumerate}
For $A\subseteq \bZ_+^2$, define 
its {\it square representation\/} by $\squrep(A)=\cup_{x\in A}(x+[0,1]^2)\subseteq \bR^2$.
Observe that, for a (discrete) zero-set $\Z$, $\squrep(\Z)$ is a Euclidean zero-set. Convergence 
of a sequence $\Z_n$ of zero-sets will mean convergence to some limit $\ZE$ of their properly 
scaled square representations. We note that we do not assume that $\ZE$ is bounded; in fact, 
unbounded continuous Young diagrams with finite area arise as a limit of a random selection 
of discrete ones; see Section~\ref{sec-random-Z}.

Next, we state our main convergence theorem, which provides the properly scaled 
limits for $\gamma$, $I$, and another extremal quantity that we now introduce.
 Call a set $A\subseteq\bZ_+^2$ {\it thin\/} if every point $x\in A$ has no 
other points of $A$ either 
on the vertical line through $x$ or on the horizontal line through $x$. We denote by
$\gamma_{\rm thin}(\Z)$ the cardinality of the smallest thin spanning set for $\Z$. 

\begin{theorem}\label{intro-d-E-theorem} There exist functions $\IE(\alpha,\beta,\ZE)$, 
$\gammaE(\ZE)=\IE(0,0,\ZE)$, and $\gammaE_{\rm thin}(\ZE)$ defined on Euclidean zero-sets $\ZE$ and $(\alpha,\beta)\in [0,1]^2$ so that the following holds. 

Assume that $\Z_n$ is a sequence of discrete zero-sets and $\delta_n>0$ is a sequence
of numbers such that $\delta_n\to 0$ and 
$\delta_n\squrep(\Z_n)\Econv \ZE$. Then 
\begin{equation}\label{dE-I}
\delta_n^{2}I(\alpha,\beta,\Z_n)\to \IE(\alpha,\beta, \ZE), 
\end{equation}
\begin{equation}\label{dE-gamma}
\delta_n^{2}\gamma(\Z_n)\to \gammaE(\ZE).
\end{equation}
and
\begin{equation}\label{dE-gamma-thin}
\delta_n^{2}\gamma_{\rm thin}(\Z_n)\to \gammaE_{\rm thin}(\ZE).
\end{equation}
If $\area(\ZE)=\infty$, then $\IE(\cdot,\cdot,\ZE)\equiv \infty$ on $[0,1)^2$ and 
$\gammaE_{\rm thin}(\ZE)=\infty$. If $\area(\ZE)<\infty$, then the following holds:
$\IE(\cdot,\cdot,\ZE)$ is finite, concave and continuous on $[0,1]^2$; $\gammaE_{\rm thin}(\ZE)<\infty$; convergence in (\ref{dE-I}) 
is uniform for $(\alpha,\beta)\in [0,1]^2$; and, if $\ZE_n$ is a sequence of Euclidean zero-sets and $\ZE_n\Econv \ZE$, then $\IE(\cdot,\cdot, \ZE_n)\to \IE(\cdot,\cdot, \ZE)$ uniformly on $[0,1]^2$.
\end{theorem}

The function $\gammaE$ can be defined through a natural Euclidean counterpart of the 
growth dynamics, replacing cardinality of two-dimensional discrete sets 
with area and cardinality 
of one-dimensional ones with length. However, if we attempt such a naive definition for
$\IE$, we get zero unless $\alpha=\beta=0$ because Euclidean sets 
can have projection lengths much larger than their areas. In fact, to properly define $\IE$, 
we need to understand the design of optimal sets for large $\Z$. Roughly, such sets are 
unions of two parts: a thick ``core'' that contributes very little to the entropy, and 
thin high-entropy tentacles. The resulting variational characterization of $\IE$
when $\ZE$ is bounded is  
given by the formula (\ref{IE-def}). We proceed to give more information on $\IE$, 
starting with the general bounds.

\begin{theorem}\label{intro-IEbounds}
For a Euclidean zero-set $\ZE$ with finite area, and $(\alpha,\beta)\in [0,1]^2$, 
\begin{equation}\label{IElb}
	\IE(\alpha,\beta,\ZE) \geq (1- \max(\alpha,\beta))\,\gammaE(\ZE)
\end{equation}
and 
\begin{equation}\label{IEub}
	\IE(\alpha,\beta,\ZE) \leq \min( ( 1-\max(\alpha,\beta) )\, \area(\ZE),\, 2(1-\min(\alpha,\beta))\,\gammaE(\ZE),\, \gammaE(\ZE)).
\end{equation}
\end{theorem}

The lower bound (\ref{IElb}) is sharp: it is attained for all $\alpha$ and $\beta$ if and
only if $\ZE=\RE_{a,b}$ for some $a,b>0$ (Corollary~\ref{E-gamma-lg}). 
The upper bound (\ref{IEub}) is almost certainly not sharp as it 
equals the trivial bound $\gammaE(\ZE)$ on a large portion of $[0,1]^2$. To what 
extent it can be improved is an interesting open problem, which we 
clarify, to some extent, by investigating the behavior of $\IE$ near the corners 
of the unit square. 

\begin{theorem}\label{intro-IEcorners}
For any Euclidean zero-set $\ZE$ with finite area, 
\begin{equation}\label{IE-10}
	\lim_{\alpha\to 1-}\frac 1{1-\alpha}
\IE(\alpha,0,\ZE)=\area(\ZE)
\end{equation}
and 
\begin{equation}\label{IE-11}
	\lim_{\alpha\to 1-}\frac 1{1-\alpha}
\IE(\alpha,\alpha,\ZE)=\gammaE_{\rm thin}(\ZE).
\end{equation}
Moreover, the following holds for the supremum over Euclidean zero-sets $\ZE$ with finite area: 
\begin{equation}\label{IE-max}
\sup_{\ZE}\frac{\IE(\alpha,\alpha,\ZE)}{\gammaE(\ZE)}=
\begin{cases}
1+o(\alpha) &\text{as $\alpha\to 0+$,}\\
2(1-\alpha)+o(1-\alpha) &\text{as $\alpha\to 1-$.}
\end{cases}
\end{equation}
\end{theorem} 
Note that (\ref{IE-max}) says that the slopes of the supremum are $0$ 
at $\alpha=0$ and $-2$ at $\alpha=1$.
These match the slopes of the two expessions involving $\gammaE$ in 
the upper bound (\ref{IEub}), while the expression involving $\area$ has the correct 
slope at $(1,0)$ due to (\ref{IE-10}). 
Therefore no linear improvement of (\ref{IEub}) is possible 
near the corners on the 
square. We obtain (\ref{IE-max}), which in particular 
implies that $\gammaE(\ZE)$ and $\gammaE_{\rm thin}(\ZE)$ 
are not always equal, by analyzing L-shaped zero-sets with 
long arms. The proof of all parts of Theorem~\ref{intro-IEcorners} again 
relies on providing a lot of
information about the design of the optimal spanning sets, which turn out to be very  
thick near $(0,0)$ and very thin near $(1,0)$ and $(1,1)$.  

We conclude with a brief outline of the rest of the paper. In Section~\ref{Preliminaries}, we prove 
some preliminary results and discuss lower bounds on $\gamma$ for small $\Z$ and for small perturbations of large $\Z$. In Sections~\ref{sec-gamma-area} 
and~\ref{sec-packing} we analyze smallest spanning sets, 
providing proofs of Theorems~\ref{intro-gamma-area-thm} and~\ref{intro-packing-minimal}.  In Section~\ref{sec-ld-existence} we prove (\ref{ld-vq}), and in Section~\ref{sec-ld-bounds} we prove general upper and lower bounds on the large deviation rate; we then complete the 
proof of Theorem~\ref{intro-ld-rate} in Section~\ref{sec-ld-support}. 
In Sections~\ref{sec-lgld} and~\ref{sec-bpld} we provide derivations for the two cases for which 
the large deviation rate $I$ is known exactly.
In Section~\ref{sec-E-limit} we introduce Hamming neighborhood growth on the continuous plane 
and prove Theorem~\ref{intro-d-E-theorem}, 
which is completed in Section~\ref{subsec-E-limit-proof}. 
Sections~\ref{ld-limit-bounds-subsec}--\ref{ld-Lshape-11-subsec} contain proofs of 
Theorem~\ref{intro-IEbounds} (completed in Section~\ref{ld-limit-bounds-subsec}) 
and Theorem~\ref{intro-IEcorners} 
(completed in Section~\ref{ld-Lshape-11-subsec})
and give some related results on $I$ for large $\Z$. We conclude with 
an application of limiting shape results for randomly 
selected Young diagrams in Section~\ref{sec-random-Z}, and 
with a selection of open problems in Section~\ref{sec-open}. 

\section{Preliminaries} \label{Preliminaries}
\subsection{The pattern-inclusion growth}\label{sec-pattern}

The neighborhood growth rules defined in Section~\ref{sec-intro} are 
part of a much larger class of pattern-inclusion dynamics, which we define in this section. 
Our reason to do so is not an attempt to develop a comprehensive theory in this general 
setting, but rather 
because we need Theorem~\ref{growth-patterns} in the proof of Theorem~\ref{intro-ld-rate}. 

Any process that takes advantage of the connectivity of the Hamming plane will 
have long range of interaction, so locality, as in cellular automata growth dynamics \cite{Gra},
is out of the question, but we retain some of its flavor by the property (G4) 
below. Again, we assume that the growth takes place on the vertex set $\bZ_+^2$.

A {\it growth transformation\/} is a map 
$\cT:2^{\bZ_+^2}\to 2^{\bZ_+^2}$ with the following properties:
\begin{enumerate}
\item[](G1) {\it solidification\/}: if $A\subseteq \bZ_+^2$, $A\subseteq \cT(A)$; 
\item[](G2) {\it monotonicity\/}: if $A_1\subseteq A_2\subseteq \bZ_+^2$, then 
$\cT(A_1)\subseteq \cT(A_2)$;
\item[](G3) {\it permutation invariance\/}: $\cT$ commutes with any permutation 
of rows and any permutation of columns of $\bZ_+^2$; and 
\item[](G4) {\it finite inducement\/}: there exists a number $K$, so that for any $A\subseteq V$ and 
$x\in \cT(A)$ there exists a set $A'\subseteq A$, such that 
$|A'|\le K$ and $x\in \cT(A')$. 
\end{enumerate} 

A {\it growth dynamics\/} starting from the initially occupied set $A$ is defined 
as in the Section~\ref{sec-intro} by $A_t=\cT^t(A)$, with $A_\infty=\cT^\infty(A)$
the set of all eventually occupied points. We say that $A\subseteq\bZ_+^2$ is {\it inert\/} 
if $\cT(A)=A$. It follows from (G4) that $A_\infty$ is always inert.  As for the 
neighborhood growth, we say that $A$ {\it spans\/} 
if $\cT^\infty(A)=\bZ_+^2$. This notion 
leads to another property of $\cT$:
\begin{enumerate}
\item[](G5)
{\it voracity\/}:  there exists a finite set $A\subseteq \bZ_+^2$ that spans.
\end{enumerate}

\begin{ex}
If $\cT$ is the neighborhood growth with $\Z$ consisting of the nonnegative $x$- and 
$y$-axis, then 
$$
\cT(A)=\{x: \text{$L^h(x)\cap A\ne\emptyset$ and $L^v(x)\cap A\ne\emptyset$}\}, 
$$
and $\cT$ fails voracity as no $A$ with an empty (horizontal or vertical) line spans. 
\end{ex}

A {\it pattern\/} is a finite subset of $\bZ_+^2$.  Two patterns are {\it equivalent\/} 
if the rows and columns of $\bZ_+^2$ can be permuted to transform one into the other, 
and {\it $0$-equivalent\/} if they could be so permuted while keeping the $0$th row and 
$0$th column fixed. We say that $A\subseteq \bZ_+^2$ {\it contains\/}
a pattern $P$ if there exist permutations $\sigma_h$ and $\sigma_v$ of 
rows and columns of $\bZ_+^2$ to obtain a set $A'$ such that 
that $P\subseteq A'$. Moreover, we say that a pattern is {\it observed\/} 
by the origin $\origin=(0,0)$ in $A$ if there exist such permutations 
$\sigma_h$ and $\sigma_v$, which also fix $0$. 

There is a bijection between growth transformations $\cT$ and finite sets of patterns $\cP$ with the
following properties: 
\begin{enumerate}
\item[](P1) $\{\origin\}\in \cP$; and 
\item[](P2) no pattern in $\cP$ is $0$-equivalent to a subset of another pattern in $\cP$.  
\end{enumerate}
We consider sets $\cP_1$ and $\cP_2$ of patterns {\it equivalent\/} 
if they have the same elements up to $0$-equivalence. 

For  a set of patterns $\cP$ that satisfies (P1--2), we call the transformation 
$\cT=\cT_\cP$ which commutes with  
any transposition of rows and any transposition of columns and satisfies
\begin{equation}\label{T-P}
\origin \in \cT(A) \text{ if and only if there exists a 
pattern $P\in\cP$, observed by $\origin$ in $A$}, 
\end{equation}
a {\it pattern-inclusion transformation\/}. Observe that $\cT_\cP$ is uniquely defined by the 
equivalence class of $\cP$.

\begin{theorem}\label{growth-patterns} A composition of two growth transformations is a growth 
transformation. Moreover, any map $\cT:2^{\bZ_+^2}\to 2^{\bZ_+^2}$ 
is a growth transformation if and only if 
it is a pattern inclusion transformation.  
\end{theorem}

\begin{proof} The first statement is easy to check by (G1--4). 
To prove the second statement assume first that  
$\cT$ is a growth transformation. 
Then gather all inclusion-minimal sets $A$ that result in 
$\origin\in \cT(A)$; there are finitely many $0$-equivalence 
classes of them by (G4), and so we can collect one pattern 
per $0$-equivalence class to form $\cP$. The converse statement is 
again easy to check by definition. 
\end{proof}

We now formally state the connection to the neighborhood growth.  

\begin{prop}\label{general-to-neighborhood}
A neighborhood growth transformation is characterized by a set $\cP$ of patterns that  
are included in the two lines through $\origin$. It is voracious if and only if its 
zero-set $\Z$ is finite.
\end{prop}

We omit the 
simple proof of this proposition. From now on, we will assume that all zero-sets are finite. 

We end this section with 
an example that show that (G4) is indeed a necessary assumption if we want the 
set $\cP$ to be finite (which is in turn a crucial property for our application). 

\begin{ex} We  give an example of a dynamics given by  (\ref{T-P}) with an infinite set 
$\cP$ of finite patterns that satisfies (G1)--(G3) and (G5), but not (G4). Define 
$\cP$ to comprise $\{\origin\}$ and the following patterns
$$
\begin{matrix}
\origin & \X\\
\end{matrix}\quad,\quad
\begin{matrix}
\X & \X & \X\\
\X\\
\origin
\end{matrix}\quad,\quad
\begin{matrix}
  & \X & \X & \X\\
\X & \X\\
\X\\
\origin
\end{matrix}\quad,\quad
\begin{matrix}
  &    & \X & \X & \X\\
  & \X & \X\\
\X & \X\\
\X\\
\origin
\end{matrix}\quad,\quad
\ldots
$$
(Here, we denote by $\X$ a point in the pattern.) No pattern above is $0$-equivalent to 
a subset of another, and 
a 2 by 1 rectangle of occupied sites spans.  
\end{ex}

\subsection{Perturbations of $\Z$}\label{sec-perturbations}

In this section, 
we prove some results on the  
effects that small perturbations to a zero-set $\Z$ have on 
the spanning sets.  We start with some notation.

Fix a zero-set $\Z$ and an integer $k\ge 1$. We define the following two Young diagrams, obtained by deleting the $k$ largest 
(bottom) rows (resp., columns) of $\Z$,
\begin{equation*}
\begin{aligned}
&\Z^{\downarrow k}\,\,=\{(u,v-k): (u,v)\in \Z, v\ge k\},\\
&\Z^{\leftarrow k}=\{(u-k,v): (u,v)\in \Z, u\ge k\}.
\end{aligned}
\end{equation*}
Then we let
\begin{equation*}
\begin{aligned}
&\zswk = (\Z^{\downarrow k})^{\leftarrow k}
\end{aligned}
\end{equation*}
and 
$$\Z^{\myll k}=\Z\setminus ((k,k)+\Z^{\swarrow k}),$$ which is the set comprised of the
$k$ longest rows and columns of $\Z$. 
Suppose $A\subseteq \bZ_+^2$, and let
$$
\angk = \{x \in A : \row(x,A)> k \text{ or } \col(x,A) > k\}
$$
denote the set of points in $A$ that lie in either a row or a column with 
at least $k$ other points of~$A$.  For example, $A_{>1}$ is the set of non-isolated 
points in $A$.  The next two lemmas let us identify low-entropy spanning sets 
for perturbations of $\Z$.

\begin{lemma} \label{retract}

If $A$ spans for $\Z$, then $\angk$ spans for $\zswk.$

\end{lemma}

\begin{proof}
For each $x\in \bZ_+^2$, 
$$\row(x,\angk) \ge (\row(x,A)-k)\vee 0\text{ and }\col(x,\angk)\ge (\col(w,A)-k)\vee 0,$$ 
since the vertices removed from $A$ to form $\angk$ are on both horizontal 
and vertical lines with at most $k$ vertices of $A$.  
Therefore, if $\cT$ and $\cT_k$ are the respective 
growth transformations corresponding to $\Z$ and $\zswk$, 
then $x \in \cT(A)\setminus A$ implies that $x\in \cT_k(\angk)\setminus \angk$.  
By induction, $\cT^t(A)\setminus A \subseteq \cT_k^t(\angk)\setminus \angk$ 
for all $t\ge 1$.  Since $A$ spans for 
$\Z$ and $A\setminus \angk$ has at most $k$ sites in each line, 
for every $x\in \bZ_+^2$, $\row(x,\cT_k^t(\angk))\to \infty$ as $t\to\infty$, 
so $\angk$ spans for $\zswk$.\end{proof}
%

\begin{lemma}\label{k-proj}
 
Let $A\subseteq \bZ_+^2$ and $k$ be a nonnegative integer.  Then 

$$|\pi_x(\angk)| + |\pi_y(\angk)| \leq \left(1+\frac{1}{k+1}\right)|\angk|.$$
\end{lemma}

\begin{proof}
 Each point in $\angk$ shares a line with at least $k$ other points in $\angk$, and we use this fact to subdivide $\angk$ into three disjoint sets.  Let 
$$
A_h = \{ x \in \angk : \row(x,\angk)>k\}.
$$   
Thus every point of $A_h$ shares a row with at least $k$ other points of $\angk$, and therefore 
with at least $k$ other points of $A_h$. Moreover, let $A_0$ be the set of points that 
are not in $A_h$ but share a column with at least one point in $A_h$. Lastly, let
$A_v = \angk \setminus (A_h\cup A_0)$. Each point $x\in A_v$ is in a column with at least 
$k$ other points of $A_v$. Indeed, $x$ shares a column with at least $k$ other points of $\angk$, 
but none of the points in this column can be in $A_h$ (as otherwise $x$ would be in $A_0$)
or in $A_0$ (as every point that shares a column with a point in $A_0$ is itself in $A_0$).
 
Each nonempty row in $A_h$ contains at least $k+1$ points of $A_h$, so 
$|\pi_y(A_h)| \leq \frac{1}{k+1}|A_h|.$   Similarly, $|\pi_x(A_v)| \leq \frac{1}{k+1}|A_v|.$
Furthermore, $\pi_x( A_h \cup A_0) = \pi_x(A_h).$  Trivially, we have $|\pi_x(A_h)|\leq |A_h|$, $|\pi_y(A_v)| \leq |A_v|$ and $|\pi_y(A_0)| \leq |A_0|.$ Then,
\allowdisplaybreaks
\begin{align*}
|\pi_x(\angk)| + |\pi_y(\angk)| &= |\pi_x(A_v\cup A_h\cup A_0)| + |\pi_y(A_v\cup A_h\cup A_0)|\\
&\leq |\pi_x(A_v)| + |\pi_x(A_h\cup A_0)| + |\pi_y(A_v)| + |\pi_y( A_h)| + |\pi_y(A_0)|\\
&\leq \frac{1}{k+1}|A_v| + |A_h| + |A_v| + \frac{1}{k+1}|A_h| + |A_0|\\
&\leq \left(1+\frac{1}{k+1}\right)(|A_v| + |A_h| + |A_0|)\\
&= \left(1+\frac{1}{k+1}\right)|\angk|.	
\end{align*}
This completes the proof.
\end{proof}

Next, we give a perturbation result that addresses removal of the shortest 
lines from $\Z$. In  particular, we conclude that this operation cannot decrease $\gamma$ by more than 
the number of removed sites. To put the result in perspective, we note 
that it is not true that $\gamma$ decreases by at most $k$ if we remove {\it any\/} $k$ sites. 
For the simplest counterexample, observe that $\gamma(R_{2,2})=4$ 
(use Proposition~\ref{lp-gamma} below or note that, with 3 initially occupied points, no point 
is added after time $1$) 
but $\gamma(R_{2,2}\setminus\{(1,1)\})=2$ (as any pair of non-collinear points spans).

\begin{theorem}\label{gamma-perturbation}
Let $\Z$ be any zero-set.  Suppose $A'$ spans for $\Z\cap R_{a,b}$, then there exists $A\supseteq A'$, which spans for $\Z$ and is such that
$$
\abs{A} = \abs{A'} + \abs{\Z\setminus R_{a,b}}.
$$
Furthermore, if $A'$ is thin, then $A$ can be made thin as well. 
Therefore, for any $\Z$ and $a,b\in[1,\infty]$, 
\begin{equation*}
\begin{aligned}
&\gamma(\Z\cap R_{a,b})\ge \gamma(\Z)-|\Z\setminus R_{a,b}|, \\
&\gamma_{\rm thin}(\Z\cap R_{a,b})\ge \gamma_{\rm thin}(\Z)-|\Z\setminus R_{a,b}|.
\end{aligned}
\end{equation*}
\end{theorem}

\begin{proof}
We may assume that $a=\infty$ 
and that $\Z\setminus  R_{\infty,b}$ consists of a single row, the 
topmost (shortest) row of $\Z$, of cardinality $k$; we then iterate to obtain the general 
result. Let $A'$ be a spanning set for the dynamics $\cT'$ with zero-set 
$\Z'=\Z\cap R_{\infty,b}$. We will construct a set $A\supseteq A'$ of cardinality $|A'|+k$ 
that spans for $\Z$.   

Order $\bZ_+^2$ in an arbitrary fashion. 
Slow down the $\cT'$-dynamics by occupying a single site at each time step, the first site in the 
order that can be occupied, with one exception: 
when a vertical line contains 
enough sites to become completely occupied under the standard synchronous rule, make 
it completely occupied at the next time step.

Mark vertices that are made occupied one-at-a-time according to the ordering on $\bZ_+^2$ in red, 
and vertices that are made occupied by completing a vertical line in black.  
Let $L_1,\ldots, L_k$ be the first $k$ vertical lines in the slowed-down dynamics for $\cT'$ that 
become occupied; say that $L_k$ becomes occupied at time $t$. Choose $k$ black sites, one on each of 
the $k$ lines, and adjoin them to $A'$ to form the 
set $A$ (if $A'$ is thin, choose these black points so that no two share a row with each other or with any points of $A'$, then $A$ is also thin). Define the slowed-down version of $\cT$ started from $A$ so that 
it only tries to occupy the site, or sites,
occupied by the $\cT'$-dynamics. We claim that, up to $t$, such 
dynamics occupies every site that $\cT'$ does from $A'$. Indeed, 
the only possible problem arises when a line in $\cT'$-dynamics from $A'$
contains $b$ occupied sites and fills in the next step, and then the $\cT$-dynamics 
from $A$ does the same by construction. 
After time $t$, $k$ vertical lines are occupied and thus the horizontal count 
of any site is at least $k$ and the two dynamics agree. 
\end{proof}

\subsection{The enhanced neighborhood growth}\label{sec-enhanced-definition}

We will need another useful generalization of the neighborhood growth, which will play 
a key role in the proof of Theorem~\ref{intro-d-E-theorem}. In this section 
we only give its definition, as 
it will be encountered in the proof of Theorem~\ref{when-done}. We postpone a more detailed 
study until Section~\ref{subsec-enhanced}. 
 
The {\it enhancements\/} 
$\vec f=(f_0, f_1,\ldots)\in \bZ_+^\infty$ 
and $\vec g=(g_0, g_1,\ldots)\in \bZ_+^\infty$
are sequences of positive integers. These increase   
horizontal and vertical counts, respectively, by fixed amounts.  
The {\it enhanced neighborhood growth\/} 
is then given by the triple $(\Z,\vec f, \vec g)$, which determines the transformation $\cT$
as follows:
$$
\cT(A)=A\cup\{(u,v)\in\bZ_+^2: (\row((u,v),A)+f_v,\col((u,v),A)+g_u)\notin \Z\}.
$$
The usual neighborhood growth given by $\Z$ is the same as its enhancement given by 
$(\Z,\vec 0,\vec 0)$, and we will not distinguish between the two.

\subsection{Completion time}\label{sec-completion-time}

Started from any finite set, the neighborhood growth clearly reaches its final state in 
a finite number of steps. We will now show that in fact this is true for any initial set, and that 
the number of steps depends only on $\Z$. 

\begin{theorem}\label{when-done} There exists a time $\tmax=\tmax(\Z)$ so that
for any set $A\subseteq\bZ_+^2$, not necessarily finite, 
$$
\cT^{\tmax+1}(A)=\cT^\tmax(A).
$$
\end{theorem}

\begin{proof}
We will prove the theorem for the more general enhanced neighborhood growth
dynamics given by $(\Z,\vec h, \vec 0)$, for some horizontal enhancement 
$\vec h=(h_0,h_1,\ldots)\in \bZ_+^\infty$, also proving that
$\tmax$ does not depend on $\vec h$. 

We prove this by induction on the number of lines in $\Z$. If $\Z=\emptyset$, then clearly the
dynamics is done in a single step. 

Now take an arbitrary $\Z$ whose longest row contains $a$ sites and fix an $\vec h$. First suppose the initial set $A$ has a row count of at least $a$ on some horizontal line (the $x$-axis, say). (We emphasize that 
all counts include the numbers from the enhancement sequence.)  Then in one step, all points on the $x$-axis become occupied.  If we let $A'$ be the set formed by running the dynamics for one step, and let $A'' = A' \setminus \{(x,0):x\in \bZ_+\}$, then the dynamics given by $(\Z,\vec h, \vec 0)$ started from $A'$ coincides with the dynamics given by $(\Z^{\downarrow 1}, (0,h_1, h_2, \ldots), \vec 0)$ started from $A''$ (except on the $x$-axis, which no longer has any effect on the running time).  By the induction hypothesis, in this case the original dynamics started from $A$ therefore terminates in at most $\tmax(\Z^{\downarrow 1}) + 1$ steps.

Fix an integer $k<a$, and assume now that the initial set $A$ has a row count of $k$ on some horizontal line, and every horizontal line has a row count of at most $k$.  Let $t_0$ be the 
first time at which there is a horizontal line with (at least) $k+1$ occupied 
sites. (Let $t_0=\infty$ if there is no such time.) 

Let  $L$ be any horizontal line with $k$ occupied sites 
at time $0$.  Assume without loss of generality that $L$ 
is the $x$-axis and that $[0,k-1-h_0]\times \{0\}$ are the sites occupied on $L$ at time $0$. 
No site above $[k-h_0, \infty)\times \{0\}$ 
becomes occupied before time $t_0$; if it did, the site below it on the 
$x$-axis would become occupied at the same time. 
Thus the dynamics above $[0,k-1-h_0]\times \{0\}$ 
behaves like the dynamics with zero-set $\Z^{\downarrow 1}$, 
and a different horizontal enhancement sequence $\vec f$, which takes into account the contributions of occupied sites outside of $[0,k-1-h_0]\times [1,\infty)$ to the row counts. By the induction hypothesis, these dynamics terminate 
by some time dependent only on $\Z^{\downarrow 1}$.  Therefore, either $t_0 \le \tmax(\Z^{\downarrow 1}) + 1$ or $t_0 = \infty$.  In the latter case, the original $(\Z,\vec h,\vec 0)$-dynamics terminate by time $\tmax(\Z^{\downarrow 1})$, so we can assume $t_0\le \tmax(\Z^{\downarrow 1})+1$.

Assume that $a=a_0\ge a_1\ge \ldots a_k>0$ are the rows of $\Z$. 
The arguments above imply that $\tmax(\Z)\le (a+1)(\tmax(\Z^{\downarrow 1})+1)$. This, together with $\tmax(\emptyset)=1$, gives 
$$
\tmax(\Z)\le (k+2)(a_0+1)(a_1+1)\cdots (a_k+1),
$$
which ends the proof.
\end{proof}

\subsection{The line growth bound}\label{sec-lgbound}



The first result on the smallest spanning sets on the Hamming plane was this simple 
formula about line growth from \cite{BBLN}. 

\begin{prop} \label{lp-gamma}
For $a,b\ge 0$, $\gamma(R_{a,b})=ab$.
\end{prop}

\begin{proof}
See Section 1 of \cite{BBLN} for a simple inductive proof, 
or Theorem~\ref{lgld-thm}.
\end{proof}

\begin{cor} \label{lp-bound}
For any zero set $\Z$, $\gamma(\Z)\ge \max\{ab:R_{a,b}\subseteq \Z\}$.
\end{cor}

\begin{proof} This follows from Proposition~\ref{lp-gamma}, and the fact that $\Z'\subseteq \Z$ implies $\gamma(\Z')\leq \gamma(\Z)$.
\end{proof}

We call the bound in  Corollary~\ref{lp-bound} the {\it line growth bound\/}. 
It is somewhat surprising that the inequality is, in fact, in many cases equality.
For example, it is equality for bootstrap percolation with arbitrary $\theta$ 
(which follows from Proposition~\ref{bpld-prop}) and when 
the $\Z$ is a union of two rectangles (a special case of a more general 
result from \cite{CGP}). On the other hand, it easily follows 
from Theorem~\ref{intro-gamma-area-thm} 
that the line growth bound can be, in general, very far from
equality when $\Z$ is large. 
In this section we give a general lower bound on $\gamma$ that tends to work 
better for small $\Z$; in particular, it proves that in general 
equality does not hold when $\Z$ is a symmetric 
zero set which is the union of three rectangles.  

%

\begin{theorem}\label{lb-general}
For any choice of a comparison rectangle
$R_{a,b}\subseteq \Z$ and a Young diagram $Y\subseteq R_{a-1,b-1}$,
$$
\gamma(\Z)\ge \frac 12 \min_{(k,\ell)\in \partial_oY}\left(kb+\ell a-k\ell+
\gamma(\Z^{\downarrow\ell})+\gamma(\Z^{\leftarrow k})\right).
$$
\end{theorem}

\begin{proof} Order the lines of $\bZ_+^2$ in an arbitrary fashion.
Assume $A$ is a finite spanning 
set for $\cZ$. We will construct
a finite sequence $\vec S$ of lines (dependent on $A$), by a recursive specification 
of sequences $\vec S_i$ of $i$ lines. 

Consider the line growth $\cT'$ with zero-set $R_{a,b}$. Note that $A$ spans for 
the growth dynamics $\cT'$; 
we now consider a slowed-down version. Let $A_0'=A$ and $\vec S_0$ the empty sequence. 
Given the sequence $\vec S_i$, $i\ge 0$, $A_i'$ is the union of $A$ and all lines in $\vec S_i$. 
Assume  $\vec S_i$
consists of $k$ vertical and $\ell$ horizontal lines, with $k+\ell=i$.  

If $(k,\ell)\in Y$, 
examine lines of $\bZ_+^2$ in order until a line $L$ is found on which $\cT'(A_i')$ adds a point 
and thus immediately makes it fully occupied (since $\cT'$ is a line growth). Adjoin $L$ to the end of the 
sequnce $\vec S_i$ to obtain 
$\vec S_{i+1}$. 
If $L$ is horizontal
(resp.~vertical), 
define its {\it mass\/} to be $a-k>0$ (resp.~$b-\ell>0$). The mass of $L$  
is a lower bound 
on the number of points in $A\cap L$ that are not on any of the preceding lines in the sequence.  

If $(k,\ell)\notin Y$,
the sequence stops, that is, $\vec S=\vec S_i$. 
As we add only one line to the sequence each time, 
the final counts $k$ and $\ell$ of vertical and horizontal lines 
satisfy $(k,\ell)\in \partial_o Y$. Let $m_h$ and $m_v$ be the respective final masses of the 
horizontal and vertical lines.  

The key step in this proof is the observation that total mass $m_h+m_v$ 
only depends on $k$ and
$\ell$ and not on the positions of vertical and horizontal lines in the sequence. 
Indeed, if $L$ is followed by $L'$ in $\vec S$, and the two lines are of different type, 
and a new sequence is formed by swapping $L$ and $L'$, 
the mass of $L'$ increases by $1$, while the mass of $L$ decreases
by $1$. 
Thus the total mass can be obtained by starting with all vertical lines:
\begin{equation}\label{lb-general-eq1}
m_h+m_v=kb+\ell(a-k)=kb+\ell a-\ell k.
\end{equation}

For a possible sequence $\vec S$ of lines, let $\gamma_{\vec S}$ be the minimal size of 
a set that spans (for $\Z$) and generates the sequence $\vec S$. Then, simultaneously, 
\begin{equation}\label{lb-general-eq2}
\begin{aligned}
\gamma_{\vec S}&\ge m_h+\gamma(\Z^{\downarrow\ell}),\\
\gamma_{\vec S}&\ge m_v+\gamma(\Z^{\leftarrow k}).\\
\end{aligned}
\end{equation}
Now we add the two inequalities of  (\ref{lb-general-eq2}) and use (\ref{lb-general-eq1}) to 
get
$$
2\gamma_{\vec S}\ge kb+\ell a-k\ell+
\gamma(\Z^{\downarrow\ell})+\gamma(\Z^{\leftarrow k}).
$$
Finally, we observe that 
$$
\gamma(\Z)=\min\{\gamma_{\vec S}: {\vec S}\text{ a possible sequence}\}
$$
to end the proof. 
\end{proof}

 
\begin{cor}
Let $\cZ=R_{b,c}\cup R_{c,b}\cup R_{a+b,a+b}$, with $a+b<c$. Then 
$$
\gamma(\Z)\ge 
\begin{cases}
bc+\frac12 a^2  & a\le b\\
bc+\frac18(a+b)(3a-b) &a>b.
\end{cases}
$$
\end{cor}
Note that, if $bc\ge (a+b)^2$, the line growth 
bound is $\gamma(\Z)\ge bc$. 

\begin{proof}
We use the comparison square $R_{a+b,a+b}$, and  
$Y=\{(k,\ell):k+\ell\le i-1\}$, for some $i\le a+b$ to be chosen later. 
Then $k+\ell=i$ when 
$(k,\ell)\in\partial_oY$.
Further, we use the bounds $\gamma(\Z^{\downarrow\ell})\ge \gamma(R_{b,c-\ell})$ and 
$\gamma(\Z^{\leftarrow k})\ge \gamma(R_{c-k,b})$ in Theorem~\ref{lb-general} to get 
\begin{equation*}
\begin{aligned}
\gamma(\Z)&\ge {\textstyle\frac 12}\min_{0\le k\le i}(i(a+b)-k(i-k)+b(c-\ell)+b(c-k))\\
&=bc+ {\textstyle\frac 12}ai- {\textstyle\frac 12}\max_{0\le k\le i}k(i-k)\\
&\ge bc+ {\textstyle\frac 12}ai- {\textstyle\frac 18}i^2.
\end{aligned}
\end{equation*}
We are free to choose $i$; if $a\le b$, then the optimal choice is $i=2a$, otherwise
it is $i=a+b$, which gives the desired inequality.
\end{proof}


\section{Smallest spanning sets}

\subsection{Proof of Theorem~\ref{intro-gamma-area-thm}}\label{sec-gamma-area}

The steps in the proof of Theorem~\ref{intro-gamma-area-thm}
are given in the next three lemmas. The first one demonstrates that 
when the initial set $A_0$ is itself a Young diagram, the growth dynamics are very simple.

\begin{lemma}\label{Z-grows} Assume $A_0$ is a Young diagram. Then $A_0$ spans if 
and only if $\Z\subseteq A_0$. 
\end{lemma}

\begin{proof} It is easy to see that $\cT$ preserves the property of being a Young diagram. 
Assume first that $A_0=\Z$. Take $z=(x,y)\in \partial_o(A_0)$. Then $\row(z,A_0)=x$ and 
$\col(z,A_0)=y$, and $(x,y)\notin\Z$, so $z\in A_1$. Let $e_1=(1,0)$ and $e_2=(0,1)$. 
It follows the translation $A_0+e_1$ is included in $A_1$, and therefore 
$A_0+[0,n]e_1\subseteq A_n$; similarly, $A_0+[0,n]e_2\subseteq A_n$. To conclude that $A_0$ spans, 
observe that $(\Z+[0,\infty)e_1)\cup(\Z+[0,\infty)e_1)$ spans in a single step. 

If $\Z\not\subseteq A_0$, there exists $z\in \Z\cap \partial_o(A_0)$. Then $z\notin A_1$ and 
therefore no point in $z+\bZ_+^2$ is in $A_1$. By induction $z\notin A_n$ for all $n$.
\end{proof}

To prove the lower bound in 
Theorem~\ref{intro-gamma-area-thm} we consider the case where the initial set is a 
union of two translated Young diagrams. To be more precise, we say that 
$A_0\subseteq \bZ_+^2$ is a {\it two-Y set\/} if 
$A_0=(y_1+Y_1)\cup (y_2+Y_2)$, 
where $Y_1$ and $Y_2$  are
Young diagrams, $y_1,y_2\in \bZ_+^2$, and no line intersects both $(y_1+Y_1)$ and 
$(y_2+Y_2)$. 

\begin{lemma}\label{two-Y} Assume $A_0$ is a two-Y set. 
If $A_0$ spans, then $|A_0|\ge \frac 12|\Z|$. 
\end{lemma}

\begin{proof} Our proof will be by induction on the number of horizontal lines 
that intersect $\Z$. If this number is $0$, the claim is trivial. 
Otherwise, let $a_0>0$ be the number of sites on the largest (i.e., bottom) 
line of $\Z$. Observe that the initial set 
consiting of $a_0-1$ vertical lines is inert.

Further, let $h_0$  and $k_0$ be the respective 
numbers of sites on bottom lines for $Y_1$ and $Y_2$. Then $h_0+k_0\ge a_0$, 
as otherwise $A_0$ would be covered by $a_0-1$ vertical lines. Therefore either 
$h_0\ge \frac 12a_0$ or  $k_0\ge \frac 12a_0$; without loss of generality we assume the 
latter. Let $Y_2'=Y_2^{\downarrow 1}$, 
$A_0'=(y_1+Y_1)\cup (y_2+Y_2')$, and $\Z'=\Z^{\downarrow1}$. By making the horizontal line that contains $k_0$ sites of $y_2+Y_2$ occupied in the original configuration $A_0$, we see that $A_0'$ spans for the dynamics with zero-set $\Z'$. 
By the induction hypothesis, $|A_0'|\ge \frac 12|\Z'|$, and then
$$
|A_0|=|A_0'|+k_0\ge \frac 12|\Z'| +\frac 12a_0=\frac 12|\Z|.
$$
\end{proof}

\begin{lemma}\label{get-two-Y} Assume $A_0$ spans. Then there exists a two-Y set $A_0'$, 
which spans and has $|A_0'|=2|A_0|$.  
\end{lemma}
\begin{remark}
A similar proof to the one below also shows that there exists a thin set $A_0''$, which spans and has $|A_0''|=2|A_0|$.
\end{remark}

\begin{proof}
Assume $A_0\subseteq R$ for some rectangle $R=[0,a-1]\times [0,b-1]$.
Let $R'=[0,2a-1]\times [0,b-1]$ be the horizontal double of $R$. Note that 
$R'\setminus R$ spans.  


Permute the columns of $A_0$ so that the column counts are in nonincreasing order, then permute the rows of $A_0$ so that the row counts are in nonincreasing order; in the sequel we refer to this set as $A_0$, as it clearly spans if and only if the original set spans.  Fix a vertical line $L$ intersecting $R'$, containing $k>0$ sites of $A_0$. 
Create a contiguous interval
of $k$ occupied sites on $L$ just above $L\cap R'$ (in particular, outside $R'$). 
Perform this operation for all vertical lines, and note that the resulting set forms a Young diagram.  Also perform an analogous operation for the horizontal lines, 
adding sites just to the right of $R'$. 
Finally, erase all the sites inside $R'$ to define 
$A_0'$. Clearly,  $|A_0'|=2|A_0|$, and $A_0'$ is a two-Y set.  Figure~\ref{two-Y fig} illustrates the construction of $A_0'$ from $A_0$.

\begin{figure}[htb]
\includegraphics[width=\textwidth]{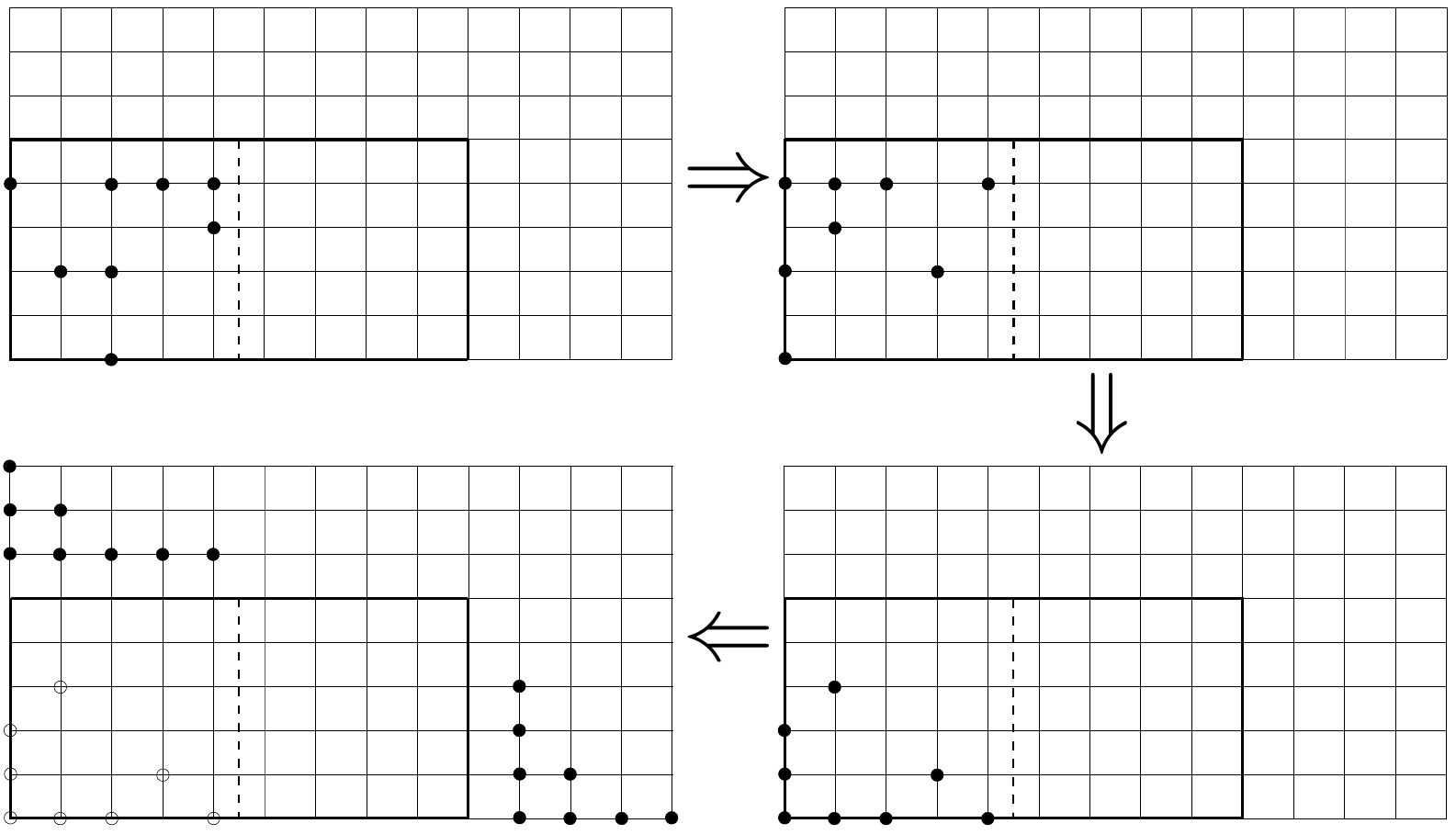}
\caption{\label{two-Y fig} Construction of a two-$Y$ set from $A_0$.  Clockwise from top left: the set $A_0$; columns sorted by descending counts; rows sorted by descending counts; the two-$Y$ set $A_0'$.  Thick lines indicate the rectangle $R'$, and the half of $R'$ to the left of the dotted line is $R$.}
\end{figure}

To see that $A_0'$ spans, it is enough to show that it eventually occupies 
every point in $R'\setminus A_0\supseteq R'\setminus R$.  

Assume, in this paragraph, that the initial set is $A_0\subseteq R'$. 
We claim that, if a point $x\notin R'$ gets occupied at any time $t$, then any line through $x$ that 
intersects $R'$ is fully occupied. This is proved by induction on $t$. The claim is trivially 
true at $t=0$, and assume it holds at time $t-1\ge 0$. Suppose  $x\notin R'$ gets occupied at time 
$t$. If its neighborhood does not intersect $R'$, then $\cT^t(A_0)=\bZ_+^2$. 
Assume now that $L^h(x)\cap R'\ne\emptyset$. Then, by the induction hypothesis, 
any $y\in L^h(x)$ has vertical and horizontal counts at time $t$ at least as large as those of $x$
and thus also becomes occupied. An analogous statement holds if $L^v(x)\cap R'\ne\emptyset$.
This proves the claim, which implies that no site outside $R'$ ever helps in occupying a site in
$R'$. 

Due to the argument in the previous paragraph, we may 
only allow the dynamics from both $A_0$ and $A_0'$ to occupy sites 
within the rectangle $R'$.   

We now claim, and will again 
show by induction on time $t\ge 0$, that every site in $R'\setminus A_0$ occupied 
at time $t$ starting 
from $A_0$ is also occupied starting from $A_0'$. This claim is trivially true at $t=0$. Assume 
the claim at time $t-1$. Fix any point $z\in R'$. Let $L$ be the horizontal line through $z$. 
By the induction hypothesis, 
$$
L\cap  (\cT^{t-1}(A_0)\setminus A_0)\subseteq L\cap \cT^{t-1}(A_0'),
$$
and by construction 
$$
|L\cap A_0|=|L\cap A_0'|,
$$
therefore 
\begin{equation}\label{get-two-Y-eq1}
|L\cap \cT^{t-1}(A_0')|\ge |L\cap  \cT^{t-1}(A_0)|.
\end{equation}
By an analogous argument, the same inequality holds if $L$ is a vertical line. If $z\in \cT^{t}(A_0)\setminus A_0$, then 
$$(\row(z,\cT^{t-1}(A_0)), \col(z,\cT^{t-1}(A_0)))\notin \Z.$$
Therefore, by~\eqref{get-two-Y-eq1}, $$(\row(z,\cT^{t-1}(A_0')), \col(z,\cT^{t-1}(A_0')))\notin \Z,$$
which implies  $z\in  \cT^{t}(A_0')$. This 
establishes the induction step and ends the proof.
\end{proof}

\begin{proof}[Proof of Theorem~\ref{intro-gamma-area-thm}] 
The upper bound is an obvious consequence of Lemma~\ref{Z-grows}, while the 
lower bound follows from Lemmas~\ref{two-Y} and~\ref{get-two-Y}.  
\end{proof} 


\subsection{Proof of Theorem~\ref{intro-packing-minimal}}\label{sec-packing}

Theorem~\ref{intro-packing-minimal} is an immediate consequence of the 
following result. 

\begin{theorem}\label{packing} 
Assume $\Z\subseteq R_{a,b}$. 
Assume that $A\subseteq \bZ_+^2$ that spans. Then there exists a set $B\subseteq R_{a,b}$ 
that spans and has $|B|\le |A|$.  
\end{theorem}

\begin{proof}[Proof of Theorem~\ref{packing}]
Assume that $A\subseteq R_{M,N}$ is a finite set that 
spans and $M>a$, $N\ge b$. We claim that 
there is a set $B\subseteq  R_{M-1,N}$ that also spans and 
$|B|\le |A|$. Without loss of generality, we will restrict our dynamics to the rectangle $R_{M,N}$ 
throughout the proof. 

We may assume that all row and column occupancy counts satisfy $|L^h(0,i)\cap A|\le a$, $0\le i<N$ and
$|L^v(i,0)\cap A|\le b$, $0\le i<M$. 
Let $$k=\min\{|L^v(i,0)\cap A|:0\le i<M\} \in [0,b]$$ be the smallest of the column
counts. We prove our claim 
by induction on $k$. If $k=0$, the claim is trivial. 

We now prove the induction step. Assume $k>0$ and that the rightmost column in $R_{M,N}$ contains exactly $k$ occupied points, that is, $|L^v(M-1,0)\cap A|=k$, and 
$|L^v(i,0)\cap A|\ge k$ for $i<M-1$.  We define the time $T$ to be the first time in the dynamics at which a point, $(M-1,j_0)$ say, on the last column becomes occupied {\em and} there exists an unoccupied point $(i_0, j_0)$ in the row $L^h(M-1,j_0)$.

First consider the case $T=\infty$.  Then every time a point $x$ in the column $L^v(M-1,0)$ becomes occupied, the entire row $L^h(x)\cap R_{M-1,N}$ also becomes occupied.  Therefore, apart from the initially occupied points in $L^v(M-1,0)$, this column plays no role in the dynamics within $R_{M-1,N}$.  Thus, each initially occupied point $z\in L^v(M-1,0)$ can be moved to an initially unoccupied location on the same row $L^h(z)\cap R_{M-1,N}$.  Such unoccupied locations exist since we assumed $M>a$ and all row occupancy counts are at most $a$. Furthermore, the resulting initial configuration eventually fills the box $R_{M-1,N}$, which spans.

Now consider the case $T<\infty$, and consider the configuration $X= \cT^{T-1}(A)$.  Let $J$ be the collection of row indices $j$ for which the $j^{\text{th}}$ row is fully occupied in $X$ ($|L^h(0,j) \cap X| = M$), and $(M-1,j)\notin A$.  We will now build a new initially occupied set $A_1$ (see Figure~\ref{column swap fig} for guidance on this construction).  First, consider the points in the $i_0^\text{th}$ column that are occupied in $A$, but not on any of the rows with indices in $J$.  Populate the last column ($M-1$) of $A_1$ with these points, keeping their rows the same.  Next, consider the points on the last column of $A$, and populate the $i_0^\text{th}$ column of $A_1$ with these points, again keeping their rows the same, in addition to the points in the $i_0^\text{th}$ column of $A$ that lie on the rows indexed by $J$ ($\{(i_0,j)\in A: j\in J\}$).  Finally, let $A_1$ agree with $A$ outside of the columns $i_0$ and $M-1$.  

\begin{figure}[htb]
\includegraphics[width=.49\textwidth]{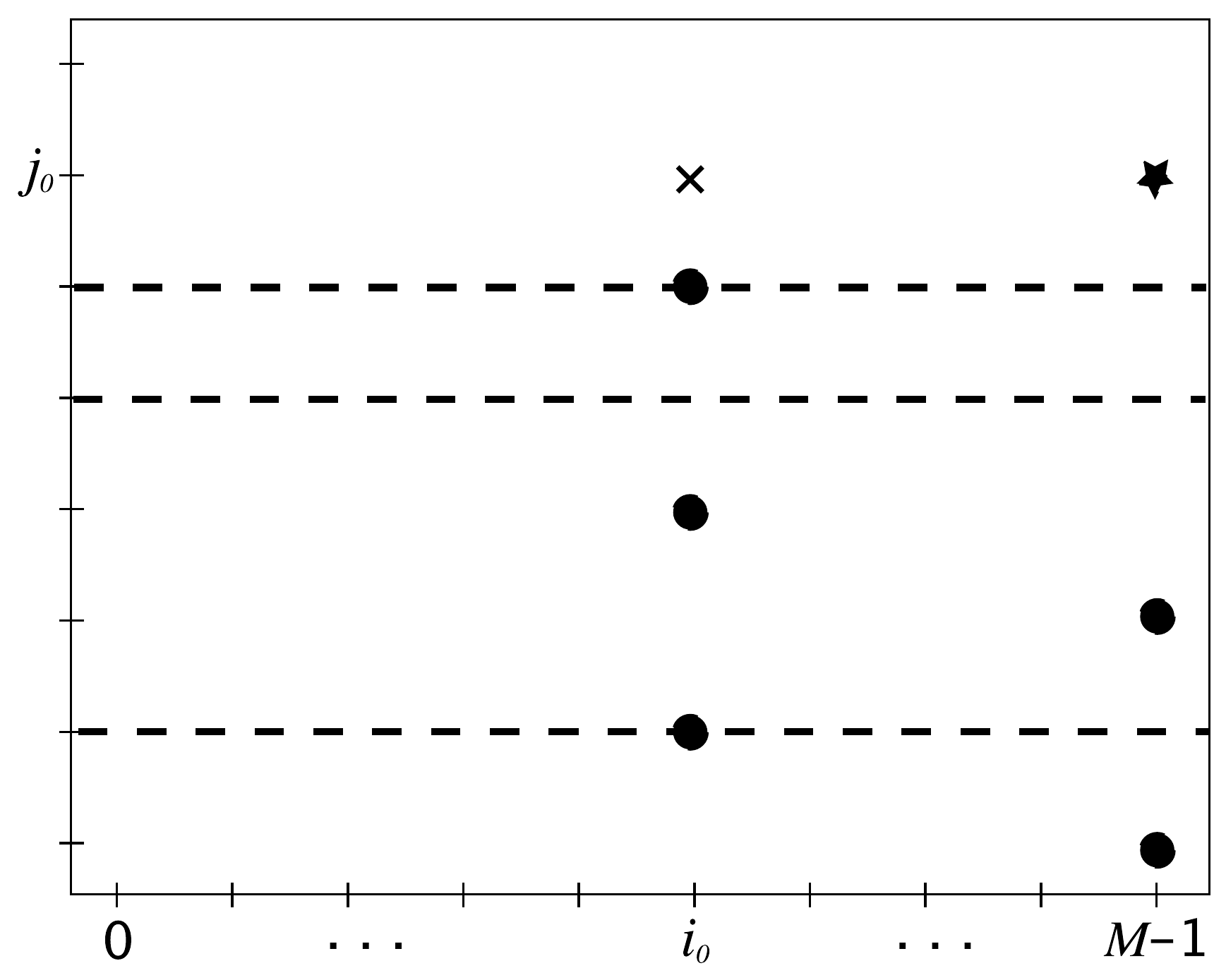}
\includegraphics[width=.49\textwidth]{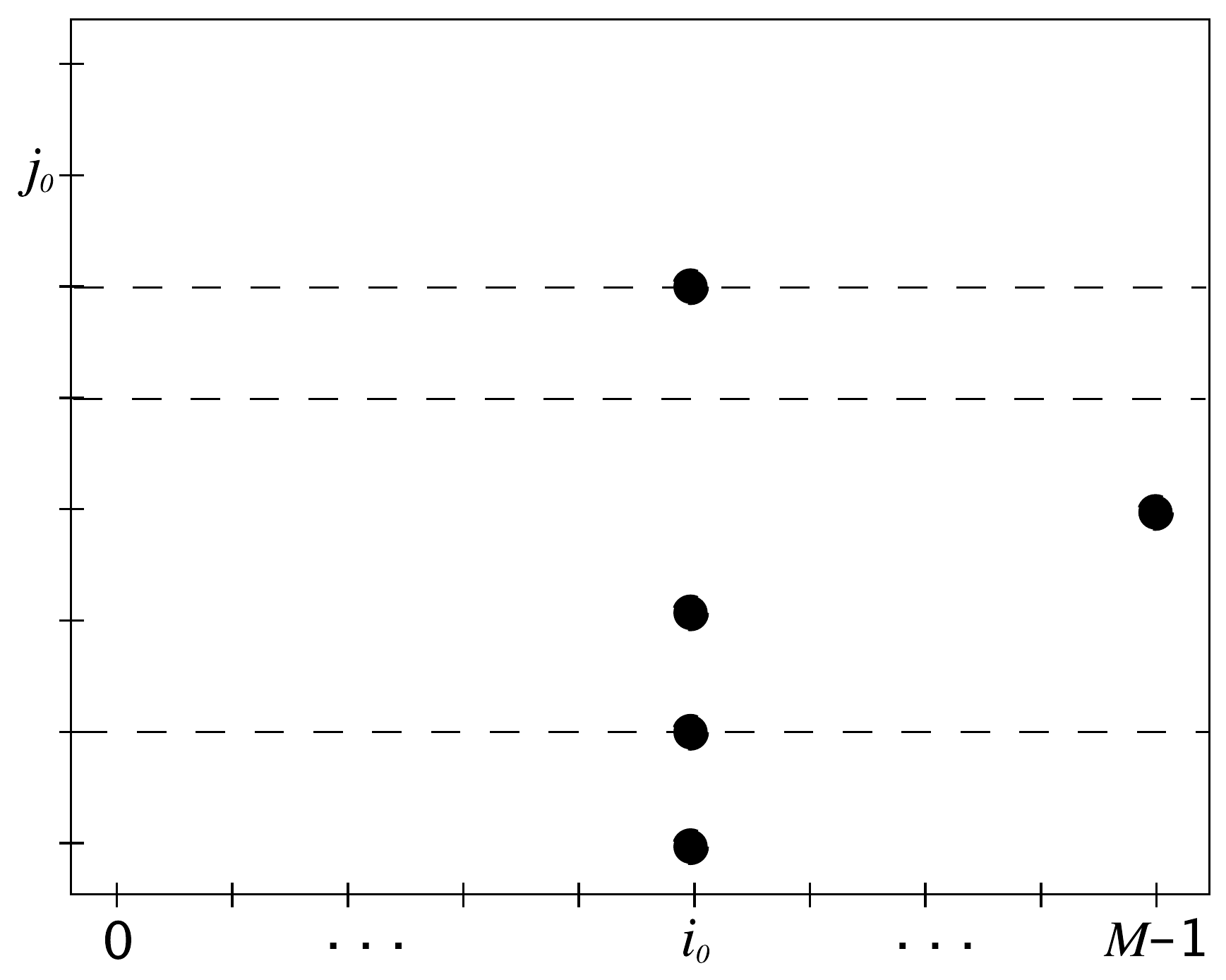}
\caption{\label{column swap fig} On the left is the configuration $\cT^{T}(A)$.  Circles represent points in $A$, and only points in columns $M-1$ and $i_0$ are shown.  In this example $k=2$.  Dashed lines are rows fully occupied by time $T-1$ (with indices in $J$).  The starred vertex becomes occupied at time $T$, while the x remains unoccupied, which is made possible by the last column having more points in $A$ off of the dashed lines.  On the right is the configuration $A'$ -- only points in columns $i_0$ and $M-1$ are shown, and the dashed lines are for reference only; the configuration off of these columns is the same as $A$.}
\end{figure}

Note that $A_1$ has strictly fewer than $k$ occupied points on the last column, $M-1$.  This is because, in the configuration $X$, the column $i_0$ has strictly fewer occupied points than the last column.  This also implies that $T\ge 2$ and $J \neq \emptyset$, since the column $i_0$ started with at least as many occupied points in $A$ as the last column.  The induction step will be completed, provided we show that $A_1$ spans.

Through time $T-1$, every point in the smaller box $R_{M-1,N}$ that becomes occupied by the dynamics from initial set $A$, also becomes occupied by the dynamics from initial set $A_1$.  That is, $$X\cap R_{M-1,N} \setminus A \subseteq \cT^{T-1}(A_1).$$   This is because first, the row occupancy counts are the same in $A_1$ and $A$, and the column occupancy counts in $R_{M-1,N}$ are larger for $A_1$ than for $A$, and second, by the definition of $T$, the points that become occupied in the last column $M-1$ do not affect either dynamics (from $A$ or $A_1$) within $R_{M-1,N}$ through time $T-1$.  Therefore, the configuration $\cT^{T-1}(A_1)$ contains all points on rows with indices in $J$ inside the box $R_{M-1,N}$.  Since $M-1\ge a$, $\cT^{T}(A_1)$ contains {\em all} points on the rows indexed by $J$.  As a result, $\cT^{T}(A_1)$ contains the configuration obtained by swapping the columns $i_0$ and $M-1$ of $A$, so $A_1$ spans.  This completes the induction step and the proof.
\end{proof}


\section{Large deviation rate: existence and bounds}

\subsection{Existence of the large deviation rate} \label{sec-ld-existence}

Throughout this section $\alpha\ge0$ and $\beta\ge0$ are fixed parameters.  We also 
fix a finite zero-set $\Z$. We remark
that the large deviation setting
makes sense for arbitrary growth transformation, not just for neighborhood growth. 
However, the key step in the proof of existence, Theorem~\ref{when-done}, is not 
available for the more general dynamics. 

We recall the setting and notation before the statement of Theorem~\ref{intro-ld-rate}. 
We will establish parts of this theorem in this and the next section. 

\begin{theorem}\label{ld-rate}
The large deviation rate $I(\alpha,\beta)=I(\alpha,\beta, \Z)$ exists. 
Moreover,  
$$
I(\alpha,\beta)=\inf\{\rho(\alpha,\beta,A): 
A\in \cA\}=\min\{\rho(\alpha,\beta,A): 
A\in \cA_0\},
$$
for a finite set $\cA_0\subseteq \cA$ that only depends on $\Z$.
\end{theorem}

First we will prove the following lemma for large deviations of the containment of specific patterns, which follows the methods for containment of small subgraphs in Erd\H os--R\'enyi random graphs, as presented in~\cite{JLR}. Throughout the rest the paper, $\omega_0$ will denote the initial 
configuration obtained by occupying every point in $R_{N,M}$ independently with 
probability $p$. 
\begin{lemma}\label{containment-lem}
For any finite pattern $A$, 
\begin{equation}\label{ld-rate-eq1}
\lim_{p\to 0}\frac{\log\probsub{\text{$\omega_0$ contains $A$}}{p}}{\log p}=\rho(\alpha,\beta,A).
\end{equation} 
\end{lemma}

\begin{proof}
For any subpattern $B\subseteq A$, the probability that $\omega_0$ contains $B$ is at most
\begin{equation}
\begin{aligned}
\probsub{\text{$\omega_0$ contains $B$}}{p} &\leq C_B \binom{N}{\pi_x(B)}
\binom{M}{\pi_y(B)} p^{\abs{B}}, \\
&\leq C_B N^{\pi_x(B)} M^{\pi_y(B)}p^{\abs{B}}\\
& = C_B p^{\abs{B} - \alpha \pi_x(B) - \beta\pi_y(B) + o(1)},
\end{aligned}
\end{equation}
where $C_B$ is a constant that accounts for the number of ways to reorder the rows and columns of $B$.  This gives the lower bound
\begin{equation}\label{ld-lowerbd}
\liminf_{p\to 0}\frac{\log\probsub{\text{$\omega_0$ contains $A$}}{p}}{\log p}\ge\rho(\alpha,\beta,A).
\end{equation}

For every subset $X\subseteq \bZ_+^2$ that is equivalent to $A$ (in the sense of a pattern) let $I_X$ be the indicator of the event that $X\subseteq \omega_0$,  and let $X\simeq A$ denote the equivalence of $X$ and $A$. Below, $X, Y, Z$ will denote subsets of $\bZ_+^2$.  Define
$$\lambda = \sum_{X \simeq A} \Esub{I_X}{p} = C_A \binom{N}{\pi_x(A)} \binom{M}{\pi_y(A)} p^{\abs{A}}.$$ 
Also, define
$$
\Lambda = \sum_{X\simeq A}\sum_{\substack{Y\simeq A\\ X\cap Y \neq \emptyset}} \Esub{I_XI_Y}{p}.
$$
  Theorem 2.18 of~\cite{JLR} states that
  $$
  \probsub{\text{$\omega_0$ does not contain $A$}}{p} \leq \exp\left[-\frac{\lambda^2}{\Lambda}\right].
  $$
Observe that
\begin{equation}
\begin{aligned}
\Lambda &= \sum_{\substack{B\subseteq A\\B\neq \emptyset}} \sum_{Z\simeq B} \sum_{X\simeq A} \sum_{\substack{Y\simeq A\\ X\cap Y = Z}} p^{2\abs{A} - \abs{B}} \\
& \le C \lambda^2 \sum_{\substack{B\subseteq A\\B\neq \emptyset}} p^{-\abs{B}} N^{-\pi_x(B)} M^{-\pi_y(B)}\\
& = C \lambda^2 \sum_{\substack{B\subseteq A\\B\neq \emptyset}} p^{-(\abs{B} - \alpha \pi_x(B) - \beta \pi_y(B)) + o(1)}\\
&\leq C \lambda^2 p^{-\rho(\alpha, \beta, A) + o(1)}.
\end{aligned}
\end{equation}
This gives the upper bound,
\begin{equation}\label{ld-upperbd}
\limsup_{p\to 0}\frac{\log\probsub{\text{$\omega_0$ contains $A$}}{p}}{\log p}\le\rho(\alpha,\beta,A).
\end{equation}

\end{proof}

\begin{proof}[Proof of Theorem~\ref{ld-rate}]
Lemma~\ref{containment-lem} directly implies that 
\begin{equation}\label{ld-rate-eq2}
\limsup_{p\to 0}\frac{\log \probsub{\Span}{p}}{\log p}\le \inf_{A\in \cA}\,\rho(\alpha,\beta,A).
\end{equation} 

Assume now that $\Span$ happens. Let $\cT'=\cT^{\tmax}$, where $\tmax$ is defined in 
Theorem~\ref{when-done}. By Theorem~\ref{growth-patterns}, 
$\cT'$ is a pattern-inclusion transformation given by a set of
patterns $\cP$. Let $\cA_0$ be the set of patterns in $\cP$ that contain no
site  in the neighborhood of the origin $\origin$. Observe that every set in $\cA_0$ spans, that is, 
$\cA_0\subseteq \cA$. Note also that $\cA_0\ne \emptyset$, which simply follows from the fact 
that there exists a finite set that spans. 

Let $G$ be the event that there exists an $x\in R_{N,M}$ whose entire neighborhood is unoccupied in $\omega_0$, that is $L^v(x)\cup L^h(x) \subseteq \omega_0^c$. Now, $\Span\subseteq \{\cT'(\omega_0)=\bZ_+^2\}$ and therefore
\begin{equation}\label{ld-rate-eq3}
\Span\cap G\subseteq \{\text{$\omega_0$ contains a member of $\cA_0$}\}.
\end{equation}

Assume without loss of generality that $M\le N$, which implies $\beta\le \alpha$. 
Assume first that $\alpha<1$. Then 
\begin{equation}\label{ld-rate-eq4}
\probsub{G^c}{p}\le (pN)^M+(pM)^N\le \exp(-p^{-\beta/2}),
\end{equation} 
for small enough $p$.  Together, (\ref{ld-rate-eq3}) and (\ref{ld-rate-eq4})
imply 
\begin{equation}\label{ld-rate-eq5}
\begin{aligned}
\probsub{\Span}{p}&\le \probsub{\text{$\omega_0$ contains a member of $\cA_0$}}{p}+
\probsub{G^c}{p}\\
&\le |\cA_0|\max_{A\in \cA_0}\probsub{\text{$\omega_0$ contains $A$}}{p}
+\exp(-p^{-\beta/2}).
\end{aligned}
\end{equation}
Now,  Lemma~\ref{containment-lem} and (\ref{ld-rate-eq5}) imply 
\begin{equation}\label{ld-rate-eq6}
\liminf_{p\to 0}\frac{\log \probsub{\Span}{p}}{\log p}\ge \min_{A\in \cA_0}\,\rho(\alpha,\beta,A).
\end{equation} 

We now consider the case $\alpha\ge 1$. For a $k\ge 1$, let
$A_k$ be the pattern 
\begin{equation*}
\begin{aligned}
&\phantom{\X\X\ldots\X\X\X\ldots\X\ldots}\X\X\ldots\X\\
&\phantom{\X\X\ldots\X\X\X\ldots\X}\ldots \\
&\phantom{\X\X\ldots\X}\X\X\ldots\X\\
&\X\X\ldots\X
\end{aligned}
\end{equation*}
The number of rows is $k$, and each interval of occupied sites has length $k$. 
For any fixed $k$ and $\epsilon>0$, 
\begin{equation}\label{ld-rate-eq7}
\probsub{\text{$\omega_0$ includes $A_k$}}{p}
\ge p^\epsilon.
\end{equation} 
Clearly, if $k$ is large enough, 
$A_k$ spans (in two time steps). Add $A_k$ to $\cA_0$. Then,  
by Lemma~\ref{containment-lem} and (\ref{ld-rate-eq7}),
\begin{equation}\label{ld-rate-eq8}
\min_{A\in \cA_0}\,\rho(\alpha,\beta,A)=0.
\end{equation}
Thus, when $\alpha\ge 1$, (\ref{ld-rate-eq8}) trivially implies (\ref{ld-rate-eq6}). 
The inequality (\ref{ld-rate-eq6}) is therefore always valid, and, together with 
(\ref{ld-rate-eq2}), gives the 
desired equalities. 
\end{proof}

\subsection{General bounds on the large deviations rate} \label{sec-ld-bounds}
Having established the existence of $I(\alpha,\beta,\Z)$, we now give three general bounds.  These will be used to establish continuity of $\IE(\alpha,\beta, \ZE)$ in Section~\ref{subsec-E-limit-proof}, and are the key components for the proof of Theorem~\ref{intro-IEbounds} in Section~\ref{ld-limit-bounds-subsec}.  Assume throughout this section that $(\alpha,\beta)\in[0,1]^2$.

\begin{prop}\label{Ilb-discrete} For any zero-set $\Z$ and nonnegative integer $k$,
\begin{equation}\label{Ilb-k}
I(\alpha,\beta,\Z) \geq 
\gamma(\zswk)\left(1-\max(\alpha,\beta)\left(1+{\frac{1}{k+1}}\right)\right).
\end{equation} 
\end{prop}
\begin{proof}
Let $A$ be a spanning set for $\Z$.  Then, by Lemma \ref{k-proj}, 
$$|\angk| - \alpha|\pi_x(\angk)| - \beta|\pi_y(\angk)| \geq |\angk|\left(1 - \max(\alpha,\beta)
\left(1+{\frac{1}{k+1}}\right)\right).$$ 
By Lemma \ref{retract}, $\angk$ spans for $\zswk$, thus $|\angk|\ge \gamma(\zswk)$. Therefore, 
$$\rho(\alpha,\beta, \angk) \geq  \gamma(\zswk)\left(1-\max(\alpha,\beta)
\left(1+{\frac{1}{k+1}}\right)\right).
$$
Moreover, $\angk$ is a subset of $A$, so
$$I(\alpha,\beta,\Z) \geq \rho(\alpha,\beta, A)\geq  \rho(\alpha, \beta ,\angk),
$$
and the desired inequality follows.
\end{proof}

\begin{prop} \label{Iupper-bound1} 
For any discrete zero-set $\Z$, 
\begin{equation}\label{I-area-ineq}
	I(\alpha,\beta,\Z) \leq ( 1-\max(\alpha,\beta) ) |\Z|. 
\end{equation}
\end{prop}

\begin{proof}
For a set $A\subseteq \bZ_+^2$ of occupied points, let $A_r\subseteq \bZ_+^2$ 
be a set such that each row in $A_r$ contains 
the same number of occupied sites as the row in $A$, but the columns of 
$A_r$ contain at most one occupied site.  Define $A_c$ analogously.  
These sets satisfy 
$$|A| = |A_r| = |A_c| = |\pi_x(A_r)| = |\pi_y(A_c)|.$$ 
For a Young diagram $\Z$ both $\Z_r$ and $\Z_c$ span: the longest row of $\Z_r$ immediately 
occupies its entire horizontal line, then the next longest does the same, and so on.  
Moreover, for any subset $B\subseteq \Z_r$, $|B| = |\pi_x(B)|$ and hence
$$\rho(\alpha,\beta,\Z_r) \leq |\Z_r|(1-\alpha).$$
Similarly
$$
\rho(\alpha,\beta,\Z_c) \leq |\Z_c|(1-\beta).
$$
The desired inequality (\ref{I-area-ineq}) follows.
\end{proof}

\begin{prop} \label{Iupper-bound2}
For any discrete zero-set $\Z$,
\begin{equation} \label{I-gamma-ineq}
	I(\alpha,\beta,\Z) \leq  2(1-\min(\alpha,\beta))\gamma(\Z).
\end{equation}
\end{prop}

\begin{proof}
Suppose the set $A$ spans for $\Z$, has 
size $|A|=\gamma(\Z)$, and $A\subseteq R_{a,b}$ for some $a$, $b$. 
Recall the definition of $A_r$ and 
$A_c$ from the previous proof. 
The key step in proving the upper bound (\ref{I-gamma-ineq}) is to show that 
the set $A_s$ defined by
$$A_s = \{(2a,0) + A_r \} \cup \{(0,2b) + A_c\}$$ 
spans for $\Z$ as well. The proof of this is similar to the proof 
of Lemma~\ref{get-two-Y}, so we only provide a brief sketch. 
Restrict the dynamics to the larger rectangle
$R_{2a,2b}$.  Then prove by induction that,
for every site $x\in R_{2a,2b} \setminus A$ and every 
$t>0$, the number of occupied sites in $\cT^t(A_s)$, in both the row and the column 
containing $x$,  will be at least as large as the number of occupied sites 
in the same row and column in $\cT^t(A).$ 
Therefore, for some $t>0$, $(a,b)+ R_{a,b}$ will be contained in $\cT^t(A_s)$.  As
$R_{a,b}$ spans, therefore so does $A_s$. 

Since $A_s$ spans, an upper bound on $\rho(\alpha,\beta, A_s)$ will also 
provide an upper bound on $I(\alpha,\beta, \Z).$  
For $B\subseteq A_s$, let $B_r = B\cap A_r$ and $B_c =B\cap A_c$. 
Then $|\pi_x(B_r) | = |B_r|$ and $|\pi_y(B_c)| = |B_c|.$  Then
\begin{align*}
|B| - \alpha |\pi_x(B)|- \beta|\pi_y(B)| 
&= |B_r|+ |B_c| - \alpha(|B_r| + |\pi_x(B_c)|) - \beta(|B_c| + |\pi_y(B_r)|)\\
&\leq |B_r|+ |B_c| - \alpha|B_r|  - \beta|B_c| \\
&\leq |B_r|+|B_c|-\min(\alpha,\beta)(|B_r|+|B_c|)\\
&= |B|(1-\min(\alpha,\beta)).
\end{align*}
Therefore $\rho(\alpha,\beta, A_s) \leq |A_s|(1-\min(\alpha,\beta))$ and  
\begin{equation*}
I(\alpha,\beta,\Z)\leq |A_s|( 1-\min(\alpha,\beta)) = 2\gamma(\Z)(1-\min(\alpha,\beta)),
\end{equation*}
as $|A_s|=2|A|=2\gamma(\Z)$. 
\end{proof}

\section{Exact results for the large deviation rate}
\subsection{Support}\label{sec-ld-support}
In this section, we conclude the proof of our main large deviations theorem; the 
most substantial remaining step is an argument for the support formula (\ref{ld-support-formula})
for a general zero-set $\Z$.

\begin{proof}[Proof of Theorem~\ref{intro-ld-rate}]
The existence of  $I$ and its variational characterization (\ref{ld-vq}) follow from 
Theorem~\ref{ld-rate}. 
Then, for every $A$, $\rho(\cdot,\cdot, A)$ is continuous and 
piecewise linear, so by (\ref{ld-vq}) the same is true for $I(\cdot,\cdot, \Z)$. Monotonicity 
in $\alpha$ and in $\beta$ follows from the definition.  

If $(\alpha,\beta)\ne (0,0)$, then $I(\alpha,\beta,\Z)<\gamma(\Z)$, since 
$\rho(\alpha,\beta,A) < \abs{A}$ whenever $A$ is nonempty.  Furthermore, 
if $\alpha+\beta<1$, then
$$
\rho(\alpha,\beta,A) =  \abs{A}-\alpha\abs{\pi_x(A)} - \beta\abs{\pi_y(A)},
$$
so $I$ is the minimum of linear functions, thus concave.  

It remains to prove the claims about the support of $I$. By continuity of 
$I(\cdot,\cdot, \Z)$, we can assume $(\alpha,\beta) \in (0,1]^2$.  
Suppose $(\alpha,\beta)$ are such that $[u(1-\alpha) -\beta] \vee [v(1-\beta) - \alpha] > 0$ for all $(u,v)\in \partial_o \cZ$, and let
$$
\epsilon = \min_{(u,v)\in\partial_o \cZ} [u(1-\alpha) -\beta] \vee [v(1-\beta) - \alpha] > 0.
$$
The event $\Span$ implies that for some $(u,v)\in \partial_o \cZ$ there exists a vertex $x \in V$ such that $\row(x,\omega_0)\ge u$ and $\col(x,\omega_0)\ge v$, and the probability of this event (for a given $(u,v)$) is bounded above by the minimum of the expected number of rows with $u$ initially occupied vertices and the expected number of columns with $v$ initially occupied vertices. Therefore,
\begin{equation}
\begin{aligned}
\probsub{\Span}{p} \le \sum_{(u,v)\in \partial_o \cZ} M(Np)^u \wedge N(Mp)^v \leq \abs{\partial_o \cZ} p^{\epsilon-o(1)},
\end{aligned}
\end{equation}
so $I(\alpha,\beta,\cZ)\ge \epsilon$, and $(\alpha,\beta) \in \supp I(\cdot,\cdot,\cZ)$.

Now suppose $(\alpha,\beta) \in (0,1]^2$ are such that there exists $(u_0,v_0)\in \partial_o \cZ$ such that $[u_0(1-\alpha) -\beta] \vee [v_0(1-\beta) - \alpha] < 0$. Let $K = \max\{ u, v : (u,v) \in \partial_o \cZ\}$, let $E$ denote the event that there are at least $K$ rows with at least $u_0$ initially occupied vertices, and let $F$ denote the event that there are at least $K$ columns with at least $v_0$ initially occupied vertices.  Observe that $E\cap F \subseteq \Span$.  We will show $\probsub{E}{p}\wedge \probsub{F}{p} \to 1$, so
$$
\probsub{\Span}{p} \ge \probsub{E\cap F}{p} \to 1,
$$
and $I(\alpha,\beta,\cZ) = 0$.

We will show $\probsub{E}{p}\to 1$, and the argument for $F$ is similar.  If $\alpha\ge 1$, then the probability that a fixed row has at least $u_0$ initially occupied vertices is at least $p^{o(1)}$, so the expected number of rows with at least $u_0$ initially occupied vertices is at least $p^{-\beta + o(1)}\to \infty$.  If $\alpha<1$ and $u_0(1-\alpha)-\beta< 0$, then the expected number of rows with at least $u_0$ initially occupied vertices is at least
$$
M\binom{N}{u_0} p^{u_0} (1-p)^N \ge M\left(\frac{Np}{3u_0}\right)^{u_0}(1-o(1)) \ge p^{u_0(1-\alpha)-\beta + o(1)} \to \infty.
$$
In either case, since rows are independent, this implies $\probsub{E}{p} \to 1$.
\end{proof}

\subsection{Large deviations for line growth}\label{sec-lgld}

\newcommand{\da}{\Delta a}
\newcommand{\db}{\Delta b}

In the next theorem, we explicitly give the large deviation rate 
for line growth with $\Z=R_{a,b}$, where $a,b\ge 0$. 
When $\alpha=\beta$ and $a=b$, the rate is given in~\cite{BBLN}
by a different method. For $\alpha,\beta\in[0,1)$, we let
\begin{equation*}
\begin{aligned}
\da=\left\lfloor \frac\beta{1-\alpha}\right\rfloor,\quad
\db=\left\lfloor \frac\alpha{1-\beta}\right\rfloor.
\end{aligned}
\end{equation*}

\begin{theorem}\label{lgld-thm}
Fix $\alpha,\beta\in [0,1)$. 
If either $b\le \db$ or $a\le \da$, then $I(\alpha,\beta,R_{a,b})=0$. 
Assume $b>\db$ and $a>\da$ for the rest of this statement. 
If $\beta\le \alpha$ and 
\begin{equation}\label{v0-ineq}
\left\lfloor\frac{\alpha}{1-\beta}\right\rfloor(1-\beta)\le \beta.
\end{equation}
holds, then 
\begin{equation}\label{I-eq1}
\begin{aligned}
I(\alpha,\beta,R_{a,b})&=
(1-\alpha)ab+((\alpha-\beta)\db-\beta)a-\beta b -(1-\beta)\da\db+\beta\da+\beta\db+\beta\\&\quad-
\max\{(1-\beta)\db,(1-\alpha)\da\}.
\end{aligned}
\end{equation}
If $\beta\le \alpha$ and (\ref{v0-ineq}) does not hold, 
\begin{equation}\label{I-eq2}
\begin{aligned}
I(\alpha,\beta,R_{a,b})&=
(1-\alpha)ab
+\alpha\db\cdot a-\beta b+\beta\db\\&\quad
+\min\{-\beta(\db+1)a-(1-\beta)\da\db+\beta\da+\beta-(1-\alpha)\da, 
-\db\cdot a\}.
\end{aligned}
\end{equation}
  If $\beta\ge \alpha$, the rate is determined 
by the equation $I(\alpha,\beta,R_{a,b})=I(\beta,\alpha,R_{b,a})$. 
\end{theorem}

Theorem~\ref{lgld-thm} implies the asymptotic result below. 
As we will see in Section~\ref{ld-limit-bounds-subsec}, (\ref{lgld-E-eq}) implies that
the line growth achieves the lower bound  (\ref{IElb}), 
thus is in this sense the most efficient neighborhood growth dynamics. 

\begin{cor} \label{lgld-E}
If $\alpha,\beta\in[0,1]$ are fixed and $\min\{a,b\}\to\infty$, 
\begin{equation}\label{lgld-E-eq}
I(\alpha,\beta, R_{a,b})\sim\gamma(R_{a,b})(1-\max\{\alpha,\beta\}).
\end{equation}
\end{cor}

\begin{proof}[Proof of Corollary~\ref{lgld-E}]
This follows from (\ref{I-eq1}) and (\ref{I-eq2}), which show that 
the difference between the 
two sides of (\ref{lgld-E-eq}) is an affine function of $a$ and $b$. 
\end{proof}

We shorten $I(a,b)=I(\alpha,\beta, R_{a,b})$ for the rest of this section. 
We begin the proof of Theorem~\ref{lgld-thm} with a recursive formula for
$I(a,b)$. 

\begin{lemma}\label{line-growth-I-recursion}
For $a, b> 0$ and $(\alpha,\beta)\in [0,1)^2$,   
$$
I(a,b) = \min\left\{\left[0\vee(-\alpha + b(1-\beta))\right] + I(a-1,b), \left[0\vee(-\beta + a(1-\alpha))\right]+I(a,b-1)\right\}.
$$
Furthermore, $I(a,0) = I(0,b) = 0$.
\end{lemma}
\begin{proof}
Let $H_a$ be the event that there is a row with at least $a$ initially occupied points, and $V_b$ be the event that there is a column with at least $b$ initially occupied points.  Also, let $\Span_{x,y}$ be the event that $\omega_0$ spans for $\cZ = R_{x,y}$.  Then,
$$
\Span_{a,b} = \left[ V_b \circ \Span_{a-1,b}\right] \cup \left[H_a \circ \Span_{a,b-1}\right],
$$
where $\circ$ denotes disjoint occurrence.
By the BK inequality and Markov's inequality,
$$
\begin{aligned}
\probsub{\Span_{a,b}}{p} &\leq \probsub{V_b}{p}\probsub{\Span_{a-1,b}}{p} +\probsub{H_a}{p} \probsub{\Span_{a,b-1}}{p}\\ 
&\leq 2 \max\left\{([N(Mp)^b]\wedge 1) \probsub{\Span_{a-1,b}}{p}, ([M(Np)^a]\wedge 1) \probsub{\Span_{a,b-1}}{p}\right\},
\end{aligned}
$$
which implies the lower bound on $I(a,b)$.  For the upper bound, observe that the density $p$ initial set $\omega_0$ dominates the union of two independent initial sets, $\omega_0^1, \omega_0^2$, each with density $p/2$.   Also, note that the probability of a fixed column being empty (and so not participating in the event $\Span_{a-1,b}$) in the initial configuration $\omega_0^2$ is at least $1-Mp/2 \ge 1/2$ for small $p$ (likewise for rows).  Furthermore, for small enough $p$
$$
\begin{aligned}
\probsub{V_b^c}{p/2} &\leq \left(1- \frac12{M\choose b} (p/2)^b\right)^N \\
&\leq \exp\left[-N(Mp/3b)^b\right] \le \begin{cases}
1 - (1/2) N(Mp/3b)^b & N(Mp/3b)^b < 1/2 \\
e^{-1/2} & N(Mp/3b)^b \ge 1/2,
\end{cases}
\end{aligned}
$$
and likewise for $H_a$.
Therefore, for small enough $p$,
\begin{align*}
&\probsub{\Span_{a,b}}{p} \ge \frac12 \max\left\{\probsub{V_b}{p/2} \probsub{\Span_{a-1,b}}{p/2}, \probsub{H_a}{p/2} \probsub{\Span_{a,b-1}}{p/2} \right\} \\
&\quad \ge  \frac14 \max\left\{([N(Mp/3b)^b]\wedge (1/2)) \probsub{\Span_{a-1,b}}{p/2}, ([M(Np/3a)^a]\wedge (1/2)) \probsub{\Span_{a,b-1}}{p/2}\right\}.
\end{align*}
This gives the upper bound on $I(a,b)$.
\end{proof}

Let 
\begin{equation*}
\begin{aligned}
&h_0=\left\lceil \left(b-\frac\alpha{1-\beta}\right)\vee 0\right\rceil
=\left(b-\db\right)\vee 0
, \\
&v_0=\left\lceil \left(a-\frac\beta{1-\alpha}\right)\vee 0\right\rceil
=\left( a-\da\right)\vee 0.
\end{aligned}
\end{equation*}
Thus, $h_0$ is the smallest number of fully occupied rows that
make the probability of spanning of a fixed column at least $p^{o(1)}$ (as $p\to 0$), 
and $v_0$ is the analogous quantity for column occupation. 

We now define a set $\cS$ of finite sequences, denoted by $\vec S=(S_1,S_2,\ldots,S_K)$. 
By convention, we let $\cS$ 
consist only of the empty sequence when either $h_0=0$ or $v_0=0$. 
Otherwise, $\cS$ consists of sequences 
$\vec S$ of length 
$K\le h_0+v_0-1$, with each coordinate $S_i\in \{H,V\}$, and the following property. 
Let $h_i=h_i(\vec S)$ and $v_i=v_i(\vec S)$ be the respective numbers of $H$s and 
$V$s in $(S_1,\ldots, S_{i-1})$;
if $S_K=H$, then $h_K=h_0-1$ and $v_K\le v_0-1$, while if 
$S_K=V$, then $h_K\le h_0-1$ and $v_K= v_0-1$. 
Every sequence represents a way to build
a spanning configuration for the line growth with $\Z=R_{a,b}$. We define
the {\it weight\/} of  $\vec S\in \cS$ as
\begin{equation}\label{weight-sequence}
w(\vec S)=\sum_{i:S_i=H} (-\beta +(1-\alpha)a-(1-\alpha)v_i)+
\sum_{i:S_i=V} (-\alpha +(1-\beta)b-(1-\beta)h_i).
\end{equation}

\begin{lemma}\label{I-w} For all $a,b\ge 0$, 
$$
I(a,b)=\min\{w(\vec S):\vec S\in \cS\}.
$$
\end{lemma}
\begin{proof}
It is clear that the statement holds if either $a=0$ or $b=0$, 
where $\cS$ consists only of the empty sequence and $I(a,b)=0$. 
It is also straightforward to check by induction that the right-hand side satisfies the 
same recursion as the one for $I(a,b)$ given in Lemma~\ref{line-growth-I-recursion}. 
\end{proof}

Next, we look at the effect of a single transposition of $H$ and $T$ to the weight of 
$\vec S$. Fix an $i\le K-2$ so that $S_i=H$, $S_{i+1}=V$, and denote $\vec S^{HV}=\vec S$. 
Let $\vec S^{VH}$ be the 
sequence obtained from $\vec S$ by transposing $H$ and $V$ at $i$ and $i+1$. Note that 
$\vec S^{VH}\in \cS$ by the restriction on $i$. The following lemma is a simple observation. 

\begin{lemma}\label{HV-VH} For any $i\le K-2$, 
$
w(\vec S^{VH})-w(\vec S^{HV})=\alpha-\beta. 
$
\end{lemma}

It is an immediate consequence of Lemma~\ref{HV-VH} that we only need to 
look for minimizers among sequences $H^{h_0-1}V^{v'}H$, $V^{v'}H^{h_0}$, 
$V^{v_0-1}H^{h'}V$, $H^{h'}V^{v_0}$, where $0\le h'\le h_0-1$ and 
$0\le v'\le v_0-1$. It is also clear from (\ref{weight-sequence}) 
that the weight is in each case a linear function 
of $v'$ or $h'$ and thus the minimum is achieved at an endpoint.
This already gives the formula for $I$ as a minimum of $8$ expressions, 
which we simplify in the proof below. 

\begin{proof}[Proof of Theorem~\ref{lgld-thm}]
We will assume $h_0\ge 1$ and $v_0\ge 1$. We will also assume that $\alpha\ge \beta$, 
as otherwise we obtain the result by exchanging $\alpha$ and $\beta$ and $a$ and $b$. 
Therefore, by  Lemma~\ref{HV-VH}, the minimizing sequence in Lemma~\ref{I-w} must 
be have one of two forms: $H^{h_0-1}V^{v'}H$
or $H^{h'}V^{v_0}$, with $0\le h'\le h_0-1$ and 
$0\le v'\le v_0-1$. We have
\begin{equation*}
\begin{aligned}
&w(H^{h_0-1}V^{v'}H)\\&=
(-\beta+(1-\alpha)a)(h_0-1)+(-\alpha+(1-\beta)(b-h_0+1))v'+(-\beta+(1-\alpha)(a-v'))\\
&=((1-\beta)(b-h_0)-\beta)v'+(-\beta+(1-\alpha)a)h_0,
\\
&w(H^{h'}V^{v_0})\\&=
(-\beta+(1-\alpha)a)h'+(-\alpha+(1-\beta)(b-h'))v_0\\
&=(-\beta+(1-\alpha)a-(1-\beta)v_0)h'+(-\alpha+(1-\beta)b)v_0.
\end{aligned}
\end{equation*}
The coefficient in front of $h'$ in $w(H^{h'}V^{v_0})$ equals 
$$
-\beta-(\alpha-\beta)a+(1-\beta)(a-v_0)\le -(\alpha-\beta)a-
\frac{\beta(\alpha-\beta)}{1-\alpha}\le 0,
$$ 
as we assumed 
$\beta\le \alpha$. Therefore, we take $h'=h_0-1$ to minimize $w(H^{h'}V^{v_0})$. 
Furthermore, 
the coefficient in front of $v'$ in $w(H^{h_0-1}V^{v'}H)$ is 
nonpositive when 
(\ref{v0-ineq}) holds, 
in which case we take $v'=v_0-1$ to minimize $w(H^{h_0-1}V^{v'}H)$;
$v'=0$ is the optimal choice when (\ref{v0-ineq}) does 
not hold. This, after some algebra, gives (\ref{I-eq1}) and (\ref{I-eq2}). 
\end{proof}

\subsection{Large deviations for bootstrap percolation}\label{sec-bpld}
As a second special case, we compute the large deviation rate for bootstrap percolation when $\alpha=\beta$.

\begin{prop}\label{bpld-prop}
Suppose $\alpha = \beta \in [0,1)$, $N = p^{-\alpha}$ and $\Ttheta$ is the Young diagram corresponding to threshold $\theta$ bootstrap percolation.   Let
$$
k = \min_{(u,v) \in \partial_o(\Ttheta)} \max\{u,v\} = \ceil{\theta/2}.
$$
If $m= \floor{\frac{1}{1-\alpha}}\le k$, then for even $\theta$,
\begin{equation}\label{bp-ld-even}
I(\alpha,\alpha, \Ttheta) = (k+m)(k-m+1) - \alpha(k+m+2)(k-m+1),
\end{equation}
and for odd $\theta$,
\begin{equation}\label{bp-ld-odd}
I(\alpha,\alpha, \Ttheta) = [(k+m-1)(k-m)+k] - \alpha\cdot [(k+m+1)(k-m)+k+1].
\end{equation}
In both cases, $I(\alpha,\alpha, \Ttheta) = 0$ for $\alpha\ge k/(k+1)$.
\end{prop}

A consequence of Proposition~\ref{bpld-prop} is that bootstrap percolation also achieves the lower bound (\ref{IElb}), at least along the diagonal $\alpha=\beta$.

\begin{cor}\label{bp-ld-limit}
As $\theta\to\infty$, for every fixed $\alpha\in [0,1]$
$$
I(\alpha,\alpha,\Ttheta) \sim \frac{\theta^2}{4}(1-\alpha) \sim \gamma(\Ttheta)(1 - \alpha).
$$
\end{cor}
\begin{proof}
For fixed $\alpha \in [0,1)$ and large enough $\theta$, $m = \floor{\frac{1}{1-\alpha}}$, so equations~\eqref{bp-ld-even} and~\eqref{bp-ld-odd} can be written
$$
I(\alpha,\alpha,\Ttheta) = \frac{\theta^2}{4}(1-\alpha) + O(\theta).
$$
The fact $\gamma(\Ttheta) \sim \theta^2/4$ is implied by sending $\alpha\to 0$ in~\eqref{bp-ld-even} and~\eqref{bp-ld-odd} and observing that $m = 1$ for small $\alpha$.  The case $\alpha = 1$ follows since $I(1,1,\Ttheta) = 0$ for all $\theta$.
\end{proof}

\begin{proof}[Proof of Proposition~\ref{bpld-prop}]
Suppose $\alpha = \beta \in (0,1)$, $N = p^{-\alpha}$ and $\Ttheta$ is the Young diagram corresponding to threshold $\theta$ bootstrap percolation. Observe that $I(\alpha, \alpha, \Ttheta) = 0$ for $\alpha\ge k/(k+1)$.

First suppose that $\theta = 2k$ and $\alpha < m/(m+1)$ where $m \in \{1, \ldots, k\}$.  Denote by $A_j$ the event that there exists a vertex, $x$, such that $\row(x,\omega_0)+\col(x,\omega_0) \geq j$, and denote by $\Span_j$ the event that $\omega_0$ spans for threshold $j$ bootstrap percolation. Then by the BK inequality
\begin{equation}\label{cross bound BK}
\begin{aligned}
\probsub{\Span_\theta}{p} \leq \probsub{A_\theta \circ \Span_{\theta-2}}{p} 
\leq \probsub{A_\theta}{p} \probsub{\Span_{\theta-2}}{p}.
\end{aligned}
\end{equation}
Iterating (\ref{cross bound BK}) gives 
\begin{equation}\label{cross bound}
\begin{aligned}
\probsub{\Span_\theta}{p}  
\leq \prod_{j=0}^{k-m} \probsub{A_{\theta-2j}}{p}
\leq \prod_{j=0}^{k-m} N^2 (2Np)^{2(k-j)} 
\leq  C \prod_{j=0}^{k-m} p^{2(k-j) - 2\alpha(k-j+1)}.
\end{aligned}
\end{equation}
Observe that in the last expression above, 
the assumption $\alpha< m/(m+1)$ guarantees that each factor is $o(1)$.  Therefore,
$$
\liminf_{p\to 0} \frac{\log\probsub{\Span_\theta}{p} }{\log p} \ge (k+m)(k-m+1) - \alpha(k+m+2)(k-m+1)
$$
whenever $0\le \frac{m-1}{m} \le \alpha < \frac{m}{m+1} \le \frac{k}{k+1}$.

Suppose now that $\theta = 2k-1$, $m\in \{1, \ldots, k\}$ and $\alpha <\frac{m}{m+1}$. Let $B_j$ denote the event that there exists a vertex $x$ such that $\row(x,\omega_0)\ge j$ or $\col(x,\omega_0) \geq j$.  Then by the BK inequality and inequality \eqref{cross bound},
\begin{equation}\label{line bound}
\begin{aligned}
\probsub{\Span_\theta}{p} &\leq \probsub{B_k \circ \Span_{2(k-1)}}{p}\\
&\leq C N^{k+1}p^k \prod_{j=1}^{k-m} N^2(N p)^{2(k-j)} \\
&= C p^{k - \alpha(k+1)} \prod_{j=1}^{k-m} p^{2(k-j) - 2\alpha(k-j+1)}.
\end{aligned}
\end{equation}
Therefore,
$$
\liminf_{p\to 0} \frac{\log\probsub{\Span_\theta}{p} }{\log p} \ge [(k+m-1)(k-m)+k] - \alpha\cdot [(k+m+1)(k-m)+k+1]
$$
whenever $0\le \frac{m-1}{m} \le \alpha < \frac{m}{m+1} \le \frac{k}{k+1}$.

Equation $(5.1)$ in~\cite{GHPS} gives the corresponding upper bounds on $\limsup_{p\to 0} \frac{\log\probsub{\Span_\theta}{p} }{\log p}$.
\end{proof}


\section{Euclidean limit of neighborhood growth}\label{sec-E-limit}

\newcommand{\rhoE}{\widetilde\rho}
\newcommand{\orho}{\overline\rho}
\newcommand{\oI}{\overline I}
\newcommand{\tI}{\widetilde{\mathcal I}}
\newcommand{\tF}{\widetilde F}
\newcommand{\tG}{\widetilde G}

The main aim of this section is the proof of Theorem~\ref{intro-d-E-theorem}, which we complete 
in Section~\ref{subsec-E-limit-proof}. As remarked 
in the Introduction, we need substantial information on the design 
of optimal spanning sets for $I(\alpha,\beta,\Z)$ when $\Z$ is large. This is given 
in Section~\ref{subsec-enhanced}, where we show that for large $\Z$, $I(\alpha,\beta,\Z)$ is well approximated by another extremal quantity that has 
a much more transparent continuum limit.  This limiting quantity 
is defined in Section~\ref{subsec-E-limit-definitions},
and the convergence is proved in Section~\ref{subsec-E-limit-enhanced}.
An analogous treatment for $\gamma_{\rm thin}$ is sketched in 
Section~\ref{subsec-E-limit-gamma-bar}.  The proof of Theorem~\ref{intro-d-E-theorem} is concluded in Section~\ref{subsec-E-limit-proof}.

\subsection{The enhancement rate}\label{subsec-enhanced}

Recall, from Section~\ref{sec-enhanced-definition}, the enhanced neighborhood growth given by 
a zero-set $\Z$ and the enhancements
$\vec f=(f_0, f_1,\ldots)$ 
and $\vec g=(g_0, g_1,\ldots)$. From now on, we assume that $\vec f$ and $\vec g$
are nondecreasing sequences with finite support. 
It will also be convenient (especially in Section~\ref{subsec-E-limit-definitions}) to represent 
$\vec f$ and $\vec g$ as Young diagrams $F$ and $G$, whereby $f_i$ is the 
$i$th row count in the digram $F$, and $g_i$ is the $i$th 
column count in the diagram $G$.

Let $\cI$ be the set of triples $(A,\vec f,\vec g)$, with $\vec f$ 
and $\vec g$ as above and $A$ a finite set that spans for $(\Z,\vec f, \vec g)$.  
We define the {\it enhancement rate\/} $\oI$ by 
$$
\oI(\alpha,\beta,\Z)=\min\{|A|+(1-\alpha)\sum \vec f+(1-\beta)\sum\vec g: 
(A,\vec f,\vec g)\in \cI\}.
$$
Observe that the elements of the above set are linear combinations  
of three nonnegative 
integers, with 
fixed nonnegative coefficients $1$, $1-\alpha$, $1-\beta$,  so its minimum 
indeed exists. 

We start with two preliminary results on $\oI$ that hold for arbitrary $\Z$. 

\begin{lemma}\label{Ibar-gamma} For any zero-set $\Z$, 
$\oI(0,0,\Z)=\gamma(\Z)$ and $\oI(\alpha,1,\Z)=\oI(1,\beta,\Z)=0$
for $\alpha,\beta\in [0,1]$. 
\end{lemma}

\begin{proof} Clearly, $\oI(0,0,\Z)\le \gamma(\Z)$, as $\gamma$ is obtained as
a minimum over a smaller set (with zero enhancements). On the other hand, 
assume that $A$ is a finite set that spans for $(\Z,\vec f,\vec g)$, 
with $\oI(0,0,\Z)=|A|+\sum\vec f+\sum \vec g$. Then we can form 
a set $A'=A\cup Y_1\cup Y_2$, such that $Y_1$ and $Y_2$ are, respectively, horizontal 
and vertical translates of corresponding Young diagrams $F$ and $G$ so that 
no horizontal line intersects both $F\cup A$ and $G$, and no vertical line 
intersects both $G\cup A$ and $F$. Using a similar argument as in the 
proof of Lemma~\ref{get-two-Y}, $A'$ spans for $\Z$ and so 
$\gamma(\Z)\le |A'|=\oI(0,0,\Z)$.

For the last claim, assume that, say, $\beta=1$ and observe that 
$\emptyset$ spans for $(\Z,\vec 0,\vec g)$  for a suitably chosen $\vec g$. 
\end{proof} 

For the rest of this subsection, we fix $\alpha,\beta\in [0,1)$ and
suppress the dependency on $\alpha$ and $\beta$ from 
the notation. 

\begin{lemma}\label{I-le-Ibar}
For any fixed $\Z$, $\alpha$ and $\beta$, 
$$
I(\Z)\le \oI(\Z).
$$
\end{lemma}
\begin{proof}
Pick $A$, $\vec f$ and $\vec g$ so that $A$ spans for $(\Z,\vec f,\vec g)$ and 
$|A|+(1-\alpha)\sum\vec f+(1-\beta)\sum \vec g=\oI(\Z)$. Create a set 
$A_0=A\cup A_h\cup A_v$ so that the union is disjoint, for every integer $v\ge 0$, 
$L^h(0,v)$ contains exactly $f_v$ sites 
of $A_h$, that every vertical line contains at most one site of $A_h$, 
and that analogous conditions hold for $A_v$. Moreover, make sure that
no horizontal line intersects both $A\cup A_h$ and 
$A_v$, and no vertical line intersects both $A\cup A_v$ and 
$A_h$. Then $A_0$ spans for $\Z$. Moreover, 
$|A_v|=\sum\vec g$, $|A_h|=\sum \vec f$. 
We now find an upper bound for $\rho(A_0)$. By dividing any subset of $A_0$ 
into three pieces, we get, with the maximum below taken over all 
sets  $B\subseteq A$, $B_h\subseteq A_h$ and  $B_v\subseteq A_v$, 
\begin{equation*}
\begin{aligned}
\rho(A_0)&=\max\{|B|+|B_h|+|B_v|-\alpha|\pi_x(B\cup B_h\cup B_v)|-\beta |\pi_y(B\cup B_h\cup B_v)|\}\\
&\le \max\{|B|+|B_h|+|B_v|-\alpha|\pi_x(B_h)|-\beta |\pi_y(B_v)|\}\\
&=\max\{|B|+|B_h|+|B_v|-\alpha|B_h|-\beta |B_v|\}\\
&=\max\{|B|+(1-\alpha)|B_h|+(1-\beta) |B_v|\}\\
&=|A|+(1-\alpha)|A_h|+(1-\beta) |A_v|.
\end{aligned}
\end{equation*}
Therefore, 
\begin{equation*}
\begin{aligned}
\oI(\Z)&=|A|+(1-\alpha)\sum\vec f+(1-\beta)\sum\vec g\\
&=|A_0|-\alpha|A_h|-\beta|A_v|\\
&\ge \rho(A_0)  \\
&\ge I(\Z),
\end{aligned}
\end{equation*}
as desired.  
\end{proof}

Finally, we show that, for large $\Z$, 
$\oI$ and $I$ are close throughout $[0,1]^2$.
The next lemma is, by far, the most substantial step in our convergence argument. 

\begin{lemma}\label{Ibar-le-I} 
Fix a bounded Euclidean zero-set $\ZE$.
Assume that $\delta>0$ and discrete zero-sets $\Z$ depend on $n$
(a dependence we suppress  from the notation), and that $\delta\squrep(\Z)\Econv\ZE$. 
Write $\ell=1/\delta$. 

Assume that positive integers $m$ and $k$ satisfy $\ell\ll m\ll \ell^2$, 
$1\ll k\ll\ell$. 
Then for some $C$ that depends on $\ZE$, $\alpha$, and $\beta$, 
$$
\oI(\Z^{\swarrow 1+2k+\lfloor C\ell^2/m\rfloor})\le I(\Z)+2m+C\frac{\ell^2}{k}.
$$
\end{lemma}

\begin{proof}
Pick a set $A$ that spans for $\Z$, and is such that 
$\rho(A)=I(\Z)$. 

\noindent{\it Step 1\/}. Let $A'=A_{>k}$. Then $A'$ spans for $\Z^{\swarrow k}$,
and there exists a constant $C$, which 
depends on $\Z$, $\alpha$ and $\beta$, so that 
$|A'|\le C\ell^2$. 

The spanning claim follows from Lemma~\ref{retract}. Moreover,  
by Lemma~\ref{k-proj} (as in the proof of Corollary~\ref{Ilb-discrete}), 
$\rho(A')\ge |A'|(1-\max\{\alpha,\beta\}\left(1+\frac 1k\right))$. As 
$\rho(A')\le \rho(A)=I(\Z)\le \gamma(\Z)$, the upper bound on $|A'|$ follows.

\noindent{\it Step 2\/}. There exists a set $\widehat A=A_d\cup A_h\cup A_v$ such that 
\begin{itemize}
\item[(1)] $A_d\subseteq A'$; 
\item[(2)] $|\widehat  A|=|A'|$;
\item[(3)] for every horizontal (resp.~vertical) line $L$, $|L\cap(A_d\cup A_h)|$
(resp.~$|L\cap(A_d\cup A_v)|$) equals $|L\cap A'|$; 
\item[(4)] $A_h$ has at most one point 
in each column and $A_v$ has at most one point 
in each row; 
\item[(5)] no horizontal line intersects both $A_d\cup A_h$ and 
$A_v$, and no vertical line intersects both $A_d\cup A_v$ and 
$A_h$; 
\item[(6)] $\widehat A$ spans for $\Z^{\swarrow k+\lfloor C\ell^2/m\rfloor}$; and 
\item[(7)] $|\pi_x(A_d)|\le m$, $|\pi_y(A_d)|\le m$.
\end{itemize}
We will inductively construct a finite sequence of 
sets $A_d^i$, $A_h^i$, $A_v^i$, $\widehat A^i=A_d^i\cup A_h^i\cup A_v^i$, 
so that, for each $i$, these sets satisfy  (1)--(5), with superscript $i$ 
on $A_d, A_h, A_v, \widehat A$, and 
\begin{itemize}
\item[(6$^i$)] $\widehat A^i$ spans for $\Z^{\swarrow k+i}$. 
\end{itemize}
We begin with $A_d^0=A'$, $A_h^0=\emptyset$, $A_v^0=\emptyset$.  

Assume we have a construction for some $i$. 
If $|\pi_x(A_d^i)|\le m$ and $|\pi_y(A_d^i)|\le m$, then the sequence is terminated. 
Otherwise, create 
a set $B\subseteq A_d^i$ by starting from $B=A_d^i$ and successively removing points 
that have both horizontal and vertical neighbors in $B$ until no such points remain.  Then no point in $B$ has both a horizontal and a vertical neighbor in $B$, and
$\pi_x(B)=\pi_x(A_d^i)$ and $\pi_y(B)=\pi_y(A_d^i)$. Divide $B$ into a disjoint union $B=B_h\cup B_v$ so that points in $B_h$ have no vertical neighbor in $B$ and 
points in $B_v$ have no horizontal neighbor in $B$. (Allocate points that 
satisfy both conditions arbitrarily.)  Let $A_d^{i+1}=A_d^i\setminus B$. Adjoin a 
horizontal translation of $B_h$ to $A_h^i$ to get $A_h^{i+1}$, and
vertical translation of $B_v$ to $A_v^i$ to get $A_v^{i+1}$, so 
that the conditions (3)--(5) are satisfied. 
For any line $L$, $|L\cap \widehat A^{i+1}|\ge |L\cap \widehat A^{i}|-1$, 
so, by the induction hypothesis, $A^{i+1}$ spans for 
$(\Z^{\swarrow k+i})^{\swarrow 1}$=$\Z^{\swarrow k+i+1}$. 

Note that $|A_d^{i}\setminus A_d^{i+1}|\ge m$, therefore 
the final $i$ satisfies $mi\le |A'|$, which, together with Step 1, 
gives (6). 

\noindent{\it Step 3\/}. For $\widehat A$ constructed in Step 2, 
$\rho(\widehat A)\le \rho(A')$. 

Let $\phi:A'\to\widehat A$ be the bijection that is identity on $A_d$, 
and an appropriate horizontal or vertical translation otherwise (corresponding to the construction of $\widehat A$ from $A'$ in Step 2). 
Pick a $B\subseteq \widehat A$ so that $|B|-\alpha|\pi_x(B)|-\beta|\pi_y(B)|=\rho(\widehat A)$. 
Let $B'=\phi^{-1}(B)$. Then 
$|\pi_x(B)|\ge|\pi_x(B')|$ because if $\phi(x)$ and $\phi(y)$
share a column, then so must $x$ and $y$ (by (4) and (5)). 
Similarly, $|\pi_y(B)|\ge|\pi_y(B')|$.  Therefore
$$
\rho(\widehat A)=|B|-\alpha|\pi_x(B)|-\beta|\pi_y(B)|\le 
 |B'|-\alpha|\pi_x(B')|-\beta|\pi_y(B')|\le \rho(A').$$

\noindent{\it Step 4\/}. Let $A_h'=(A_h)_{>k}$ and $A_v'=(A_v)_{>k}$. 
The set $A_0=A_d\cup A_h'\cup A_v'\subset \widehat A$ 
spans for $\Z^{\swarrow 2k+\lfloor C\ell^2/m\rfloor}$. 

This follows by the same argument as in the proof of Lemma~\ref{retract}. 

Define $f_v=|A_h'\cap L^h(0,v)|$ and $g_u=|A_v'\cap L^v(u,0)|$.
We may assume, by a rearrangement of rows and columns of $A_0$, that 
these are nonincreasing sequences.

\noindent{\it Step 5\/}. For so defined $\vec f$ and $\vec g$, 
$A_d$ spans for $(\Z^{\swarrow 1+2k+\lfloor C\ell^2/m\rfloor},\vec f,\vec g)$. 
Moreover, 
$$|A_d|+(1-\alpha)\sum \vec f+(1-\beta)\sum\vec g\le 
|A_0|-\alpha|\pi_x(A_0)|-\beta|\pi_y(A_0)|+2m+\frac1kC\ell^2.$$

Spanning follows from the fact that $A_h'$ has at most one point on 
any vertical line (which follows from (4)), and the analogous fact about 
$A_v'$. To show the inequality, note that $|\pi_x(A_d)|\le m$, 
$|\pi_y(A_d)|\le m$ (by (6)), 
$|\pi_y(A_v')|=|A_v'|=\sum\vec g$, $|\pi_x(A_h')|=|A_h'|=\sum \vec f$ (by (4)), 
$|\pi_x(A_v')|\le \frac1k |A_v'|$, and $|\pi_y(A_h')|\le \frac1k |A_h'|$, 
so
\begin{equation*}
\begin{aligned}
& |A_0|-\alpha|\pi_x(A_0)|-\beta|\pi_y(A_0)|\\
\ge\, & |A_d| +\sum\vec f+\sum \vec g\\&-\alpha(|\pi_x(A_d)|+|\pi_x(A_h')|+|\pi_x(A_v')|)
-\beta(|\pi_y(A_d)|+|\pi_y(A_h')|+|\pi_y(A_v')|)\\
\ge\, &|A_d| +(1-\alpha)\sum\vec f+(1-\beta)\sum \vec g\\
&-(|\pi_x(A_d)|+|\pi_y(A_d)|)-\frac1k(|A_h'|+|A_v'|)\\
\ge\, & |A_d| +(1-\alpha)\sum\vec f+(1-\beta)\sum \vec g\\
&-2m-\frac1kC\ell^2,
\end{aligned}
\end{equation*}
as $|A_h'|+|A_v'|\le |A_0|\le |A'|\le C\ell^2$. 

\noindent{\it Step 6\/}. End of the proof of Lemma~\ref{Ibar-le-I}. 
\begin{equation*}
\begin{aligned}
I(\Z)&=\rho(A)\\
&\ge \rho(A') &&\text{(as $A'\subseteq A$)}\\
&\ge \rho(\widehat A) &&\text{(by Step 2)}\\
&\ge \rho(A_0) &&\text{(as $A_0\subseteq \widehat A$)}\\
&\ge |A_0|-\alpha|\pi_x(A_0)|-\beta|\pi_y(A_0)|\\
&\ge |A_d|+(1-\alpha)\sum \vec f+(1-\beta)\sum\vec g-2m-\frac1kC\ell^2
&&\text{(by Step 5)}\\
&\ge \oI(\Z^{\swarrow 1+2k+\lfloor C\ell^2/m\rfloor})-2m-\frac1kC\ell^2&&\text{(by Step 5),}\\
\end{aligned}
\end{equation*}
as desired.
\end{proof}

\subsection{Definitions of limiting objects and their basic properies}
\label{subsec-E-limit-definitions}

We will assume throughout this section that $\ZE$ is a 
bounded Euclidean zero-set. 
Pick two left-continuous nonincreasing functions $f,g:[0,\infty)\to\bR$
with compact support. 
The {\it enhanced Euclidean neighborhood growth transformation\/} $\TE$ is determined 
by the triple $(\ZE,f,g)$ and is defined on Borel 
subsets $A$ of the plane as follows. 
For a Borel set $A\subseteq \bR_+^2$, and $x\in \bR_+^2$,
let $\hE(x,A)=\length(L^h(x)\cap A)$ and $\vE(x,A)=\length(L^v(x)\cap A)$. 
Then let
\begin{equation}\label{E-growth}
\TE(A)=A\cup\{(u,v)\in\bR_+^2: (\hE((u,v),A)+f(v),\vE((u,v),A)+g(u))\notin \ZE\}.
\end{equation}
Similar to the discrete case, the functions $f$ and $g$ 
may be represented by continuous Young diagrams $\tF$ and $\tG$, so that 
$f(v)=\length (L^h(0,v)\cap \tF)$ and $g(u)=\length (L^v(u,0)\cap \tG)$. 
Also as in discrete case, the non-enhanced transformation is given by $(\ZE,0,0)$ and 
we assume this version whenever we refer only to $\ZE$. 

Note $\TE(A)$ is also Borel for any Borel set $A$, thus $\TE$ can be iterated. 
Also, as $\ZE$ is a continuous Young diagram,  
$\TE(A)$ is well-defined even if $A$ is unbounded 
and one or both of the lengths are infinite. 
We say that a Borel set $A$ {\it E-spans\/} if 
$\TE^\infty(A)=\cup_n\TE^n(A)=\bR_+^2$, and we call $A$ {\it E-inert\/} if $\TE(A)=A$.

The connection between discrete and continuous transformations is 
give by the following simple but useful lemma, which says that 
$\TE$ is an extension of $\cT$ in the sense
that $\cT$ and $\TE$ are conjugate on 
square representations of discrete sets.

\begin{lemma} \label{conjugacy}
Assume $A\subseteq \bZ_+^2$, and assume $\cT$ is given by a discrete zero set $\Z$ and
enhancing Young diagrams $F$ and $G$. 
Let 
$\ZE=\squrep(\Z)$ be the corresponding Euclidean zero-set and 
$\tF=\squrep(F)$, $\tG=\squrep(G)$ the corresponding enhancements. 
Then 
$$
\TE(\squrep(A))=\squrep(\cT(A)).
$$
\end{lemma}

\begin{proof}
This is straightforward to check.
\end{proof}

The Euclidean counterpart 
of $\gamma$ has a straightforward definition through the
non-enhanced dynamics
\begin{equation}\label{E-gamma-def}
\gammaE(\ZE)= 
\inf\{\area(A): A\text{ is a compact subset of $\bR^2$ that E-spans for }\ZE\}.
\end{equation}
To define the counterparts of $I$ and $\gamma_{\rm thin}$, 
let $\tI$ be the set of triples $(A,f,g)$, where 
$f$ and $g$ are, as in (\ref{E-growth}), 
left-continuous nonincreasing functions
 and $A\subset \bR_+^2$ is a compact set that spans for $(\ZE,f,g)$.
Then let
\begin{equation}\label{IE-def}
\begin{aligned}
&\IE(\alpha,\beta,\ZE)=\inf\{\area(A)+(1-\alpha)\int_0^\infty f+(1-\beta)\int_0^\infty g: (A,f,g)\in \tI\}.
\end{aligned}
\end{equation}
and
\begin{equation}\label{gammaE-thin-def}
\gammaE_{\rm thin}(\ZE)=\inf \{\int_0^\infty f+\int_0^\infty g: (\emptyset,f,g)\in \tI\}.
\end{equation}

\begin{lemma}\label{scaling}
Fix an $a>0$. Then for any $\alpha,\beta\in[0,1]^2$, 
$$\IE(\alpha,\beta,a\ZE)=a^2\IE(\alpha,\beta,\ZE).$$
Moreover, 
$\gammaE(a\ZE)=a^2\gammaE(\ZE)$ and $\gammaE_{\rm thin}(a\ZE)
=a^2\gammaE_{\rm thin}(\ZE).$
\end{lemma} 

\begin{proof}
A set $A\subset \bR_+^2$ spans for $(\ZE,\tF,\tG)$ if and only if
$aA$ spans for $(a\ZE,a\tF,a\tG)$.
\end{proof}

Next are three lemmas on non-enhanced growth.

\begin{lemma} \label{inert-inclusion} Assume $\TE$ is given by a Euclidean zero-set $\ZE$. 
Suppose $A_n\subseteq \bR_+^2$ is an increasing sequence of Borel sets and $A=\cup_n A_n$. Then $\TE(A)=\cup_n\TE(A_n)$. Consequently, $\TE^\infty(A)$ is E-inert for any Borel 
set $A\subseteq \bR_+^2$.  
\end{lemma}

\begin{proof} 
Assume $x\notin \cup_n\TE(A_n)$. Then $(\hE(x,A_n),\vE(x,A_n))\in \ZE$ for all $n$. As
$\hE(x,A_n)\to \hE(x,A)$, $\vE(x,A_n)\to\vE(x,A)$ and $\ZE$ is closed, $(\hE(x,A),\vE(x,A))\in \ZE$ 
and therefore $x\notin \TE(A)$. This proves the first claim, which implies, for any Borel set $A$, 
$$
\TE(\TE^\infty(A))=\TE(\cup_n\TE^n(A))=\cup_n\TE(\TE^n(A))=\cup_n\TE^{n+1}(A)=\TE^\infty(A),
$$
as desired.
\end{proof}

\begin{lemma}\label{open-preserved}
A map $\TE$, given by a Euclidean zero-set $\ZE$, maps open sets to open sets.
\end{lemma}

\begin{proof}
Assume $A\subset \bR_+^2$ is open. To prove that $\TE(A)$ is open we may, by Lemma~\ref{inert-inclusion}, 
assume that $A$ is bounded. Pick an 
$x\in \TE(A)$. If $x\in A$, then there exists $\delta>0$ such that $B_\delta(x)\subset A \subset \TE(A)$.  Suppose now that $x\notin A$. Then $(\hE(x,A),\vE(x,A))\notin \ZE$. As $\ZE$ is closed, $(\hE(x,A)-\epsilon, \vE(x,A)-\epsilon)\notin\ZE$, for some $\epsilon>0$.
Find a compact subset $K\subseteq L^h(x)\cap A$, with 
$\length(K)>\hE(x,A)-\epsilon$. Let $\delta>0$ be the distance between $K$ and $A^c$. 
Then every point $y\in B_\delta(x)$ has a translate of $K$ on 
$L^h(y)\cap A$ (in particular, $y+K \subseteq A$) and so $\hE(y,A)>\hE(x,A)-\epsilon$. Similarly, by choosing a possibly smaller $\delta>0$, $\vE(y,A)>\vE(x,A)-\epsilon$ 
for all $y\in B_\delta(x)$. Thus, for any $y\in B_\delta(x)$, 
$(\hE(y,A),\vE(y,A))\notin \ZE$, thus $B_\delta(x)\subseteq \TE(A)$, and consequently  
$\TE(A)$ is open.
\end{proof}

\begin{lemma}\label{Z-E-spans}
Assume $\TE$ is given by a Euclidean zero-set $\ZE$ and 
$A$ is a Borel set that includes $\ZE$ in its interior. Then 
$A$ E-spans.
\end{lemma}

\begin{proof}
Let $A\varsubsetneqq \bR_+^2$ be an open set that includes $\ZE$. We claim that $A$ cannot 
be E-inert. To see this, assume that a vertical line $L$ includes a point not in $A$. 
Take the lowest closed horizontal 
line segment bounded by the vertical axis and $L$ that includes a point not in $A$, then let $x=(u,v)$ 
be the leftmost point outside $A$ on this segment. 
Clearly $(\hE(x,A),\vE(x,A))=(u,v)\notin \ZE$ and therefore 
$x\in \TE(A)$. Thus $A$ is not E-inert. 
The proof is concluded by Lemmas~\ref{inert-inclusion} and~\ref{open-preserved}.
\end{proof}

The final two lemmas of this section connect $\IE$, $\gammaE$, and $\area(\ZE)$. 

\begin{lemma}\label{IE0-gamma} For any Euclidean zero-set $\ZE$,
$\IE(0,0,\ZE)=\gammaE(\ZE).
$
\end{lemma}

\begin{proof} By definition, we may assume that 
$\ZE$ is bounded. Then the inequality $\IE(0,0,\ZE)\le \gammaE(\ZE)$
is obvious as $\gammaE$ is obtained as an infimum over a smaller set (with $f=g=0$). 
The reverse inequality can be obtained by replacing the two Young
diagram enhancements with the corresponding two initially occupied Young diagrams. 
We leave out the details, which are very similar to the proof in the discrete 
case (Lemma~\ref{Ibar-gamma}). 
\end{proof}

\begin{cor}\label{gamma-upper}
For any Euclidean zero-set $\ZE$, $\gammaE(\ZE)\le \area(\ZE)$. In 
particular, if $\area(\ZE)<\infty$, then $\IE(\alpha,\beta,\ZE)\le \gammaE(\ZE)<\infty$ 
for all $(\alpha,\beta)\in [0,1]^2$. 
\end{cor}
\begin{proof}
The first claim follows from the definition of $\gammaE(\ZE)$ and Lemma~\ref{Z-E-spans}.  The second claim follows from Lemma~\ref{IE0-gamma} and monotonicity in $\alpha$ and $\beta$.
\end{proof}

\subsection{Euclidean limit for the enhanced growth}\label{subsec-E-limit-enhanced}

In this subsection, we establish the limit for the 
enhanced rate $\oI$. 

\begin{lemma}\label{Ibar-conv} Assume $\ZE$ is a bounded Euclidean zero-set. 
Suppose that Euclidean zero-sets $\Z_n$ and $\delta_n\to 0$ are 
such that $\delta_n\squrep(\Z_n)\Econv \ZE$ as $n\to\infty$. Then 
$$\delta_n^2\oI(\Z_n)\to \IE(\ZE).$$ 
\end{lemma}
\begin{proof}
Let $\epsilon\in (0,1)$.  Define the Euclidean zero-set $\ZE_n=\delta_n\squrep (\Z_n)$. 
For large enough $n\ge N_1=N_1(\epsilon)$, by (C1),
\begin{equation}\label{E-conv}
(1-\epsilon)\ZE\subseteq \ZE_n\subseteq (1+\epsilon)\ZE.
\end{equation}
Pick a compact set $K\subseteq \bR_+^2$, and two continuous Young diagrams $\tF$ and $\tG$ so that
$K$ E-spans for $(\ZE,\tF,\tG)$ and 
with
$$\area(K)+(1-\alpha)\area(\tF)+(1-\beta)\area(\tG)<\IE(\ZE)+\epsilon.$$ 
Define $A\subseteq \bZ_+^2$ and discrete Young diagrams $F$ and $G$ by 
\begin{equation*}
\begin{aligned}
&A=\{x\in \bZ_+^2: (x+[0,1]^2)\cap (\delta_n^{-1}(1+\epsilon)K)\ne \emptyset\},\\
&F=\{x\in \bZ_+^2: (x+[0,1]^2)\cap (\delta_n^{-1}(1+\epsilon)\tF)\ne \emptyset\},\\
&G=\{x\in \bZ_+^2: (x+[0,1]^2)\cap (\delta_n^{-1}(1+\epsilon)\tG)\ne \emptyset\}.
\end{aligned}
\end{equation*}
Then 
$\delta_n\squrep(A)\supseteq (1+\epsilon) K$ 
E-spans for $((1+\epsilon)\ZE, (1+\epsilon)\tF, (1+\epsilon)\tG)$, 
thus by (\ref{E-conv}) also for  $(\ZE_n,(1+\epsilon)\tF, (1+\epsilon)\tG)$, 
and then also for $(\ZE_n,\delta_n\squrep(F), \delta_n\squrep(G))$. Therefore,
by Lemma~\ref{conjugacy}, $A$ spans 
for $(\Z_n,F,G)$, and so
\begin{equation*}
\begin{aligned}
\oI(\Z_n)&\le |A|+(1-\alpha)|F|+(1-\beta)|G|\\
&=\delta_n^{-2}\left(\area(\delta_n\squrep(A))+
(1-\alpha)\area(\delta_n\squrep(F))+
(1-\beta)\area(\delta_n\squrep(G))\right)
\\
&\le \delta_n^{-2} \left((1+\epsilon)^2(\area(K)+(1-\alpha)\area(\tF)+
(1-\beta)\area(\tG))+\epsilon\right),
\end{aligned}
\end{equation*}
if $n$ is large enough. Thus 
\begin{equation}\label{dE3}
\oI(\Z_n)\le 
\delta_n^{-2} ((1+\epsilon)^2\IE(\ZE)+5\epsilon)
\le \delta_n^{-2} (1+\epsilon)^2(\IE(\ZE)+5\epsilon). 
\end{equation}

To get an inequality in the opposite direction, assume that $n\ge N_1$ and pick a 
finite set $A\subset \bZ_+^2$  and Young diagrams $F$ and $G$, such that $A$
spans for $(\Z_n, F,G)$.
Then $\delta_n\squrep(A)$ is a compact set 
that, by Lemma~\ref{conjugacy}, spans for $(\ZE_n,\delta_n\squrep(F),\delta_n\squrep(G))$, 
and then by (\ref{E-conv}) it also spans for 
$(1-\epsilon)\ZE$. Therefore,
\begin{equation*}
\begin{aligned}
&\IE((1-\epsilon)\ZE)\\&\le
\area(\delta_n\squrep(A))+ (1-\alpha)\area(\delta_n\squrep(F)) +(1-\beta)\area(\delta_n\squrep(G))\\
&=\delta_n^2\left(|A|+(1-\alpha)|F|+(1-\beta)|G|\right)
\end{aligned}
\end{equation*}
By taking infimum over all triples $(A,F,G)$, we get 
\begin{equation}\label{dE4}
(1-\epsilon)^2\IE(\ZE)=
\IE((1-\epsilon)\ZE)\le \delta_n^2\oI(\Z_n).
\end{equation}
The two inequalities (\ref{dE3}) and (\ref{dE4}) end the proof. 
\end{proof}

\subsection{The smallest thin sets} \label{subsec-E-limit-gamma-bar}

Fix a zero-set $\Z$. To prove (\ref{dE-gamma-thin}), we need a comparison quantity, 
analogous to $\oI$. To this end, we define
$\overline\gamma_{\rm thin}(\Z)$ to be the minimum of $\sum\vec f+\sum\vec g$ 
over all sequences 
$\vec f,\vec g$ such that $\emptyset$ spans for $(\Z,\vec f,\vec g)$. We first sketch proofs of a couple 
of simple comparison lemmas. 

\begin{lemma}\label{gamma-thin-gamma} For any zero-set $\Z$,
$\gamma(\Z)\le 
\gamma_{\rm thin}(\Z)\le 2\gamma(\Z)$.
\end{lemma}

\begin{proof} The lower bound is clear as $\gamma_{\rm thin}$ is the 
minimum over a smaller set than $\gamma$.  
The upper bound is a simple construction (similar to the one in the proof of
Lemma~\ref{get-two-Y}): one may replace any spanning set $A$ by a thin spanning 
set consisting of two pieces, one with the row counts the same as those of $A$, and the other with the column counts the same as those of $A$.
\end{proof}

\begin{lemma}\label{gamma-bar-gamma} For any zero-set $\Z$,
$\overline\gamma_{\rm thin}(\Z^{\swarrow 1})\le 
\gamma_{\rm thin}(\Z)\le \overline\gamma_{\rm thin}(\Z)$.
\end{lemma}

\begin{proof} This is again a simple construction argument as in Lemma~\ref{get-two-Y}.
If $\emptyset$ spans for $(\Z,\vec f,\vec g)$, 
then the thin set $A$ constructed by populating row $i$ with $f_i$ occupied points and column $\sum_i f_i + 1+j$ with $g_j$ occupied points has 
\begin{equation}\label{A-seq}
|A|=\sum_i f_i+\sum_j g_j,
\end{equation} 
and spans for $\Z$. Conversely, if a thin set $A$ spans for $\Z$, 
then the row and column counts of $A$ can be gathered into $\vec f$ and $\vec g$ (once sorted), so that (\ref{A-seq}) holds and 
$\emptyset$ spans for $(\Z^{\swarrow 1}, \vec f,\vec g)$.
\end{proof}

Recall the definition of  
$\gammaE_{\rm thin}$ from Section~\ref{subsec-E-limit-definitions}. 
We will omit the proof of the following convergence result, which 
can be obtained by adapting the 
argument for enhancement rates.  
\begin{lemma}\label{dE-thin} Assume $\ZE$ is a bounded Euclidean zero-set. 
Then $\gammaE_{\rm thin}(\ZE)\le \area(\ZE)$. Moreover,  
suppose discrete zero-sets $\Z_n$ and $\delta_n>0$ satisfy
$\delta_n\to 0$ and 
$\delta_n\squrep(\Z_n)\Econv \ZE$. Then 
$
\delta_n^{2}\overline\gamma_{\rm thin}(\Z_n)\to \gammaE_{\rm thin}(\ZE)$.
\end{lemma}

\subsection{Proof of the main convergence theorem}\label{subsec-E-limit-proof}

We begin with an extension of Theorem~\ref{gamma-perturbation} that is needed
to reduce our argument to bounded Euclidean zero-sets.

\begin{lemma}\label{I-perturbation}
Let $\Z$ be any zero-set, $(\alpha,\beta)\in [0,1]^2$, and $R>0$ an integer. Then  
$$
I(\alpha,\beta,\Z\cap [0,R]^2)\le I(\alpha,\beta,\Z)\le 
I(\alpha,\beta,\Z\cap [0,R]^2)+|\Z\setminus [0,R]^2| .
$$
\end{lemma}
\begin{proof} 
Pick a set $A$ that spans for $\Z\cap [0,R]^2$, such that 
$\rho(A)=I(\alpha,\beta,\Z\cap [0,R]^2)$. By Theorem~\ref{gamma-perturbation}, 
there exists a set $A_1$ with $|A_1|\le |\Z\setminus [0,R]^2|$, such that 
$A\cup A_1$ spans for $\Z$. Therefore, with supremum below over all sets 
$B\subseteq A$ and $B_1\subseteq A_1$, 
\begin{equation*}
\begin{aligned}
I(\alpha,\beta,\Z)&\le \rho(A\cup A_1)\\
&=\sup_{B,B_1} |B\cup B_1|-\alpha |\pi_x(B\cup B_1)|
-\beta |\pi_y(B\cup B_1)|\\
&\le 
\sup_{B} |B|+|A_1|-\alpha |\pi_x(B)|
-\beta |\pi_y(B)|\\
&=\rho(A)+|A_1|\\
&\le I(\alpha,\beta,\Z\cap [0,R]^2)+|\Z\setminus [0,R]^2|,
\end{aligned}
\end{equation*}
as desired. 
\end{proof} 

Recall the definition of E-convergence from Section~\ref{sec-intro}. 
We omit the routine proof of the following lemma. 

\begin{lemma}
\label{C2-uniformity} 
Assume that (C1) holds, $\area(\ZE)<\infty$, and $\area(\ZE_n)<\infty$ for all 
$n$. Then (C2) is equivalent to 
$$
\lim_{R\to\infty} \area(\ZE_n\setminus [0,R]^2)= 0
$$
uniformly in $n$. 
\end{lemma}

We are now ready to prove our main convergence result, Theorem~\ref{intro-d-E-theorem}. 
Before we proceed, we need to extend the definitions 
of $\IE$ and $\gammaE_{\rm thin}$ to unbounded Euclidean zero-sets. For 
an arbitrary $\ZE$, we define 
\begin{equation}\label{IE-def-unbounded}
\IE(\alpha,\beta,\ZE)=\lim_{R\to\infty} \IE(\alpha,\beta,\ZE\cap[0,R]^2)
\end{equation}
and 
\begin{equation}\label{gammaEthin-def-unbounded}
\gammaE_{\rm thin}(\ZE)=\lim_{R\to\infty} \gammaE_{\rm thin}(\ZE\cap[0,R]^2).
\end{equation}
Observe that, if $\area(\ZE)<\infty$, $\IE(\ZE)\le \gammaE(\ZE)\le \area(\ZE)<\infty$, 
and likewise $\gammaE_{\rm thin}(\ZE)<\infty$. 


\begin{lemma} \label{dE-IE-partial}
Assume $\ZE$ is an arbitrary Euclidean zero-set. 
Suppose that discrete zero-sets $\Z_n$ and $\delta_n\to 0$ are 
such that $\delta_n\squrep(\Z_n)\Econv \ZE$ as $n\to\infty$. Then 
$\delta_n^2I(\alpha,\beta,\Z_n)\to \IE(\alpha,\beta,\ZE).$
If $\area(\ZE)<\infty$ this convergence is uniform for $(\alpha,\beta)\in [0,1]^2$
and the limit is concave and continuous on $[0,1]^2$. 
If $\area(\ZE)=\infty$, the limit is infinite on $[0,1)^2$. 
\end{lemma}

\begin{proof}
We first prove (\ref{dE-I}) for fixed $(\alpha,\beta)\in [0,1)^2$, which 
we suppress from the notation.
If $\area(\ZE)=\infty$, then 
$\delta_n^{2}I(\Z_n)\to \infty$ by Lemma~\ref{conjugacy}, Proposition~\ref{Ilb-discrete}, (C2) and Theorem~\ref{intro-gamma-area-thm}. 
We assume $\area(\ZE)<\infty$ for the remainder of the proof. 

Fix an $\epsilon\in(0,1)$. By definition, we can choose $R$ large enough 
so that 
\begin{equation}
\label{dE0}
\IE(\ZE\cap [0,R]^2)>\IE(\ZE)-\epsilon.
\end{equation} 
It follows by  Lemma~\ref{C2-uniformity} that, if $R$ is large enough,
$\delta_n^2 |\Z_n\setminus [0,\delta_n^{-1}R]^2|<\epsilon$, for all $n$.  
Then, by Lemma~\ref{I-perturbation},
\begin{equation}
\label{dE1}
I(\Z_n\cap [0,\delta_n^{-1}R]^2)\le I(\Z_n)\le 
I(\Z_n\cap [0,\delta_n^{-1}R]^2)+\epsilon\delta_n^{-2}, 
\end{equation}
for every $n$.

For every $R>0$, $\delta_n\squrep(\Z_n\cap [0,\delta_n^{-1}R]^2)\Econv \ZE\cap[0,R]^2$, and therefore, by 
Lemma~\ref{Ibar-conv}, 
$$
\delta_n^2 \oI(\Z_n\cap [0,\delta_n^{-1}R]^2)\to \IE(\ZE\cap [0,R]^2),
$$
and then, by Lemmas~\ref{Ibar-le-I} and~\ref{I-le-Ibar}, 
$$
\delta_n^2 I(\Z_n\cap [0,\delta_n^{-1}R]^2)\to \IE(\ZE\cap [0,R]^2).
$$
By (\ref{dE0}) and (\ref{dE1}), it follows that
\begin{equation*}
\begin{aligned}
\IE(\ZE) -\epsilon\le \IE(\ZE\cap [0,R]^2) &\le \liminf \delta_n^2 I(\Z_n)\\
&\le \limsup \delta_n^2 I(\Z_n)\le  \IE(\ZE\cap [0,R]^2) +\epsilon\le \IE(\ZE) +\epsilon,\\
\end{aligned}
\end{equation*}
which ends the proof of the convergence claim. 


By Proposition~\ref{Iupper-bound1} and the established convergence,
\begin{equation}\label{dE2}
\IE(\alpha,\beta,\ZE)\le (1-\max\{\alpha,\beta\}) \area(\ZE),
\end{equation}
for any $(\alpha,\beta)\in[0,1]^2$ and any Euclidean zero-set $\ZE$ with finite area.

If $\ZE$ is bounded, the function $\IE(\cdot,\cdot,\ZE)$ is concave on $[0,1]^2$ because it is an 
infimum of linear functions.  By passing to the limit (\ref{IE-def-unbounded}), 
this holds for arbitrary $\ZE$. Clearly, $\IE(\alpha,\beta,\ZE)$ is
nonincreasing in $\alpha$ and $\beta$, so by concavity and (\ref{dE2}), $\IE(\cdot,\cdot,\ZE)$ is continuous on $[0,1]^2$.
The functions $\delta_n^{2}I(\cdot,\cdot,\Z_n)$ are also nonincreasing in 
each argument for every $n$, so pointwise convergence implies uniform convergence. 
\end{proof}

\begin{cor}\label{gamma-E-perturbation} For any 
Euclidean zero-set $\ZE$ with $\area(\ZE)<\infty$, any $(\alpha,\beta)\in[0,1]^2$, and any $R>0$,
$$
\IE(\alpha,\beta,\ZE)\le 
\IE(\alpha,\beta,\ZE\cap[0,R]^2)
+\area(\ZE\setminus[0,R]^2).
$$
\end{cor}

\begin{proof}
Define $\Z_n$ to be the inclusion-maximal subset of $\bZ_+^2$ such that 
$\frac 1n\squrep(\Z_n)\subseteq \ZE$. Then 
$\frac 1n\squrep(\Z_n)\Econv \ZE$, 
$(\frac 1n\squrep(\Z_n))\cap[0,R]^2\Econv \ZE\cap [0,R]^2$
and 
\begin{equation}\label{gEp1}
{\textstyle \frac 1{n^2}}|\Z_n\setminus[0,nR]^2|=
\area(({\textstyle\frac1n}\squrep(\Z_n))\setminus[0,R]^2)+\cO({\textstyle\frac 1n})\to 
\area(\ZE\setminus[0,R]^2).
\end{equation}
By Lemma~\ref{I-perturbation}, we have 
\begin{equation}\label{gEp2}
I(\alpha,\beta,\Z_n)\le I(\alpha,\beta,\Z_n\cap[0,nR]^2)+|\Z_n\setminus [0,nR]^2|.
\end{equation}
Upon dividing (\ref{gEp2}) by $n^2$ and sending $n\to\infty$, 
Lemma~\ref{dE-IE-partial} and (\ref{gEp1}) give the desired inequality. 
\end{proof}

\begin{cor}\label{E-area-bound} For any 
Euclidean zero-set $\ZE$,
$\gammaE(\ZE)\ge \frac 14 \area(\ZE)$. 
\end{cor}

\begin{proof}
If $\area(\ZE)<\infty$ then the argument is similar to the one in the preceding 
corollary. If $\area(\ZE)=\infty$, then for any $R>0$, 
$\gammaE(\ZE)\ge \gammaE(\ZE\cap [0,R]^2) \ge \frac 14 \area(\ZE\cap[0,R]^2)$, 
and so $\gammaE(\ZE)=\infty$. 
\end{proof}

\begin{cor}\label{gamma-E-continuous} Assume $\area(\ZE)<\infty$. 
If $\ZE_n\Econv \ZE$, then $\IE(\cdot,\cdot,\ZE_n)\to \IE(\cdot,\cdot,\ZE)$, 
uniformly on $[0,1]^2$. 
\end{cor}

\begin{proof}
If $\area(\ZE)<\infty$ we may assume all areas are finite. By Lemma~\ref{C2-uniformity} and 
Corolllary~\ref{gamma-E-perturbation}, we may also assume that 
all $\ZE_n$ and $\ZE$ are subsets of $[0,R]^2$, for some $R$. 
In this case, for any $\epsilon>0$, 
$(1-\epsilon)\ZE\subseteq \ZE_n\subseteq (1+\epsilon)\ZE,
$
when $n$ is large enough. Thus, by Lemma~\ref{scaling},  
$(1-\epsilon)^2\IE(\ZE)\le \IE( \ZE_n)\le(1+\epsilon)^2\IE(\ZE),
$
which clearly suffices. 
\end{proof}

\begin{proof}[Proof of Theorem~\ref{intro-d-E-theorem}]
All statements on large deviation rates follow from Lemma~\ref{dE-IE-partial}
and Corollary~\ref{gamma-E-continuous}, and imply (\ref{dE-gamma}). 
We omit the similar proof of (\ref{dE-gamma-thin}). 
\end{proof}

\section{Bounds on large deviations rates for large zero-sets}

In Sections~\ref{ld-limit-bounds-subsec}--\ref{ld-Lshape-11-subsec} we 
address bounds on $\IE(\alpha,\beta,\ZE)$.  In Section~\ref{ld-limit-bounds-subsec}, 
we complete the proof of Theorem~\ref{intro-IEbounds}.   
In  Sections~\ref{ld-10-subsec},~\ref{ld-Lshape-00-subsec} and~\ref{ld-Lshape-11-subsec}, 
we prove lower bounds on $\IE$ near the corners of $[0,1]^2$, either for general Euclidean 
zero-sets or an L-shaped Euclidean zero-set, 
which establish Theorem~\ref{intro-IEcorners} and show that each of the three
upper bounds on $\IE(\alpha,\beta,\ZE)$ is, in a sense, impossible to 
improve near one of the corners.

\subsection{General bounds on $\IE$}\label{ld-limit-bounds-subsec}

We assume that 
$(\alpha,\beta)\in [0,1]^2$.  Having established the existence of $\IE$, we now recall the three propositions in Section~\ref{sec-ld-bounds} and complete the proof of Theorem~\ref{intro-IEbounds}.

%

\begin{proof}[Proof of Theorem~\ref{intro-IEbounds}]
Pick a sequence of zero-sets $\Z_n$, such that 
$\delta_n \squrep(\Z_n)\Econv \ZE$ for some sequence of positive numbers $\delta_n\to 0$.
To prove the lower bound, we use the Proposition~\ref{Ilb-discrete}
with any numbers $k=k_n$ that satisfy $1\ll k\ll 1/\delta_n$, so that also
$\delta_n \squrep(\Z_n^{\swarrow k})\Econv \ZE$. 
To prove the upper bound (\ref{IEub}), we use the inequalities (\ref{I-area-ineq}), 
(\ref{I-gamma-ineq}), and the inequality $I(\alpha,\beta,\Z_n)\le \gamma(\Z_n)$ 
(see Theorem~\ref{intro-ld-rate}). We multiply these four
inequalities by $\delta_n^2$, take the limit as $n\to\infty$, and 
use $\delta_n^2|\Z_n|\to\area(\ZE)$ (by definition of E-convergence)
and Theorem~\ref{intro-d-E-theorem} to obtain~\eqref{IElb} and
(\ref{IEub}). 
\end{proof}

A continuous version of Theorem~\ref{lgld-thm} follows.

\begin{cor}\label{E-gamma-lg}
For any Euclidean rectangle $\RE_{a,b}$,
	$$\IE(\alpha,\beta,\RE_{a,b})=(1-\max(\alpha,\beta))ab.$$
\end{cor}
 
\begin{proof}
 	It follows from Theorems~\ref{lp-gamma} and~\ref{intro-d-E-theorem} that 
$\gammaE(\RE_{a,b}) = \area(\RE_{a,b}) = ab,$ so the upper and lower bounds on $\IE(\alpha,\beta,\RE_{a,b})$ given in Theorem~\ref{intro-IEbounds} agree.
(Alternatively, one may use Corollary~\ref{lgld-E}.)
 \end{proof}

\subsection{The $(1,0)$ corner} \label{ld-10-subsec}

\begin{theorem} 
\label{one-zero-edge}
Fix a continuous zero-set $\ZE$ with finite area. Then
\begin{equation}\label{one-zero-edge-eq}
\liminf_{\alpha\to 1-}\frac 1{1-\alpha}\IE(\alpha,0,\ZE)\ge \area(\ZE).
\end{equation}
\end{theorem}

A consequence of this theorem is a characterization of Euclidean zero-sets 
which attain the lower bound (\ref{IElb}).

\begin{cor} 
\label{one-zero-edge-cor} 
Assume $\ZE$ is a Euclidean zero-set with $\area(\ZE)<\infty$. 
Then $\IE(\alpha,\beta,\ZE)=(1-\max\{\alpha,\beta\})\gammaE(\ZE)$  for all 
$(\alpha,\beta)\in [0,1]^2$ if and only if 
$\gammaE(\ZE)=\area(\ZE)$, which in turn holds if and only if $\ZE=\RE_{a,b}$ for some $a,b\ge0$. 
\end{cor}

\begin{proof}
By Corollary~\ref{E-gamma-lg} and Theorem~\ref{one-zero-edge}, we only need to show that the second statement implies the 
third.
Suppose there do not exist $a,b\ge0$ such that $\ZE = \RE_{a,b}$.  Since $0<\area(\ZE)<\infty$, we may choose $a,b>0$ such that for some $\epsilon>0$ the boundary of $\ZE$ 
intersects $\RE_{a,b}$ in intervals of length at least $\epsilon>0$ and such that $(a-\epsilon,b-\epsilon) + [0,\epsilon]^2 \subset \RE_{a,b}\setminus \ZE$.   If $\TE'$ is the growth transformation for the dynamics given by $\ZE\cap \RE_{a,b}$, then it follows that $\TE'((\ZE\cap \RE_{a,b}) \setminus [0,\epsilon]^2)\supseteq \ZE\cap \RE_{a,b}$, so $\gammaE(\ZE\cap \RE_{a,b}) \leq \area(\ZE\cap \RE_{a,b}) - \epsilon^2$. By Corollary~\ref{gamma-E-perturbation}, 
$$
\gammaE(\ZE) \le \gammaE(\ZE\cap \RE_{a,b}) + \area(\ZE\setminus \RE_{a,b}) \leq \area(\ZE) - \epsilon^2,
$$
which ends the proof.
\end{proof}

\begin{proof}[Proof of Theorem~\ref{one-zero-edge}]
We first argue that it is enough to prove (\ref{one-zero-edge-eq}) 
when $\ZE$ is bounded. Indeed, once we achieve that, the $\liminf$ 
in (\ref{one-zero-edge-eq}) is, for any $\ZE$ and any $R>0$, at least $\area(\ZE\cap[0,R]^2)$.
The general result then follows by sending $R\to\infty$.
We assume that $\ZE$ is bounded for the rest of the proof.

We fix an $\alpha\in [0,1)$. We also fix $\epsilon,\delta>0$, to be 
chosen to depend on $\alpha$ (and go to $0$ as $\alpha\to 1$) later.  
We assume the discrete zero-sets $\Z$ are large, depend on $n$, and $\frac1n\squrep(\Z)\Econv\ZE$, but for readability 
we will drop the dependence on $n$ from the notation. 

In addition, we fix an integer $k\ge 2$ that will also depend on $\alpha$ and increase to infinity
as $\alpha\to 1$. We say that a zero-set $\Z$ satisfies {\it the slope condition\/}
if  there is no contiguous horizontal or vertical interval 
of $k$ sites in $\partial_o\Z$.
Let $a_0$ and $b_0$ be the
longest row and column lengths of $\Z$.  

We claim that for any $\Z$ there exists 
a zero-set $\Z'\supseteq \Z^{\swarrow \lfloor a_0/k\rfloor + \lfloor b_0/k\rfloor}$
that satisfies the slope condition. To see why this holds, assume there is a leftmost 
horizontal interval 
of $k$ sites in $\partial_o\Z$, ending at site $(u_0,v_0)$. Replace $\Z$ by the 
zero set obtained by moving down points on the line $R^v(u_0,v_0)$ and to its right, that is, by
$$\{(u,v)\in \bZ_+^2: (u<u_0\text{ and }(u,v)\in \Z) 
\text{ or } (u\ge u_0\text{ and }(u,v+1)\in \Z) \}.
$$
Observe that, first, the resulting set includes $\Z^{\downarrow 1}$;
second, if $\partial_o\Z$ does not have a contiguous vertical interval 
of $k$ sites, this operation does not produce one; and, third, after at most
$\lfloor a_0/k\rfloor$ iterations we obtain a zero-set whose boundary 
has no contiguous horizontal interval of $k$ sites. Thus we can produce 
a zero-set that satisfies the slope condition after at most $\lfloor a_0/k\rfloor$ 
steps for horizontal intervals, followed by at most $\lfloor b_0/k\rfloor$ steps for vertical ones, 
which proves the claim. The resulting $\Z'$ satisfies
\begin{equation}\label{edge10-eq0}
|\Z'|\ge|\Z|-|\Z^{\myll \lfloor a_0/k\rfloor + \lfloor b_0/k\rfloor}|\ge |\Z|-\frac 1k(a_0+b_0)^2.
\end{equation}

Assume that $A$ spans for $\Z$, therefore also for $\Z'$, 
and that $|A|\le |\Z'|$. 
If $|\pi_x(A)|\le (1-\delta)|A|$, then 
\begin{equation}
\label{edge10-eq1} 
\rho(\alpha,0,A)\ge \delta|A|\ge \delta\gamma(\Z)\ge \frac 14\delta|\Z|.
\end{equation} 
We now concentrate on the case when 
$|\pi_x(A)|\ge (1-\delta)|A|$. Define the {\it narrow region\/} of $\bZ_+^2$ to be the union of vertical lines that contain exactly one point of $A$,  and the {\it wide region\/} to be the union of vertical lines that contain at least two points of $A$.  Let $A_{\rm narrow}$ be the subset of $A$ that lies in the narrow region, and $A_{\rm wide}$ be the remaining points of $A$.  We claim that $|A_{\rm wide}|\le 2\delta |A|$. To see this, observe that
$$2|\pi_x(A_{\rm wide})|+|\pi_x(A_{\rm narrow})|\le |A|,$$
so 
$$|\pi_x(A_{\rm wide})|\le |A|-|\pi_x(A)|\le \delta |A|$$
and then 
$$
|A_{\rm wide}|=|A|-|A_{\rm narrow}|
=|A|-|\pi_x(A_{\rm narrow})|=|A|-|\pi_x(A)|+|\pi_x(A_{\rm wide})|\le 2\delta|A|.
$$

We will successively paint whole lines of $\bZ_+^2$, including points in $A$, 
red and blue, transforming the
zero-set $\Z'$ in the process. The resulting (finitely many) zero-sets $\Z_i'$, $i=0,1,\ldots$, will  
satisfy the slope condition, and will span with initial set $A$ from which the
points painted by that time have been removed. 
The painted points will 
dominate the set of points that become occupied in a slowed-down version of 
neighborhood growth with zero-set $\Z'$.
Initially, no point is painted and we let $\Z_0'=\Z'$, with $a_0'$ and $b_0'$ its largest row and 
column counts.

Assume that $i\ge 0$ and we have a zero-set $\Z_i'$, with $a_i'$ its largest row 
count. If $a_i'<\epsilon a_0'$, the procedure stops with this final $i$. 
Otherwise, choose an unpainted point $x\notin A$ that gets occupied by 
the growth given by $\Z_i'$, applied to $A$ without the painted points. 
The first possibility is that at least $(1-\epsilon)a_i'$ unpainted
points of $A$ are on $L^h(x)$.  Then paint blue all points on $L^h(x)$ that have 
not yet been painted, and let $\Z_{i+1}'=\Z_i'^{\downarrow 1}$. The second possibility 
is that fewer than $(1-\epsilon)a_i'$ unpainted
points of $A$ are on $L^h(x)$. Then $x$ is in the wide region and 
there must be at least $\frac12 \epsilon a_i'/k\ge \frac12 \epsilon^2 a_0'/k$ 
points of $A$ on $L^v(x)$, due to the slope condition. Paint all 
unpainted points in the entire neighborhood of $x$ red, and let $\Z_{i+1}'=\Z_i'^{\swarrow 1}$.

If $\ell$ is the number of times the red points are added, then 
$$
\ell\le 4k\epsilon^{-2}\delta |A|/a_0'\le 4k\epsilon^{-2}\delta |\Z'|/a_0'\le
4k\epsilon^{-2}\delta b_0'. 
$$ 
Observe that $|\Z'^{\myll \ell}|\le \ell(a_0'+b_0')$. Moreover,
the number of points in $\Z'$ in rows of length at most $\epsilon a_0'$ is 
at most $k(\epsilon a'_0)^2$, by the slope condition.  
Therefore, the points of $A$ colored blue at the final step have cardinality at least
$$
(1-\epsilon)|\Z'|-k(\epsilon a'_0)^2-\ell(a_0'+b_0').
$$ 
Choose $\delta=\epsilon^3$ to get  
\begin{equation}
\label{edge10-eq2} 
|A|\ge (1-\epsilon) |\Z'|-4k\epsilon (a_0'+b_0')^2.
\end{equation}
Clearly, (\ref{edge10-eq2}) holds if $|A|\ge |\Z'|$ as well. 
Therefore, (\ref{edge10-eq0}) and (\ref{edge10-eq2}) imply
\begin{equation}
\label{edge10-eq21} 
|A|\ge (1-\epsilon) |\Z|-4k\epsilon (a_0+b_0)^2-\frac 1k(a_0+b_0)^2.
\end{equation}
We now choose $k=1/\sqrt\epsilon$. Moreover, we observe that there exists a constant $C>1$ that depends on the limiting shape $\ZE$ such that $(a_0+b_0)^2\le C|\Z|$ for all sufficiently large $n$. (It is here we use 
the assumption that $\ZE$ is bounded, so $a_0/n$ and $b_0/n$ converge.) Therefore, when $|\pi_x(A)|\ge (1-\delta)|A|$, 
(\ref{edge10-eq21}) implies
\begin{equation}
\label{edge10-eq22} 
\rho(\alpha,0,A)\ge (1-6C\sqrt\epsilon)(1-\alpha)\abs{\Z}.
\end{equation}
Then (\ref{edge10-eq1}) and (\ref{edge10-eq22}) together imply
\begin{equation}
\label{edge10-eq3} 
\liminf_n I(\alpha,0,\Z)/|\Z|\ge \min\{(1-6C\sqrt\epsilon)(1-\alpha), \frac 14\epsilon^3\}.
\end{equation}
Finally, we pick $\epsilon=2(1-\alpha)^{1/3}$ to get from (\ref{edge10-eq3}) that 
\begin{equation}
\label{edge10-eq31} 
\IE(\alpha,0,\ZE)\ge \area(\ZE)\cdot \left((1-\alpha) -12C(1-\alpha)^{7/6}\right),
\end{equation}
which implies (\ref{one-zero-edge-eq}). 
\end{proof}

\subsection{The $(0,0)$ corner for the L-shapes}\label{ld-Lshape-00-subsec}

As the lower bound (\ref{IElb}) can be attained, we know that
$\inf_{\ZE} \IE(\alpha,\beta,\ZE)/\gamma(\ZE)$ 
is a piecewise linear function that is nonzero on $[0,1)^2$. 
It is natural to inquire to what extent the upper bound (\ref{IEub}) on
$\sup_{\ZE}\IE(\alpha,\beta,\ZE)/\gamma(\ZE)$ can be improved.
One might ask, for example, for a piecewise linear bound which is, 
unlike (\ref{IEub}), strictly less than $1$ on $(0,1]^2$. 
We will now demonstrate by an example 
that such an improvement is impossible. 

Our example is the limit of L-shaped zero-sets consisting of $(2a-1)$ symmetrically placed 
$n\times n$ 
squares. For simplicity, 
we will assume that $a\ge 3$ is an integer. 
(A variation of the argument can be made for any real number 
$a>2$.) We will only consider the diagonal $\alpha=\beta$, which 
suffices for the purposes discussed above. 

\begin{theorem}\label{L-thm}
For the Euclidean zero set $\ZE=R_{a,1}\cup R_{1,a}$ we have, for all $\alpha\in (0,1)$,  
$$
a-2\alpha-9a\alpha^{3/2}\le \IE(\alpha,\alpha,\ZE)\le a-2\alpha.
$$
\end{theorem}

\begin{proof}[Proof of Theorem~\ref{L-thm}]
For the sequence of zero-sets $\Z_n=R_{an,n}\cup R_{n,an}$, we clearly have 
$$\squrep(\Z_n)/n\Econv R_{a,1}\cup R_{1,a}=\ZE.$$
We will show that 
\begin{equation}\label{L-eq0}
a-2\alpha-9a\alpha^{3/2}\le \liminf \frac 1{n^2}I(\alpha,\alpha,\Z_n)
\le \limsup \frac 1{n^2}I(\alpha,\alpha,\Z_n)\le a-2\alpha.
\end{equation}
This will show that $\gammaE(\ZE)=a$ and prove the desired bounds. 

To prove the upper bound, we build a spanning set $A$ by a suitable 
placement of $a$ patterns.  Of these, $a-2$ are full $n\times n$ squares, 
one consist of $n$ diagonally adjacent $1\times n$ intervals, and the final one consist of $n$ diagonally adjacent $n\times 1$ intervals. To obtain $A$, place these $a$ patterns
so that any horizontal or vertical line intersects at most one of them. It is easy to check that $A$ 
spans. Now any $B\subseteq A$ has 
$$
\pi_x(B)+\pi_y(B)\ge |B|-(a-2)n^2
$$
and so 
\begin{equation*}
\begin{aligned}
\rho(A)&\le \sup_B(1-\alpha)|B|+\alpha(a-2)n^2 
\\&=
(1-\alpha)an^2+\alpha(a-2)n^2
\\&= (a-2\alpha)n^2,
\end{aligned}
\end{equation*}
which proves the upper bound in (\ref{L-eq0}). 

To prove the lower bound, 
assume that 
$A$ is any set that spans for $\Z$. By Lemma~\ref{retract}, we may replace $A$ 
with another set, that we still denote by $A$, that spans for $\Z^{\swarrow k}$ and 
whose every point has $k$ other points in $A$ on some line of its neighborhood. We assume that
$1\ll k\ll n$. 

Fix an $\epsilon>0$, to be chosen later to be dependent on $\alpha$. 
Assume first that $|A|>(1+\epsilon)\cdot an^2$. 
Then, by Lemma~\ref{k-proj},
\begin{equation}\label{L-eq1}
\rho(A)\ge (1+\epsilon)(1-(1+1/k)\alpha)\cdot an^2.
\end{equation}

Now assume that $|A|\le(1+\epsilon)\cdot an^2$. Fix numbers $s\ge n$ and $r>0$, to be chosen 
later. If there exist $r$ horizontal lines, each with at least $s$ sites of $A$ on it, 
then $r(s-n+k)$ sites of $A$ are wasted for the $R_{n-k,an-k}$ line growth, with 
$\gamma(R_{n-k,an-k})=(n-k)(an-k)$, so 
$$
r(s-n+k)+(an-k)(n-k)\le (1+\epsilon)\cdot an^2.
$$
It follows that, if we assume
\begin{equation}\label{L-eq2}
r(s-n)-(a+1)nk\ge \epsilon \cdot an^2,
\end{equation}
then at most $r$ horizontal lines and at most $r$ vertical lines contain $s$ or more sites of $A$. 
Now, $A$ is a spanning set for both line growths with zero-sets $R_{an-k,n-k}$ and $R_{n-k,an-k}$. 
Using the slowed-down version of line growth in which a single line is occupied each time step, we
see that  
there exist some $an-k-s$ vertical lines, and some $an-k-s$ horizontal lines, each with
at least $n-k-r$ sites of $A$. Let $A_1$ and $A_2$ be the respective sets formed by occupied points on these vertical lines and horizontal lines and $A_{\rm dense}=A_1\cap A_2$. 
Then 
$$
2(an-k-s)(n-k-r)-\abs{A_{\rm dense}}\le |A_1\cup A_2|\le (1+\epsilon)\cdot an^2,
$$
and so 
\begin{equation}\label{L-eq3}
|A_{\rm dense}|\ge (a-2)n^2-2(s-n)n-2(ar+(a+1)k)n-\epsilon\cdot an^2.
\end{equation}

We now need a variant of the argument in the proof of Lemma~\ref{k-proj} for an 
upper bound on the entropy of $A$. Let $A_h'$ be the 
set of points of $A$ that are not in $A_{\rm dense}$ but lie on a horizontal line
of a point in $A_{\rm dense}$. 
Let $A_h$ be the set of points of $A$ that are not in $A_{\rm dense}\cup A_h'$ 
but lie on a horizontal line with at least $k$ other points of $A$
(and therefore with at least $k$ other points of $A_h$).  
Let $A_v'$ be the set of points that are not 
in $A_{\rm dense}\cup A_h\cup A_h'$ but lie on a vertical line of 
a point in this union. Let $A_v=A\setminus (A_{\rm dense}\cup A_h\cup A_h')$, so that any points 
of $A_v$ shares a vertical 
line with at least $k$ other points of $A_v$. Then 
\begin{equation*}
\begin{aligned}
|\pi_x(A)|&\le |\pi_x(A_{\rm dense})| +|\pi_x(A_v)|+|\pi_x(A_h)|+|\pi_x(A_h')|\\
&\le \frac1{n-r-k}|A_{\rm dense}| +\frac1k|A_v|+|A_h|+|A_h'| 
\end{aligned}
\end{equation*}
and 
\begin{equation*}
\begin{aligned}
|\pi_y(A)|&\le |\pi_y(A_{\rm dense})| +|\pi_y(A_h)|+|\pi_x(A_v)|+|\pi_x(A_v')|\\
&\le \frac1{n-r-k}|A_{\rm dense}| +\frac 1k|A_h|+|A_v|+|A_v'|
\end{aligned}
\end{equation*}
and so 
\begin{equation}\label{L-eq4}
\begin{aligned}
|\pi_x(A)|+|\pi_y(A)|&\le \frac2{n-r-k}|A_{\rm dense}|+\left(1+\frac 1k\right)(|A_h|+|A_v|)+|A_h'|+|A_v'|\\
&\le  \frac2{n-r-k}|A_{\rm dense}|+\left(1+\frac 1k\right)(|A|-|A_{\rm dense}|)
\end{aligned}
\end{equation}

By (\ref{L-eq4}), the fact that $\gamma(Z^{\swarrow k})\ge (an-k)(n-k)$ 
(which follows from Proposition~\ref{lp-bound}), and (\ref{L-eq3})
\begin{equation}\label{L-eq5}
\begin{aligned}
\rho(A)&\ge |A|-\alpha(|\pi_x(A)|+|\pi_y(A)|)
\\&\ge |A|\left(1-\left(1+\frac 1k\right)\alpha\right)
+\alpha \left(1+\frac 1k-\frac2{n-r-k}\right)|A_{\rm dense}|\\
&\ge  (an-k)(n-k)\left(1-\left(1+\frac 1k\right)\alpha\right)\\
&\quad+\alpha  \left(1+\frac 1k-\frac2{n-r-k}\right)
((a-2)n^2-2(s-n)n-2(ar+(a+1)k)n-\epsilon\cdot an^2).
\end{aligned}
\end{equation}

To guarantee (\ref{L-eq2}) for large $n$, we choose
$s-n=a\sqrt{\epsilon}n$ and 
$r=\frac 32\sqrt{\epsilon}n$.
We know that for any spanning set $A$, 
either (\ref{L-eq1}) or (\ref{L-eq5}) holds, so that 
$$
\liminf \frac 1{n^2} I(\alpha,\alpha, \Z_n)\ge \min\{a(1+\epsilon)(1-\alpha), a-2\alpha-5a\alpha\sqrt\epsilon-a\alpha\epsilon\}.
$$
To assure that the second quantity inside the $\min$ is the smaller one, we need that
$$
(a-2)\alpha \le a\epsilon+5\alpha\sqrt \epsilon , 
$$
which is assured for all $\alpha\in (0,1)$ with $\epsilon=\frac{a-2}{a}\alpha$. This 
finally gives 
\begin{equation}
\begin{aligned}
\liminf \frac 1{n^2} I(\alpha,\alpha, \Z_n)&\ge a-2\alpha-5a\sqrt{\frac a{a-2}}\alpha^{3/2}-(a-2)\alpha^2\\&
\ge  a-2\alpha-9a\alpha^{3/2}, 
\end{aligned}
\end{equation}
ending the proof of the lower bound in (\ref{L-eq0}). 
\end{proof}


\subsection{The $(1,1)$ corner}\label{ld-Lshape-11-subsec}

The upper bound (\ref{IEub}) provides a lower bound of $-2$ for the slope of 
$\sup_{\ZE}\IE(\alpha,\alpha,\ZE)/\gamma(\ZE)$
at $\alpha=1-$. Continuing with the theme from the previous section, we show that 
this bound cannot be improved either. To achieve this, we again show that the L-shapes asymptotically
attain this bound, a fact that easily follows from our next theorem. 

\begin{theorem}\label{edge11-L}
Assume the Euclidean zero set $\ZE=\RE_{a,1}\cup \RE_{1,a}$ for some $a\ge 2$. Then, 
\begin{equation*}
\begin{aligned}
2(a-1)
\left((1-\alpha)-2(1-\alpha)^2\right)
&\le \IE(\alpha,\alpha,\ZE)\le 2(a-1)
(1-\alpha),
\end{aligned}
\end{equation*}
for all $\alpha\in [0,1]$.  
\end{theorem}

We note that for $\ZE$ as in the above theorem, $\gammaE(\ZE)=a$, and therefore 
the L-shape with $a=2$ provides 
another case (apart from the line and bootstrap growths) for which the 
lower bound (\ref{IElb}) is attained on the entire diagonal $\alpha=\beta$. 

The proof of Theorem~\ref{edge11-L} proceeds in two main steps. 
In the first step, which holds for general $\ZE$, we show that in the relevant 
circumstances an arbitrary spanning set $A$ 
can be replaced by a thin spanning set of a similar size,
and use this to prove (\ref{IE-11}). The second step is a lower bound
on $\gamma_{\rm thin}(\Z)$ for the L-shaped zero-sets 
$\Z$.

  
\begin{lemma}\label{makethin}
Fix a $\delta\in (0,1)$ and a positive integer $k$. 
Let $A$ be a set that satisfies both  
$|\pi_x(A)| + |\pi_y(A)| \geq ( 1-\delta ) |A|$ and 
$A = A_{>k}.$  Then there exists a thin set $A'\subseteq A$ such that 
$$|\pi_x(A')| + |\pi_y(A')| = |\pi_x(A)| + |\pi_y(A)|$$
and 
$$|A\setminus A'|\leq \left(\delta+\frac {2}k\right)|A|.$$
\end{lemma}

\begin{proof}
Partition $A$ into three disjoint sets $A_h$, $A_v$, and $A_0$ as in the proof of 
Lemma~\ref{k-proj}.   Points in $A_h$ lie in a row with at least 
$k$ other points of $A_h$, points in $A_v$ lie in a column with at least 
$k$ other points of $A_v$, and points of $A_0$ lie in a column with 
at least $k$ other points of $A$. 

Choose any point in $A$ that shares both a row and a column with other points in $A$, 
then remove it. Repeat until no point can be removed. Let $A'$ be the so obtained 
final set. Observe that $A'$ is thin and that, as the removed points do not affect either 
projection, 
$$|\pi_x(A')| + |\pi_y(A')|= |\pi_x(A)| + |\pi_y(A)|.$$
Let $A_h' = A_h\cap A'$, $A_v' = A_v \cap A'$, and $A_0' = A_0\cap A'$.  Then, 
\begin{equation}\label{makethin-eq1}
\begin{aligned}
|\pi_x(A')| + |\pi_y(A')| 
&\leq |\pi_x(A_0'\cup A_v'\cup A_h')| + |\pi_y(A'_0\cup A_v'\cup A_h')|\\
&\leq |\pi_x(A_h')| + |\pi_y(A'_0\cup A_v')|+ |\pi_x( A_v' )| + |\pi_y(A_h')| + |\pi_x(A_0')|\\
&\leq |A'| + \frac 1k(|A_v| + |A_h| + |A|)\\
&\leq |A'| + \frac 2k |A|.
\end{aligned}
\end{equation}
Moreover,
\begin{equation}\label{makethin-eq2}
\begin{aligned}
(1-\delta)|A| \leq |\pi_x(A)| + |\pi_y(A)| = |\pi_x(A')| + |\pi_y(A')|.
\end{aligned}
\end{equation}
Combining (\ref{makethin-eq1}) and (\ref{makethin-eq2}) gives 
$\left(1-\delta - \frac 2k\right) |A|\leq |A'|$ 
and hence $|A\setminus A'| \leq \left(\delta + \frac 2k\right)|A|.$
\end{proof}


\begin{lemma}\label{makethin-span} Assume $\delta$, $k$ and $A$ satisfy conditions in
Lemma~\ref{makethin}, and 
suppose in addition that $A$ spans for some zero-set $\Z$. 
Then there exists a thin set $B$ that spans for $\Z$, such that
$$
|B|\le \left(1+ \delta + \frac 2k\right) |A|.
$$ 
\end{lemma}

\begin{proof} Let $A'\subseteq A$ be the thin set guaranteed by Lemma~\ref{makethin}.  
Let $B_r$ be a set with the same row counts as $A\setminus A'$ 
but with no two points in the same column,  
and let $B_c$ be a set with the same column counts 
as $A\setminus A'$ with no two points in the same row.  
Assuming $A\subseteq R_{a,b}$, 
let $B_s = \left((a,0) + B_r\right) \cup \left((0,b) + B_c\right).$
The set $B= A' \cup B_s$ is a thin set that spans (see the proof of
Lemma~\ref{get-two-Y}), and satisfies $|B| \leq  (1+ \delta + \frac 2k) |A|$. 
\end{proof}

\begin{lemma}\label{gamma-thin-diagonal}
For any discrete zero-set $\Z$, and $\alpha\in[0,1]$,
$I(\alpha,\alpha,\Z)\le (1-\alpha)\gamma_{\rm thin}(\Z)$.
\end{lemma}

\begin{proof}
Take a thin set $A$ that spans for $\Z$, with 
$|A|=\gamma_{\rm thin}(\Z)$. For any $B\subset A$, $|\pi_x(B)|+|\pi_y(B)|\ge |B|$, therefore
$$\rho(A)
=\sup_{B\subseteq A} |B|-\alpha(|\pi_x(B)|+|\pi_y(B)|)
\le\sup_{B\subseteq A} (1-\alpha)|B|
=(1-\alpha)|A|
=(1-\alpha)\gamma_{\rm thin}(\Z),
$$
and consequently 
$I(\alpha,\alpha,\Z)\le (1-\alpha)\gamma_{\rm thin}(\Z)$.
\end{proof}

\begin{theorem}\label{IE-diagonal} Suppose $\ZE$ is a Euclidean zero-set with finite area. 
Then 
\begin{equation}\label{IE-diagonal-ineq}
\gammaE_{\rm thin}(\ZE)\cdot \left((1-\alpha)-2(1-\alpha)^2\right)\le 
\IE(\alpha,\alpha,\ZE)\le \gammaE_{\rm thin}(\ZE)\cdot (1-\alpha).
\end{equation}
Furthermore, 
$\IE(\alpha,\alpha,\ZE)=(1-\alpha)\gamma(\ZE)$ for all $\alpha\in [0,1]$ 
if and only if $\gammaE_{\rm thin}(\ZE)=\gammaE(\ZE)$. 
\end{theorem} 
\begin{proof}
Pick discrete zero-sets $\Z_n$ so that $n^{-2}\squrep(\Z_n)\to \ZE$.
Assume that $A$ spans for $\Z_n$. Assume $1\ll k\ll n$ throughout. 
The number $\delta\in(0,1)$ will eventually be chosen to depend on $\alpha\in (0,1)$. 

By Lemma~\ref{retract}, $A'=A_{>k}$
spans for $\Z_n^{\swarrow k}$.
If $|\pi_x(A')|+|\pi_y(A')|\le (1-\delta)|A'|$, then 
\begin{equation}
\label{edge11-eq11} 
\rho(A)\ge \rho(A')\ge \delta|A'|\ge \delta\gamma(\Z_n^{\swarrow k})\ge 
\frac 12\delta\gamma_{\rm thin}(\Z_n^{\swarrow k}),
\end{equation} 
the last inequality following from Lemma~\ref{gamma-thin-gamma}. 
If $|\pi_x(A')|+|\pi_y(A')|\ge (1-\delta)|A'|$, then by Lemma~\ref{makethin-span} we can find a thin 
set $B$ that spans for $\Z_n^{\swarrow 2k}$ and has  
\begin{equation}
\label{edge11-eq2} 
|B|\le \left(1+\delta+\frac 2k\right)|A'|.
\end{equation}
Finally, we take $B'=B_{>k}$ to get a thin set that spans for 
$\Z_n^{\swarrow 3k}$. Therefore, by (\ref{edge11-eq2}),
\begin{equation}
\label{edge11-eq3} 
|A'|\ge \frac{1}{1+\delta+\frac 2k}\cdot \gamma_{\rm thin}(\Z_n^{\swarrow 3k}).
\end{equation}
By Lemma~\ref{k-proj}, 
$$
|\pi_x(A')|+|\pi_y(A')|\le \left(1+\frac 1k\right)|A'|
$$
and therefore, by (\ref{edge11-eq3}), in this case, 
\begin{equation}
\label{edge11-eq4} 
\rho(A)\ge \rho(A')\ge\frac{1-\alpha-\frac \alpha k}{1+\delta+\frac 2k}\cdot \gamma_{\rm thin}(\Z_n^{\swarrow 3k}).
\end{equation}
Now we divide (\ref{edge11-eq11}) and (\ref{edge11-eq4}) by $n^2$, send $n\to\infty$, and use Theorem~\ref{intro-d-E-theorem} to conclude that 
$$
\IE(\alpha,\alpha,\ZE)\ge \min\left\{\frac 12\delta,\frac{1-\alpha}{1+\delta}\right\}\cdot
\gammaE_{\rm thin}(\ZE).
$$
We choose $\delta$ so that the two quantities inside the minimum are equal, 
that is, 
$
\delta+\delta^2=2(1-\alpha).
$
The observation that $\delta\ge (\delta+\delta^2)-(\delta+\delta^2)^2=2(1-\alpha)-4(1-\alpha)^2$
concludes the proof of the lower bound.

The upper bound is a consequence of Lemma~\ref{gamma-thin-diagonal} 
and Theorem~\ref{intro-d-E-theorem}, and then 
the claimed equivalence follows from (\ref{IE-diagonal-ineq}) and (\ref{IElb}). 
\end{proof}

 
The key bound we need for the proof of  Theorem~\ref{edge11-L} is given by the next lemma, 
which implies that, for an L-shaped zero-set $\Z$, 
$\gamma_{\rm thin}(\Z)$ can be much larger than $\gamma(\Z)$. 

\begin{lemma}\label{thin-for-L}
Assume an L-shaped zero-set given by $\Z=R_{a+b,c}\cup R_{a, c+d}$, for some $a,b,c,d\ge 0$. 
Then $\gamma_{\rm thin}(\Z)\ge bc+ad-b-d$. 
\end{lemma}

\newcommand{\ah}{\widehat a}
\newcommand{\ch}{\widehat c}

To prove Lemma~\ref{thin-for-L}, we need 
some definitions. 
Consider two line growths, the {\it horizontal\/} one with zero-set $R_{a+b,c}$ and 
{\it vertical\/} one with zero-set $R_{a,c+d}$. Fix integers $\ah$, $\ch$ such that 
$a\le \ah\le a+b$ and $c\le \ch\le c+d$. We say that 
a set $A$ {\it H-spans\/}  if $A$ spans for $R_{a+b,c}$ after a thin set with $c$ rows of 
$\ah$ sites each is added to $A$ so that no point in it shares a row or a column with a point of $A$. 
We also say that 
a set $A$ {\it V-spans\/} if $A$ spans for $R_{a,c+d}$ after a 
thin set with $a$ columns of 
$\ch$ sites each is added to $A$, none of whose points share a row or column with $A$.
We say that a set $A$ {\it approximately spans\/} if it both H-spans and V-spans.
Clearly, any set that spans for $\Z$ as in Theorem~\ref{thin-for-L} also approximately 
spans with $\ah=a$ and $\ch=c$, so the next lemma proves Lemma~\ref{thin-for-L}.

\begin{lemma}\label{thin-for-L-approx}
Any thin set $A$ that approximately spans has $|A|\ge (c-1)(a+b-\ah)+(a-1)(c+d-\ch)$. 
\end{lemma}

\begin{proof}
We emphasize that $\ah$ and $\ch$ will stay fixed throughout the proof, 
while $a\ge 1$, $b\ge \ah-a$, $c\ge 1$, $d\ge \ch-c$ will decrease. 
We will proceed by induction on $a+b+c+d$. The claim clearly 
holds if either of the four equalities hold: $a=1$, $c=1$, $a+b=\ah$, or $c+d=\ch$, 
by the formula for the line growth $\gamma$ (Proposition~\ref{lp-gamma}). 
We will from now on assume that none of these equalities hold.

Suppose $A$ is a thin set that approximately spans for the quadruple $(a,b,c,d)$. 
The argument is divided into three cases below. We will use the slowed-down version of 
the line growth whereby a single full line (horizontal or vertical) is occupied in a
single time step, which is equivalent to the removal of that line and 
shrinking of the rectangular zero-set by eliminating one row or one column from it. 

\noindent{\it Case 1\/}. There is a horizontal line $L_h$ 
with at least $a+b$ points of $A$. Eliminate 
all points on $L_h$ from $A$ to get $A'$, and take 
$a'=a$, $b'=b$, $c'=c-1$, $d'=d$. Clearly, $A'$ is thin and V-spans for 
$R_{a',c'+d'}=R_{a,c+d}^{\downarrow 1}$. To see that $A'$  H-spans for 
$R_{a'+b',c'}=R_{a+b,c}^{\downarrow 1}$, we need to check that the addition of 
a thin set of $c-1$ rows of 
$\ah$ sites each, added to $A$, actually produces a spanning set for $R_{a+b,c}$
in this case. Indeed, after $L_h$ is made fully occupied, at most $c-1$ horizontal lines ever 
need to be spanned in the line-by-line slowed down version 
of the line growth. By the induction hypothesis, 
\begin{equation*}
\begin{aligned}
|A|\ge a+b+|A'|&\ge a+b+(c'-1)(a'+b'-\ah) + (a'-1)(c'+d'-\ch)
\\&= (c-1)(a+b-\ah)+(a-1)(c+d-\ch)+\ah-a+1
\\&> (c-1)(a+b-\ah)+(a-1)(c+d-\ch),
\end{aligned}
\end{equation*}
as $\ah\ge a$. 

\noindent{\it Case 2\/}. There is a vertical line $L_v$ 
with at least $c+d$ points of $A$. 
Using {\it Case 1\/}, this case follows by symmetry. 

\noindent{\it Case 3\/}. There exists a horizontal line $L_h$ with $a_0 \ge a$ points 
of $A$, {\it and\/} 
there exists a vertical line $L_v$ with  $c_0\ge c$ points of $A$. We assume that 
$a_0$ is the smallest such number, that is, that any horizontal line 
with strictly fewer than $a_0$ points has strictly fewer than $a$ points, and 
thus strictly fewer than $\ah$ points. We also assume the analogous condition for $c_0$.
Observe that the points on $L_h$ and $L_v$ are disjoint, 
because $A$ is thin and $a,c\ge 2$. This is the only place where we use thinness;
the necessity for disjointness is the reason that $a$ or $c$ cannot be $1$, leading to 
the factors $(c-1)$ and $(a-1)$ in the statement. 

Now we let $a'=a$, $c'=c$, 
$b'=b-1$ and $d'=d-1$. 
We will remove $a$ points from $L_h$ and $c$ points from $L_v$, redistributing the remaining
points on these two lines to make a thin set $A'$ that approximately spans. 
Once we achieve that, 
the induction hypothesis will imply that
\begin{equation*}
\begin{aligned}
|A|\ge a+c+|A'|&\ge a+c+(c'-1)(a'+b'-\ah) + (a'-1)(c'+d'-\ch)
\\&= (c-1)(a+b-\ah)+(a-1)(c+d-\ch)+2
\\&> (c-1)(a+b-\ah)+(a-1)(c+d-\ch).
\end{aligned}
\end{equation*}

It remains to demonstrate the construction and approximate spanning of $A'$.  
Clearly, if we remove the points on $L_v$ from $A$, the resulting 
set $A_0$ H-spans for $R_{a'+b',c'}=R_{a+b-1,c}=R_{a+b,c}^{\leftarrow 1}$, even 
without the redistribution 
of excess points from $L_v$. Now we address the removal and redistribution 
of points from $L_h$. Let $B_0$ be the set $A_0$ augmented 
with the set $A_0'$ of $c$ horizontal lines of $\ah$ 
points, so that $B_0$ is a thin set that spans for $R_{a+b-1,c}$. 
The set $B_0$ still contains $a_0$ points on $L_h$. 

Consider the line-by-line slowdown of line growth $R_{a+b-1,c}$, accompanied by the 
corresponding removal and shrinking of the zero-set (spanning of a horizontal line results in removal of that line and of the bottom row from the zero-set; likewise for vertical lines). If $a_0\le \ah$, then the line 
$L_h$ is {\it never\/} used, as the lines  in $A_0'$ complete the spanning before it 
{\it could\/} be used, that is, because 
lines in $A_0'$ suffice after the shrunken zero-set has $\ah$ columns. Thus the points on $L_h$ 
may be removed from $B_0$ to form $B_1$.  
Assume now $a_0>\ah$, and recall the minimality of $a_0$. 
When $L_h$ is spanned, the shrunken zero-set has at most $a_0$ columns. By minimality, 
only vertical lines, 
say, $L_1,\ldots, L_m$, $m\le a_0-\ah\le a_0-a$, are spanned before the zero-set 
shrinks to $\ah$ columns, then lines in $A'_0$ finish the job. 
Place $m$ points on the lines  $L_1,\ldots, L_m$, one point per line, 
so that they share no rows with any 
other points of $B_0$, and remove all points on line $L_h$, forming the set $B_1$. Then the lines 
$L_1,\ldots, L_m$ become occupied as before, since the extra point formerly
provided by (spanning of) the line $L_h$ has been compensated. This brings the reduced zero-set to 
$\ah$ columns and leads to spanning. Therefore, $B_1\setminus A_0'$ is a thin set that H-spans for 
$R_{a+b-1,c}$. 

The redistribution of at most $b_0-b$ points from $L_v$ 
is obtained analogously; add those redistributed points to $B_1\setminus A_0'$ to obtain the 
desired set $A'$. This justifies the induction step in this case and finishes the proof.
\end{proof}

 
\begin{proof}[Proof of Theorem~\ref{edge11-L}]
Let $\Z_n=R_{\lceil an\rceil,n}\cup R_{n,\lceil an\rceil}$. Then
Lemma~\ref{thin-for-L} implies that 
$
\gamma_{\rm thin}(\Z_n)\ge 2(a-1)n^2+\cO(n).
$
The opposite inequality follows from the fact that a thin set with $\lceil an \rceil -n$ sites on each 
of $n$ horizontal and $n$ vertical lines spans for $\Z_n$. Therefore, 
$$
\gamma_{\rm thin}(\Z_n)= 2(a-1)n^2+\cO(n).
$$
Clearly $\frac 1{n^2}\squrep(\Z_n)\Econv \ZE$, thus by (\ref{dE-gamma-thin}),
$
\gammaE_{\rm thin}(\ZE)=2(a-1).
$
Theorem~\ref{IE-diagonal} now concludes the proof.
\end{proof}

\begin{proof}[Proof of Theorem~\ref{intro-IEcorners}]
The claimed limits (\ref{IE-10}) and (\ref{IE-11})
follow from, respectively, Theorem~\ref{one-zero-edge} together with~\eqref{IEub}, and 
Theorem~\ref{IE-diagonal}. 
To prove (\ref{IE-max}), first observe that (\ref{IEub})
provides an upper bound for all $\alpha$, which has the 
slope $0$ (resp.~$-2$) when $\alpha$ is close to $0$ 
(resp.~$1$). The matching lower bound is provided by 
Theorems~\ref{L-thm} and~\ref{edge11-L} upon sending $a\to\infty$. 
\end{proof}

\section{A law of large numbers for random zero-sets}\label{sec-random-Z}

Assume that $n$ is large and that we pick 
at random a Young diagram of cardinality $n$. 
We consider the following two ways to make this random choice.
\begin{itemize}
\item Let $\Z_n$ be a Young diagram of cardinality $n$ chosen uniformly at random. We call this 
the {\it Vershik\/} sample \cite{Ver}. 
\item Build $\Z_n$ sequentially: start with $\Z_0=\emptyset$ and, 
given $\Z_k$, choose $\Z_{k+1}$ by adding a
single site to $\Z_k$ chosen at random among corners, i.e., from all sites that
make $\Z_{k+1}$ a Young diagram.   We call this 
the {\it corner growth\/} or {\it Rost\/} sample \cite{Rom}. 
\end{itemize}  

See \cite{Rom} for a review of the fascinating research into 
properties of the many possible random choices of a Young diagram.
The key property of these selections are the corresponding asymptotic shapes. Let 
$$
\ZE_{\mathrm{Vershik}}=\{(x,y)\in \bR^2: \exp\left(-{\textstyle\frac\pi{\sqrt 6}}x\right)
+\exp\left(-{\textstyle\frac\pi{\sqrt 6}}y\right)\ge 1\}
$$
and 
$$
\ZE_{\mathrm {Rost}}=\{(x,y)\in \bR^2: \sqrt{x}+\sqrt{y}\le 6^{1/4}\}.
$$
We now state the shape theorem. See \cite{Rom} and \cite{Pet} for concise proofs.

\begin{theorem} \label{shape-theorem}
For any $\epsilon>0$, the Rost sample $\Z_n$ satisfies
$$\prob{(1-\epsilon)\ZE_{\mathrm {Rost}}\subseteq n^{-1/2}\squrep(\Z_n)
\subseteq (1+\epsilon)\ZE_{\mathrm {Rost}}}\to 1,
$$
as $n\to\infty$.

For any $\epsilon>0$ and $R>0$, the Vershik sample $\Z_n$ satisfies
\begin{equation*}
\begin{aligned}
\mathbb P\left((1-\epsilon)(\ZE_{\mathrm {Vershik}}\cap [0,R]^2)\right.
&\subseteq (n^{-1/2}\squrep(\Z_n))\cap[0,R]^2\\
&\left.\subseteq (1+\epsilon)(\ZE_{\mathrm {Vershik}}\cap [0,R]^2)\right)\to 1,
\end{aligned}
\end{equation*}
as $n\to\infty$.

\end{theorem}

As a consequence, we obtain the following law of large numbers.

\begin{cor}\label{strong-law}
For either the Rost or Vershik samples 
$$
\sup_{(\alpha,\beta)\in[0,1]^2}\left|\frac 1n I(\alpha,\beta,\Z_n)-\IE(\alpha,\beta,\ZE)\right|\to 0,
$$
where $\ZE$ is the corresponding limit shape, and the convergence is in probability. 
\end{cor} 

\begin{proof} This follows from Theorems~\ref{intro-d-E-theorem} and~\ref{shape-theorem}.
\end{proof}

\section{Final remarks and open problems}\label{sec-open}
\begin{enumerate}
\item Does the completion time property given by Theorem~\ref{when-done} hold
for a more general class of growth dynamics than neighborhood growth? 
\item What is $\sup_{\Z}I(\alpha,\beta,\Z)/\gamma(\Z)$? See (\ref{I-area-ineq}) and 
(\ref{I-gamma-ineq}), and observe that we only have
trivial upper bound $1$ for this quantity 
when $\alpha$ and $\beta$ are small.
\item  Is there a simple characterization 
of Euclidean  zero-sets $\ZE$ for which $\gammaE_{\rm thin}(\ZE)=\gammaE(\ZE)$?   
We know that this holds for rectangles, isosceles right triangles, 
and L-shapes $\RE_{1,a}\cup \RE_{a,1}$, for $a\le 2$, but not for L-shapes with $a>2$
(see Section~\ref{ld-Lshape-11-subsec}). 
\item 
Does the slope $\lim_{\alpha\to 0+} \alpha^{-1}(\IE(\alpha,\alpha,\ZE)-\gammaE(\ZE))$ 
have a variational characterization?
\item
What is the slope of $\IE(\alpha,\beta,\ZE)$ as $(\alpha,\beta)$ approaches 
one of the corners at a different direction from those considered in 
Section~\ref{ld-10-subsec}--\ref{ld-Lshape-11-subsec}? What can be said about  
other boundary points? 
\item Fix $(\alpha,\beta)\ne (0,0)$ and a zero-set $\Z$. What is  
the minimal $a$ such that  
there exists an $A\subseteq R_{a,a}$ with $\rho(\alpha,\beta,A) = I(\alpha,\beta,\Z)$?
\item Can an explicit analytical
formula for $I(\alpha,\beta,T_\theta)$ be given for all $(\alpha,\beta)\in [0,1]^2$?
\item Can existence of large deviation rates be proved
for bootstrap percolation \cite{GHPS} or for line growth \cite{BBLN} in three dimensions?  
A result in this direction is proved in \cite{BBLN}, where it is also 
pointed out that it is not at all clear that the 
completion time result holds in higher dimensions.  

\item What is the algorithmic complexity for computation of $\gamma(\Z)$, when $\Z$ is given as 
input?  

\end{enumerate}

\section*{Acknowledgments} 

Janko Gravner was partially supported by the NSF grant DMS-1513340, 
Simons Foundation Award \#281309,
and the Republic of Slovenia's Ministry of Science
program P1-285. David Sivakoff 
was partially supported by NSF CDS\&E-MSS Award \#1418265, and a portion of this work was conducted while a visitor of the Mathematical Biosciences Institute.

\end{document}